\documentclass[a4paper,11pt]{article}
\usepackage{amsmath,amsfonts,amssymb,hyperref}
\usepackage{amsthm,mathtools,dsfont,enumitem,color}
\usepackage[margin = 3cm]{geometry}
\usepackage[utf8]{inputenc}	
\usepackage{xparse}

\newtheorem{thm}{Theorem}[]
\newtheorem{lem}{Lemma}[section]
\newtheorem{prop}{Proposition}[section]

\newtheorem{rem}{Remark}[section]
\newtheorem{ex}{Example}[section]
\newtheorem{defn}{Definition}[section]

\numberwithin{equation}{section}

\makeatletter
\DeclarePairedDelimiter{\abs}{\lvert}{\rvert}
\DeclarePairedDelimiter{\norm}{\lVert}{\rVert}
\NewDocumentCommand{\norml}{ s O{} m m }{%
	\IfBooleanTF{#1}{\norm*{#4}}{\norm[#2]{#4}}_{L^{#3}}%
}
\NewDocumentCommand{\normL}{ s O{} m m }{%
	\IfBooleanTF{#1}{\norm*{#4}}{\norm[#2]{#4}}_{\LP^{#3}}%
}
\NewDocumentCommand{\normw}{ s O{} m m m }{%
	\IfBooleanTF{#1}{\norm*{#4}}{\norm[#2]{#5}}_{W^{#3,#4}}%
}
\NewDocumentCommand{\normW}{ s O{} m m m }{%
	\IfBooleanTF{#1}{\norm*{#4}}{\norm[#2]{#5}}_{\W^{#3,#4}}%
}
\NewDocumentCommand{\normh}{ s O{} m m }{
	\IfBooleanTF{#1}{\norm*{#4}}{\norm[#2]{#4}}_{H^{#3}}
}
\NewDocumentCommand{\normH}{ s O{} m m }{%
	\IfBooleanTF{#1}{\norm*{#4}}{\norm[#2]{#4}}_{\H^{#3}}%
}
\makeatother

\def\Td{\mathcal{T}_\delta}
\def\bPhi{\boldsymbol{\Phi}}
\def\R{\mathbb R}

\def\N{\mathbb N}
\def\n{\mathbf n}
\def\a{\mathbf a}

\def\f{\mathbf{f}}
\def\DD{\mathcal D}
\def\z{\mathbf{z}}
\def\k{\mathbf{k}}

\def\LP{\mathbf L}
\def\H{\mathbf H}
\def\bh{\mathbf h}

\def\u{\mathbf{u}}
\def\v{\mathbf{v}}
\def\w{\mathbf{w}}
\def\y{\mathbf{y}}
\def\T{\mathbf{T}}
\def\I{\mathbf{I}}
\def\D{{\mathbf{D}}}

\def\W{{\mathbf W}}
\def\F{{\mathbf{F}}}
\def\q{{\mathbf{q}}}

\def\KK{\mathcal{K}}
\def\bpsi{\boldsymbol{\psi}}

\def\OmegaT{{\Omega_T}}

\def\intO{\int_\Omega}
\def\intS{\int_\Sigma}
\def\intT{\int_{0}^{T}}
\def\intOT{\int_\OmegaT}

\def\Do{\mathrm{do}}
\def\eps{\varepsilon}

\def\dx{\;\mathrm dx}

\def\dt{\;\mathrm dt}

\def\dH{\; \mathrm d \mathcal{H}}

\def\del{\partial}

\def\delt{\partial_{t}}

\def\deln{\partial_\n}

\def\d{{\mathrm{d}}}

\def\grad{\nabla}
\def\lap{\Delta}
\def\div{\mathrm{div}}
\def\0{\mathbf 0}

\newcommand{\inn}[2]{\left \langle #1, #2 \right \rangle}

\begin{document}

\title{Weak and stationary solutions to a Cahn--Hilliard--Brinkman model with singular potentials and source terms}
\author{Matthias Ebenbeck \footnotemark[1] \and Kei Fong Lam \footnotemark[2]}
\date{}

\renewcommand{\thefootnote}{\fnsymbol{footnote}}
\footnotetext[1]{Fakult\"at f\"ur Mathematik, Universit\"at Regensburg, 93040 Regensburg, Germany ({\tt matthias.ebenbeck@mathematik.uni-regensburg.de}).}
\footnotetext[2]{Department of Mathematics, The Chinese University of Hong Kong, Shatin, N.T., Hong Kong
({\tt kflam@math.cuhk.edu.hk}).}

\maketitle

\begin{abstract}
We study a phase field model proposed recently in the context of tumour growth.  The model couples a Cahn--Hilliard--Brinkman (CHB) system with a elliptic reaction-diffusion equation for a nutrient.  The fluid velocity, governed by the Brinkman law, is not solenodial, as its divergence is a function of the nutrient and the phase field variable, i.e., solution-dependent, and frictionless boundary conditions are prescribed for the velocity to avoid imposing unrealistic constraints on the divergence relation.  In this paper we give a first result on the  existence of weak and stationary solutions to the CHB model with singular potentials, specifically the double obstacle potential and the logarithmic potential, which ensures the phase field variable stay in the physically relevant interval.  New difficulties arise from the interplay between the singular potentials and the solution-dependent source terms, but can be overcome with several key estimates for the approximations of the singular potentials, which may be of independent interests.  As a consequence, included in our analysis is a weak existence result for a Darcy variant, and our work serve to generalise recent results on weak and stationary solutions to the Cahn--Hilliard inpainting model with singular potentials.
\end{abstract}



\section{Introduction}
Phase field models \cite{CL,Oden} have recently emerged as a new mathematical tool for tumour growth, which offer new advantages over classical models \cite{Byrne,Fried} based on a free boundary description, such as the ability to capture metastasis and other morphological instabilities like fingering in a natural way.  In conjunction with clinical data, statistical methodologies and image analysis, they have begun to impact the development of personalised cancer treatments \cite{Agosti1,Agosti2,HPVO,KLLU,Lima,Lorenzo}.   

To support this continuing effort, in this paper we expand on the recent analysis performed in \cite{EG_jde,EG2} for a class of phase field tumour models based on the Cahn--Hilliard--Brinkman (CHB) system.  For a bounded domain $\Omega \subset \R^d$ for $d \in \{2,3\}$ with boundary $\Sigma := \del \Omega$ and outer unit normal $\n$, and an arbitrary but fixed terminal time $T$, we consider for $\Omega_T := \Omega \times (0,T)$ the following system of equations  
\begin{subequations}\label{MEQ}
\begin{alignat}{3}
\label{MEQ_1} \div(\v)&=\Gamma_{\v}(\varphi,\sigma) := b_\v(\varphi) \sigma + f_\v(\varphi) && \quad \text{ in }   \Omega_T,\\
\label{MEQ_2}-\div(\T(\v,p)) +\nu\v&= (\mu + \chi\sigma)\grad\varphi  && \quad \text{ in } \Omega_T,\\
\label{MEQ_3}\delt \varphi + \div(\varphi\v) - \lap \mu &= \Gamma_{\varphi}(\varphi,\sigma) := b_\varphi(\varphi) \sigma + f_\varphi(\varphi)&& \quad \text{ in } \Omega_T,\\
\label{MEQ_4}\mu&= \psi'(\varphi)-\lap \varphi -\chi\sigma && \quad \text{ in } \Omega_T,\\
\label{MEQ_5}0 &= \lap \sigma - h(\varphi)\sigma&& \quad \text{ in } \Omega_T,
\end{alignat}
\end{subequations}
where the viscous stress tensor $\T$ and the symmetrised velocity gradient $\D\v$ are given by
\begin{align*}
\T(\v,p) := 2\eta(\varphi) \D\v+\lambda(\varphi) \div(\v) \I - p \I, \quad \D\v :=  \frac{1}{2}(\grad\v+\grad\v^\top).
\end{align*}
The model \eqref{MEQ} is a description of the evolution of a two-phase cell mixture, containing tumour cells and healthy host cells, surrounded by a chemical species acting as nutrients only for the tumour cells, and is transported by a fluid velocity field.  The variable $\sigma$ denotes the concentration of the nutrient, which we assume evolves quasistatically, while $\varphi$ denotes the difference in the volume fractions of the cells, with the region $\{\varphi = 1 \}$ representing the tumour cells and $\{\varphi = -1\}$ representing the host cells.  The fluid velocity $\v$ is taken as the volume averaged velocity, with pressure $p$, and $\mu$ denotes the chemical potential associated to $\varphi$.

Equations \eqref{MEQ_3}-\eqref{MEQ_4} comprise a convective Cahn--Hilliard system for $(\varphi, \mu)$ with source term $\Gamma_{\varphi}(\varphi, \sigma)$, that couples the nutrient equation \eqref{MEQ_5} also through the consumption term $h(\varphi) \sigma$.  The constant parameter $\chi$ appearing in \eqref{MEQ_4} captures the chemotaxis effect, see \cite{GLSS}. On the other hand, equations \eqref{MEQ_1}-\eqref{MEQ_2} form a Brinkman system for $(\v, p)$ with the term $(\mu + \chi \sigma) \nabla \varphi$ modelling capillary forces, and in the definition of the stress tensor, the functions $\eta$ and $\lambda$ represent the shear and bulk viscosities, respectively.  An interesting contrast with previous phase field models in two-phase flows that employ a volume averaged velocity, such as \cite{AGG,Boyer,Ding}, is that the velocity in \eqref{MEQ} is not solenoidal.  This can be attributed to the fact that in the case of unmatched densities, the gain $\sigma b_\varphi(\varphi)$ and loss $f_\varphi(\varphi)$ of cellular volume leads to sources $\sigma b_\v(\varphi)$ and sinks $f_\v(\varphi)$ in the mass balance, see for example \cite{GLSS,LW} for more details in the model derivation.

As a consequence of \eqref{MEQ_1}, we find the relation
\begin{align*}
\intO \Gamma_{\v}(\varphi, \sigma) \dx = \intO \div(\v) \dx = \intS \v \cdot \n \dH
\end{align*}
and the typical no-penetration boundary condition $\v \cdot \n = 0$ on $\Sigma \times (0,T)$ will lead to a mean-zero compatibility condition for $\Gamma_{\v}(\varphi, \sigma)$, which may not be satisfied in general.  For models with Darcy's law instead of \eqref{MEQ_2}, there are some works in the literature \cite{FLRS,GLIndam} that prescribe alternate boundary conditions, such as zero transport flux $\nabla \mu \cdot \n = \varphi \v \cdot \n$ and zero pressure $p = 0$ on $\Sigma_T := \Sigma \times (0,T)$ to circumvent the compatibility condition on $\Gamma_{\v}$.  For the Brinkman model, the compatibility issue can be avoided by prescrbing the frictionless boundary condition $\T \n = \0$ on $\Sigma_T$, as considered in \cite{EG_jde,EG2,EK,EK2}.

Hence, we furnish \eqref{MEQ} with the following initial and boundary conditions
\begin{subequations}\label{BIC}
\begin{alignat}{3}
\label{BC_1}\deln\mu=\deln\varphi &= 0 &&\quad \text{ on }\Sigma_T ,\\
\label{BC_2}\deln\sigma &= K(1-\sigma)&&\quad\text{ on }\Sigma_T
,\\
\label{BC_3}\T(\v,p)\n  &= \mathbf{0}&&\quad\text{ on }\Sigma_T,\\
\label{IC}\varphi(0) &= \varphi_0 &&\quad\text{ in }\Omega,
\end{alignat} 
\end{subequations}
where $\deln f = \nabla f \cdot \n$ is the normal derivative on $\Sigma$, $\varphi_0$ is a prescribed initial condition and $K$ is a positive permeability constant.

The resulting system \eqref{MEQ}-\eqref{BIC} has been studied previously by the first author in a series of works \cite{EG_jde,EG2,EK,EK2} concerning well-posedness, asymptotic limits and optimal control, in the setting where the double well potential $\psi$, whose first derivative appears in \eqref{MEQ_4}, is continuously differentiable.  The common example is the quartic potential $\psi(s) = (s^2-1)^2$.  The purpose of the present paper is to provide a first result for the weak existence to the CHB model when $\psi$ is a singular potential, where either the classical derivative of $\psi$ does not exist, or the derivative $\psi'$ blows up outside a certain interval.  For the former, the prime example is the double obstacle potential \cite{BE,BE_t}
\begin{align*}
\psi_{\Do}(r)=\frac{1}{2}(1-r^2)+ \mathbb{I}_{[-1,1]}(r) = \begin{cases} \frac{1}{2}(1-r^2) & \text{ for } r \in [-1,1], \\
+\infty & \text{ otherwise},
\end{cases}
\end{align*}
and for the latter the logarithmic potential
\begin{align*}
\psi_{\log}(r) = \frac{\theta}{2}\big((1+r)\log(1+r)+ (1-r)\log(1-r)\big)+\frac{\theta_c}{2}(1-r^2)\quad \text{ for } r\in(-1,1)
\end{align*}
with positive constants $0 < \theta < \theta_c$ is the standard example.

Let us briefly motivate the need to consider singular potentials such as $\psi_{\Do}$ and $\psi_{\log}$.  The variable $\varphi$ has a physical interpretation as the difference between the volume fractions of the tumour cells and the host cells.  As volume fractions are non-negative and bounded above by 1, $\varphi$ should lie in the physical range $[-1,1]$.  This natural boundedness of $\varphi$ becomes particularly important for physical models in fluid flow and biological sciences, where the mass density of the mixture expressed as an affine linear function of $\varphi$ may become negative if $\varphi$ strays out of $[-1,1]$.  A counterexample in \cite{CMZ} shows that for polynomial $\psi$, the phase field variable $\varphi$ can take values outside $[-1,1]$, and hence a remedy to enforce the natural bounds for $\varphi$ is to introduce nonsmooth potentials into the model.

For the original Cahn--Hilliard equation and its variants in solenoidal two-phase flow with the boundary condition \eqref{BC_1}, there are numerous contributions in the literature involving singular potentials, of which we cite \cite{ADG,AW,BE,Conti,GalGW,GGW,GMT} and refer to the references cited therein.  The most challenging aspect with singular potentials in the global existence of weak solutions is to show that the chemical potential $\mu$ is controlled in the Bochner space $L^2(0,T;H^1(\Omega))$.  This is equivalent to controlling the mean value $\mu_\Omega := \frac{1}{|\Omega|} \intO \mu \dx$ in $L^2(0,T)$, since the gradient $\nabla \mu$ is controlled from basic energy identities.  For polynomial $\psi$ this can be done via the relation $\intO \mu \dx = \intO \psi'(u) \dx$ with suitable growth conditions on $\psi'$, but this argument fails for the singular case.  The techniques used in the aforementioned references depend on first establishing the assertion $\varphi_\Omega(t) \in (-1,1)$ for all $t > 0$, which holds automatically for suitably chosen initial data thanks to the property of mass conservation $\varphi_\Omega(t) = \varphi_\Omega(0)$ for all $t > 0$.

When the Cahn--Hilliard component is affixed with a source term, like \eqref{MEQ_3}, then in general the mass conservation property is lost.  Therefore, simply using previous techniques would only give a local-in-time weak existence result holding as long as $\varphi_\Omega(t) \in (-1,1)$, see for instance the work of \cite{CFM} on the so-called Cahn--Hilliard inpainting model \cite{BEG} with logarithmic potential obtained from \eqref{MEQ_3}-\eqref{MEQ_4} by setting $\psi = \psi_{\log}$, $\v = \0$, $\sigma = 0$ and $f_\varphi(\varphi) = \lambda (I - \varphi)$ for given nonnegative $\lambda \in L^{\infty}(\Omega)$ and $|I| \leq 1$ a.e.~in $\Omega$.  Recently, the second author has established a global existence result to the inpainting model with the double obstacle potential $\psi_{\Do}$ in \cite{GLS}.  The key ingredients involved are a careful analysis of the interplay between the source terms and approximations of the singular part $\mathbb{I}_{[-1,1]}$ of $\psi_{\Do}$, as well as showing the source terms lead to the conclusion that $\varphi_\Omega(t) \in (-1,1)$ for all $t > 0$.  Then, previous methods for controlling $\mu_\Omega$ can be applied to achieve global weak existence.

To generalise the methodology of \cite{GLS} for the CHB model \eqref{MEQ} with $\psi_{\Do}$, we found that it suffices to prescribe suitable conditions on $b_\v$, $b_\varphi$, $f_\v$ and $f_\varphi$ at $\varphi = \pm 1$.  Namely, we ask that $b_\v$ and $b_\varphi$ are nonnegative and vanish at $\pm 1$, while $f_\varphi - f_\v$ is negative (resp.~positive) at $1$ (resp.~$-1$).  In Remark \ref{ex:source} below we give a biologically relevant example of $b_\v$, $b_\varphi$, $f_\v$ and $f_\varphi$ satisfying the above conditions.  With the help of a refined version of a key estimate (see Proposition \ref{prop:CFM}), the same methodology can be used to show global weak existence for the CHB model \eqref{MEQ} with $\psi_{\log}$, thereby improving also the results of \cite{CFM,Mir} concerning the inpainting model with the logarithmic potential $\psi_{\log}$. 

By formally sending the bulk and shear viscosities in the stress tensor $\T$ to zero, the CHB model \eqref{MEQ} reduces to a Cahn--Hilliard--Darcy (CHD) tumour model with 
\begin{align*}
\v = -\frac{1}{\nu} \big (\nabla p - (\mu + \chi \sigma) \nabla \varphi \big )
\end{align*}
replacing \eqref{MEQ_2}.  Without source terms and the nutrient, the CHD model is also called the Hele-Shaw--Cahn--Hilliard system, where recent progress on strong well-posedness with the logarithmic potential can be found in \cite{G,GGW}.  Meanwhile, for the CHD tumour model with polynomial $\psi$, it appears that the regularity of the weak solutions in \cite{GLDarcy,GLIndam} is insufficient to replicate the usual procedure of approximating the singular potential with a sequence of polynomial potentials, deriving uniform estimates and passing to the limit.  To the authors' best knowledge the global weak existence to the CHD tumour model remains an open problem, and inspired by the relationship between the Brinkman and Darcy laws, we employ an idea of \cite{EG2} to deduce the global weak existence to the CHD tumour model with singular potentials by scaling the shear and bulk viscosities in an appropriate way.

The second contribution of the present paper is to provide a first result for stationary solutions to the CHB model with singular potentials, where the time derivative in \eqref{MEQ_3} vanishes but the convection term remains, i.e., we allow for the possibility of a non-zero velocity field.  This is interesting as the fluid velocity can be used to inhibit the growth of the tumour cells, and in an optimal control framework, these stationary solutions are ideal candidates in a tracking-type objective functional, e.g.~see \cite{CRW}.  Unfortunately, the particular form of the source terms $\Gamma_{\v}(\varphi, \sigma)$ and $\Gamma_{\varphi}(\varphi, \sigma)$ indicates that the CHB model does not admit an obvious Lyapunov structure, and so it is unlikely that the a priori estimates used to prove weak existence are uniform-in-time.  This is in contrast to the phase field tumour model of \cite{HVO} studied in \cite{CRW,CGH,FGR,Sig1,Sig2} where tools such as global attractors and the {\L}ojasiewicz--Simon inequality can be used to quantify the long-time behaviour of global solutions.  We also mention the recent work \cite{MRS} establishing a global attractor for a simplified model similar to the subsystem \eqref{MEQ_3}-\eqref{MEQ_5} with $\v = \0$ and $\chi = 0$, where despite the lack of an obvious Lyapunov structure, the authors can derive dissipative estimates under some constraints on the model parameters.

However, we opt to prove directly the existence of stationary solutions using the methodology of \cite{GLS}, as oppose to investigate the long-time behaviour of time-dependent solutions.  Thanks to the well-posedness of the Brinkman subsystem \eqref{MEQ_1}, \eqref{MEQ_2}, \eqref{BC_3} and the nutrient subsystem \eqref{MEQ_5}, \eqref{BC_2}, we can express $\sigma = \sigma_\varphi$, $\v = \v_\varphi$ and $p = p_\varphi$ as functions of $\varphi$, since $\mu$ can also be interpreted as a function of $\varphi$ by \eqref{MEQ_4}.  Then, the stationary CHB system is formally equivalent to the fourth order elliptic problem
\begin{align*}
\begin{cases}
\div (\varphi \v_\varphi) + \lap^2 \varphi - \lap \psi'(\varphi) + \chi \lap \sigma_\varphi = \Gamma_\varphi(\varphi, \sigma_\varphi) & \text{ in } \Omega, \\
\deln \varphi = \deln ( \lap \varphi + \chi \sigma_\varphi) = 0 & \text{ on } \Sigma.
\end{cases}
\end{align*}
The novelty of our work lies in controlling the convection term $\div (\varphi \v_\varphi)$ with the help of the unique solvability of the Brinkman system, and the existence result (Theorem \ref{thm:stat}) is a non-trivial extension of \cite{GLS}.  Furthermore, this also shows the existence of stationary solutions to the inpainting model with logarithmic potential, thereby completing the analysis of weak and stationary solutions to the Cahn--Hilliard inpainting model with both singular potentials.

The remainder of this paper is organised as follows.  In section \ref{sec:main} we cover several useful preliminary results and state our main results on the existence of global weak solutions and stationary solutions to the CHB model \eqref{MEQ}-\eqref{BIC} with singular potentials, as well as the global weak existence to the CHD tumour model.  The proofs of these results are then contained in sections \ref{sec:time}, \ref{sec:stat} and \ref{sec:Darcy}, respectively.  In section \ref{sec:appendix} we give a proof of a well-posedness result for the Brinkman system that plays a significant role in our work.

\section{Main results}\label{sec:main}
\subsection{Notation and preliminaries}
For a real Banach space $X$ we denote by $\norm{\cdot}_X$ its norm, by $X^*$ its dual space and by $\inn{\cdot}{\cdot}_X$ the duality pairing between $X^*$ and $X$.  If $X$ is an inner product space the associated inner product is denoted by $(\cdot, \cdot)_X$.  For two matrices $\mathbf{A} = (a_{jk})_{1 \leq j,k \leq d}, \mathbf{B} = (b_{jk})_{1 \leq j,k \leq d}\in \R^{d \times d}$, the scalar product $\mathbf{A} : \mathbf{B}$ is defined as the sum $ \sum_{j,k=1}^{d}a_{jk}b_{jk}$.

For $1 \leq p \leq \infty$, $r \in (1,\infty)$, $\beta \in (0,1)$ and $k \in \N$, the standard Lebesgue and Sobolev spaces on $\Omega$ and on $\Sigma$ are denoted by $L^p := L^p(\Omega)$, $L^p(\Sigma)$, $W^{k,p} := W^{k,p}(\Omega)$, and $W^{\beta,r}(\Sigma)$ with norms $\norml{p}{\cdot}$, $\norm{\cdot}_{L^p(\Sigma)}$, $\normw{k}{p}{\cdot}$ and $\norm{\cdot}_{W^{\beta,r}(\Sigma)}$, respectively.  In the case $p=2$ we use the notation $H^k := W^{k,2}$ with the norm $\normh{k}{\cdot}$.  The space $H_0^1$ is the completion of $C_0^{\infty}(\Omega)$ with respect to the $H^1$-norm, while the space $H_{n}^2$ is defined as $ \{w\in H^2 \colon \deln w = 0 \text{ on } \Sigma\}$.  For the Bochner spaces, we use the notation $L^p(X):= L^p(0,T;X)$ for a Banach space $X$ with $p\in [1,\infty]$, and
\begin{align*}
\norm{\cdot}_{A \cap B} = \norm{\cdot}_{A} + \norm{\cdot}_{B}
\end{align*}
for two or more Bochner spaces $A$ and $B$.  For the dual space $X^*$ of a Banach space $X$, we introduce the (generalised) mean value by 
\begin{equation*}
v_{\Omega}\coloneqq \frac{1}{|\Omega|}\intO v\dx \quad\text{for } v\in L^1, \quad v_{\Omega}^*\coloneqq \frac{1}{|\Omega|}\langle v{,}1\rangle_X\quad\text{for } v\in X^*,
\end{equation*}
and also the subspace of $L^2$-functions with zero mean value:
\begin{align*}
&L_0^2 \coloneqq  \{w\in L^2\colon w_{\Omega}=0\}.
\end{align*}
The corresponding function spaces for vector-valued or tensor-valued functions are denoted in boldface, i.e., $\LP^p$, $\W^{k,p}$, $\H^k$, $\LP^p(\Sigma)$ and $\W^{\beta,r}(\Sigma)$.  
Furthermore, we introduce the space
\begin{align*}
\LP^p_{\div} \coloneqq \{ \f \in \LP^p \, : \, \div (\f) \in L^p \}
\end{align*}
equipped with the norm
\begin{align*}
\norm{\f}_{\LP^p_{\div}} \coloneqq \Big ( \norm{\f}_{\LP^p}^p + \norm{\div(\f)}_{L^p}^p \Big)^{1/p},
\end{align*}
where $\div$ is the weak divergence.  We now state three auxiliary lemmas that are used for the study of our model.  The first lemma concerns the solvability of the divergence problem, which will be useful for the mathematical treatment of equation \eqref{MEQ_1}.
\begin{lem}[{\cite[Sec.~III.3]{Galdi}}]\label{LEM_DIVEQU}
Let $\Omega \subset \R^d$, $d\geq 2$, be a bounded domain with Lipschitz-boundary and let $1<q<\infty$. Then, for every $f \in L^q$ and $\a \in \W^{1-1/q,q}(\Sigma)$ satisfying
\begin{align}\label{DIV_COMP_COND}
\intO f\dx  = \intS \a\cdot\n \dH,
\end{align}
there exists at least one solution $\u \in \W^{1,q}$ to the problem
\begin{align}\label{DIV_EQU}
\begin{cases}
\div(\u) = f & \text{ in }\Omega, \\
\u  = \a & \text{ on }\Sigma,
\end{cases}
\end{align}
satisfying for some positive constant $C = C(\Omega, q)$ the estimate
\begin{align}\label{DIV_EST}
\normW{1}{q}{\u}\leq C\left(\norml{q}{f}+\norm{\a}_{\W^{1-1/q,q}(\Sigma)}\right).
\end{align}
\end{lem}
Since we invoke Lemma \ref{LEM_DIVEQU} frequently with $\a = \frac{1}{\abs{\Sigma}} \Big ( \int_\Omega f \dx \Big ) \n$, we introduce 
\begin{align*}
\u = \DD(f) \quad \Leftrightarrow \quad \begin{cases}
\div(\u) = f & \quad \text{ in } \Omega, \\
\u = \tfrac{1}{\abs{\Sigma}} \Big ( \int_\Omega f \dx \Big ) \n & \quad \text{ on } \Sigma,
\end{cases}
\end{align*}
where by the smoothness of the normal $\n$, see \eqref{ass:dom} below, \eqref{DIV_EST} yields
\begin{align*}
\normW{1}{q}{\DD(f)} \leq C \norml{q}{f}.
\end{align*}

The second lemma concerns the existence of weak solutions to the model \eqref{MEQ} which forms the basis of our approximation procedure.

\begin{lem}[{\cite[Thm.~2.5]{EG2}}]\label{THM_WSOL_1}
Suppose the following assumptions are satisfied:
\begin{enumerate}[label=$(\mathrm{A \arabic*})$, ref = $\mathrm{A \arabic*}$]
\item \label{ass:dom} For $d \in \{2,3\}$, $\Omega \subset \R^d$ is a bounded domain with $C^3$-boundary.
\item \label{ass:const} The positive constants $\nu$, $K$ and the nonnegative constant $\chi$ are fixed.  
\item \label{ass:visc} The viscosities $\eta$ and $\lambda$ belong to $C^2(\R) \cap W^{1,\infty}(\R)$ and satisfy
\begin{align}\label{ASS_VISC}
\eta_0\leq \eta(t)\leq \eta_1,\quad 0\leq \lambda(t)\leq \lambda_0\quad\forall t\in\R,
\end{align}
for positive constants $\eta_0,\eta_1$ and a nonnegative constant $\lambda_0$.  The function $h \in C^0(\R) \cap L^{\infty}(\R)$ is nonnegative.
\item \label{ass:source} The source terms $\Gamma_{\v}$ and $\Gamma_\varphi$ are of the form
\begin{align}\label{source:form}
\Gamma_{\v}(\varphi, \sigma) = b_{\v}(\varphi) \sigma + f_{\v}(\varphi), \quad \Gamma_\varphi(\varphi, \sigma) = b_{\varphi}(\varphi)\sigma + f_\varphi(\varphi),
\end{align}
where $b_\v, f_\v \in C^1(\R) \cap W^{1,\infty}(\R)$ and $b_\varphi, f_\varphi \in C^0(\R) \cap L^{\infty}(\R)$.  
\item \label{ass:ini} The initial condition $\varphi_0$ belongs to $H^1$.
\item \label{ass:psi} The function $\psi \in C^2(\R)$ is nonnegative and satisfies
\begin{align}\label{psi1psi2}
\psi(s) \geq R_0|s|^2-R_1, \quad |\psi'(s)| \leq R_2\left(1+|s|\right), \quad |\psi''(s)|\leq R_3\quad\forall s \in \R, 
\end{align}
where $R_0$, $R_1$, $R_2$, $R_3$ are positive constants.
\end{enumerate}
Then, there exists a quintuple $(\varphi,\mu,\sigma,\v,p)$ satisfying
\begin{align*}
&\varphi \in H^1(0,T;(H^1)^*)\cap L^{\infty}(0,T;H^1)\cap L^4(0,T;H_n^2)\cap L^2(0,T;H^3) ,\quad \sigma\in L^{\infty}(0,T;H^2),\\
&\mu\in L^4(0,T;L^2)\cap L^2(0,T;H^1),\quad\v\in L^{\frac{8}{3}}(0,T;\H^1),\quad p\in L^{2}(0,T;L^2),
\end{align*}
and is a weak solution to \eqref{MEQ}-\eqref{BIC} in the sense that \eqref{MEQ_1} holds a.e.~in $\Omega_T$ and
\begin{subequations}
\begin{alignat}{3}
\label{WFORM_1a}0 &= \intO 2 \eta(\varphi) \D \v : \D \bPhi + (\lambda(\varphi) \div (\v) - p ) \div (\bPhi )+ \nu \v \cdot \bPhi  - (\mu+\chi\sigma) \grad\varphi \cdot \bPhi\dx, \\
\label{WFORM_1b} 0 &= \inn{\delt \varphi}{\zeta}_{H^1}  + \intO \grad\mu\cdot\grad \zeta+ (\grad\varphi\cdot\v+ \varphi\Gamma_{\v}(\varphi,\sigma) - \Gamma_{\varphi}(\varphi,\sigma)) \zeta \dx,  \\
\label{WFORM_1c} 0 &= \intO (\mu+ \chi\sigma) \zeta - \psi'(\varphi)\zeta - \grad\varphi\cdot\grad\zeta \dx,\\
\label{WFORM_1d} 0  &=  \intO \grad\sigma\cdot\grad\zeta + h(\varphi)\sigma\zeta \dx +\intS K(\sigma-1)\zeta \dH,
\end{alignat}
\end{subequations} 
for a.e.~$t\in(0,T)$ and for all $\bPhi\in \H^1$ and $\zeta \in H^1$.  Furthermore, there exists a positive constant $C$ not depending on $(\varphi, \mu, \sigma, \v, p)$ such that 
\begin{equation}\label{THM_WSOL_1_EST_1}
\begin{aligned}
&\norm{\varphi}_{H^1((H^1)^*)\cap L^{\infty}(H^1)\cap L^4(H^2)\cap L^2(H^3) } + \norm{\sigma}_{L^{\infty}(H^2)} + \norm{\mu}_{L^4(L^2)\cap L^2(H^1)}\\
&\quad + \norm{\div(\varphi\v)}_{L^2(L^2)}  + \norm{\v}_{L^{8/3}(\H^1)}+ \norm{p}_{L^2(L^2)}\leq C.
\end{aligned}
\end{equation}
\end{lem}
\begin{rem}
We point out that the assumptions on the potential in \cite{EG2} are more restrictive in the case of potentials with quadratic growth. However, it can be checked easily that the proof of \cite[Thm.~2.5]{EG2} follows through when using potentials that satisfy \eqref{psi1psi2}.
\end{rem}
The third lemma concerns the solvability of the Brinkman system, which we will use to study stationary solutions.  A proof is contained in the appendix.
\begin{lem}\label{lem:Brink}
Let $\Omega \subset \R^d$, $d = 2,3$, be a bounded domain with $C^3$-boundary $\Sigma$ and outer unit normal $\n$.  Let $c \in W^{1,r}$ with $r > d$ be given and fix exponent $q \leq r$ such that $q > 1$ for $d = 2$ and $q \geq \frac{6}{5}$ for $d = 3$, and suppose the functions $\eta(\cdot)$ and $\lambda(\cdot)$ satisfy \eqref{ass:visc}.  Then, for any $\f \in \LP^q$, $g \in W^{1,q}$, $\bh \in \W^{1-1/q,q}(\Sigma)$, there exists a unique solution $(\v, p) \in \W^{2,q} \times W^{1,q}$ to
\begin{subequations}\label{BM_SUBSY}
\begin{alignat}{2}
- \div ( 2 \eta(c) \D \v + \lambda(c) \div (\v) \I) + \nu \v + \nabla p = \f & \qquad\text{a.e.~in } \Omega, \\
\div(\v) = g & \qquad\text{a.e.~in } \Omega, \\
(2 \eta(c) \D \v + \lambda(c) \div(\v) \I - p \I)\n = \bh  &\qquad\text{a.e.~on } \Sigma,
\end{alignat}
\end{subequations}
satisfying the following estimate
\begin{align}\label{Brink:strong:est}
\normW{2}{q}{\v} + \normw{1}{q}{p} \leq C \Big (\normL{q}{\f} + \normw{1}{q}{g} + \norm{\bh}_{\W^{1-\frac{1}{q}, q}(\Sigma)} \Big ),
\end{align}
with a constant $C$ depending only on $\eta_0$, $\eta_1$, $\lambda_0$, $\nu$, $q$, $\normw{1}{r}{c}$ and $\Omega$.
\end{lem}

\subsection{The time-dependent problem}
We begin with a suitable notion of weak solutions for the model with the double obstacle potential $\psi_{\Do}$ and the logarithmic potential $\psi_{\log}$.

\begin{defn}\label{defn:timedep}
A quintuple $(\varphi,\mu,\sigma,\v,p)$ is a weak solution to the CHB system \eqref{MEQ}-\eqref{BIC} with the double obstacle potential $\psi_{\Do}$ if the following properties hold:
\begin{enumerate}
\item[$(\mathrm{a})$] the functions satisfy
\begin{align*}
& \varphi \in H^1(0,T;(H^1)^*)\cap L^{\infty}(0,T;H^1) \cap L^2(0,T;H^2_n), \quad \mu \in L^2(0,T;H^1), \\
& \sigma \in L^{\infty}(0,T;H^2),  \quad \v \in L^2(0,T;\H^1), \quad p \in L^2(0,T;L^2)
\end{align*}
with $\varphi(0) = \varphi_0$ a.e.~in $\Omega$.
\item[$(\mathrm{b})$] equation \eqref{MEQ_1} holds a.e.~in $\Omega_T$, while \eqref{WFORM_1a}, \eqref{WFORM_1b} and \eqref{WFORM_1d} are satisfied for a.e.~$t \in (0,T)$ and for all $\boldsymbol{\Phi} \in \H^1$ and $\zeta \in H^1$.
\item[$(\mathrm{c}1)$] for a.e.~$t \in (0,T)$, $\varphi(t) \in \KK := \{ f \in H^1 : \abs{f} \leq 1 \text{ a.e.~in } \Omega \}$ and
\begin{align}\label{subdiff}
\intO (\mu + \chi \sigma + \varphi) (\zeta - \varphi) - \grad \varphi \cdot \grad(\zeta - \varphi) \dx \leq 0 \quad \forall \zeta \in \KK.
\end{align}
\end{enumerate}
We say that $(\varphi, \mu, \sigma, \v, p)$ is a weak solution to \eqref{MEQ}-\eqref{BIC} with the logarithmic potential $\psi_{\log}$ if properties $(\mathrm{a})$ and $(\mathrm{b})$ hold along with
\begin{enumerate}
\item[$(\mathrm{c}2)$] $|\varphi(x,t)| < 1$ for a.e.~$(x,t) \in \Omega \times (0,T)$ and for a.e.~$t \in (0,T)$,
\begin{align}\label{log:weak}
\intO (\mu + \chi \sigma - \psi_{\log}'(\varphi)) \zeta - \grad \varphi \cdot \grad \zeta \dx = 0 \quad \forall \zeta \in H^1.
\end{align}
\end{enumerate}
\end{defn}

Our first result concerns the existence of weak solutions to the CHB system \eqref{MEQ}-\eqref{BIC} with singular potentials.
\begin{thm}\label{thm:timedep}
Suppose \eqref{ass:dom}-\eqref{ass:visc} hold along with  
\begin{enumerate}[label=$(\mathrm{B \arabic*})$, ref = $\mathrm{B \arabic*}$]
\item \label{ass:do:source} The source terms $\Gamma_{\v}$ and $\Gamma_{\varphi}$ are of the form \eqref{source:form} with $f_{\v} \in C^1([-1,1])$, $f_{\varphi}\in C^0([-1,1])$, nonnegative $b_{\v} \in C^1([-1,1])$, nonnegative $b_{\varphi} \in C^0([-1,1])$ satisfying
\begin{align}\label{ASS_SOURCE} 
b_{\v}(\pm 1) = b_{\varphi}(\pm 1) = 0, \quad f_{\varphi}(1) - f_{\v}(1) < 0, \quad f_{\varphi}(-1) + f_{\v}(-1) > 0.
\end{align}
\item \label{ass:do:ini} The initial condition $\varphi_0 $ belongs to $\KK$.
\end{enumerate}
Then, there exists a weak solution $(\varphi, \mu, \sigma, \v, p)$ to \eqref{MEQ}-\eqref{BIC} with the double obstacle potential $\psi_{\Do}$ in the sense of Definition \ref{defn:timedep}, and $\sigma$ additionally satisfies $0 \leq \sigma \leq 1$ a.e.~in $\Omega_T$.\\

\noindent In addition, if the following assumptions are satisfied:
\begin{enumerate}[label=$(\mathrm{C \arabic*})$, ref = $\mathrm{C \arabic*}$]
\item \label{ass:log:source} There exists a constant $c$ such that for any $0 < \delta \ll 1$, 
\begin{align*}
b_\varphi(s) \leq c \delta \quad \text{ for all } s \in [-1, -1+\delta] \cup [1-\delta, 1].
\end{align*}
\item \label{ass:log:source:2}  $b_\varphi(s) \log ( \frac{1+s}{1-s} ) \in C^0([-1,1])$.
\item \label{ass:log:ini} The initial condition $\varphi_0 \in H^1(\Omega)$ satisfies $|\varphi_0(x)| < 1$ for a.e.~$x \in \Omega$.
\end{enumerate}
Then, there exists a weak solution $(\varphi, \mu, \sigma, \v, p)$ to \eqref{MEQ}-\eqref{BIC} with the logarithmic potential $\psi_{\log}$ in the sense of Definition \ref{defn:timedep}, and $0 \leq \sigma \leq 1$ a.e.~in $\Omega_T$.  Furthermore, for a.e.~$t \in (0,T)$, it holds that
\begin{equation}\label{log:energy}
\begin{aligned}
& \intO \psi_{\log}(\varphi(t)) + \tfrac{1}{2} \abs{\nabla \varphi(t)}^2 \dx + \int_0^t \intO |\nabla \mu|^2 + 2 \eta(\varphi) |\D \v|^2 + \nu |\v|^2 \dx \dt \\
& \quad \leq \int_0^t \intO - \chi \nabla \mu \cdot \nabla \sigma + (\Gamma_\varphi - \varphi \Gamma_\v)(\psi_{\log}'(\varphi) - \lap \varphi) \dx \dt \\
& \qquad + \int_0^t \intO 2 \eta(\varphi) \D \v \colon \D\u  + \nu \v\cdot \u \dx  - (\psi_{\log}'(\varphi) -\lap\varphi)\grad\varphi \cdot \u \dx \dt \\
& \qquad + \intO \psi_{\log}(\varphi_0) + \tfrac{1}{2} \abs{\nabla \varphi_0}^2 \dx,
\end{aligned}
\end{equation}
where $\u=\DD(\Gamma_{\v}(\varphi, \sigma))$.
\end{thm}

\begin{ex}\label{ex:source}
We now give a biologically relevant example of the source terms $\Gamma_{\v}$ and $\Gamma_{\varphi}$ that satisfy \eqref{ASS_SOURCE}.  Following \cite[Sec.~2.4.1]{GLSS}, in a domain $\Omega$ occupied by both tumour cells and healthy cells, we denote by $\rho_1$ the actual mass density of the healthy cells per unit volume in $\Omega$ and by $\bar \rho_1$ the (constant) mass density of the healthy cells occupying the whole of $\Omega$.  Then, it follows that $\rho_1 \in [0, \bar \rho_1]$ and the volume fraction of the healthy cells can be defined as the ratio $u_1 = \frac{\rho_1}{\bar \rho_1} \in [0,1]$.  Let $\rho_2$, $\bar \rho_2$ and $u_2$ be the actual mass density of the tumour cells per unit volume in $\Omega$, the (constant) mass density of the tumour cells occupying the whole of $\Omega$, and the volume fraction of the tumour cells, respectively.   Assuming there are no external volume compartment aside from the tumour and healthy cells, we have $u_1 + u_2 = 1$.  Then, setting $\varphi := u_2 - u_1$, we consider the function 
\begin{equation*}
\Gamma(\varphi,\sigma)\coloneqq P (1-\varphi^2) \sigma-A \varphi \quad \text{ for } \varphi \in [-1,1]
\end{equation*}
where $P, A > 0$ are the constant proliferation and apoptosis rates, so that proliferation occurs only at the interface region $\{ -1 < \varphi < 1 \}$.  From the modelling, we have
\begin{equation*}
\Gamma_{\v}=\alpha\Gamma,\quad \Gamma_{\varphi}=\rho_S\Gamma,\quad \alpha\coloneqq \frac{1}{\bar{\rho}_2}-\frac{1}{\bar{\rho}_1},\quad \rho_S\coloneqq \frac{1}{\bar{\rho}_1}+\frac{1}{\bar{\rho}_2},
\end{equation*}
and so
\begin{align*}
b_{\v}(\varphi) = \alpha P (1-\varphi^2),\quad b_{\varphi}(\varphi)=\rho_S P (1-\varphi^2) ,\quad f_{\v}(\varphi)=-\alpha A\varphi,\quad f_{\varphi}(\varphi)=-\rho_SA\varphi,
\end{align*}
where $b_{\v}(\pm 1) = b_{\varphi}(\pm 1) = 0$ and
\begin{equation*}
f_{\varphi}(1)-f_{\v}(1)= -A\frac{2}{\bar{\rho}_1}<0,\quad f_{\varphi}(-1)+f_{\v}(-1)= A\frac{2}{\bar{\rho}_2}>0.
\end{equation*}
It is also clear that $b_\varphi(s) = \rho_s P(1-s^2)$ satisfies \eqref{ass:log:source} and \eqref{ass:log:source:2}.
\end{ex}

\subsection{The stationary problem}
The stationary problem comprises of the equations \eqref{MEQ_1}, \eqref{MEQ_2}, \eqref{MEQ_4}, \eqref{MEQ_5} posed in $\Omega$ and
\begin{align}\label{MEQ_stat}
\div (\varphi \v) = \lap \mu + \Gamma_\varphi(\varphi, \sigma) \quad\text{in } \Omega,
\end{align}
together with the boundary conditions \eqref{BC_1}-\eqref{BC_3} posed on $\Sigma$.
\begin{defn}\label{defn:stat}
A quintuple $(\varphi, \mu, \sigma, \v, p)$ is a weak solution to the stationary CHB system with the double obstacle potential $\psi_{\Do}$ if the following properties hold:
\begin{enumerate}
\item[$(\mathrm{d})$] the functions satisfy
\begin{align*}
\varphi \in H^2_n, \quad \mu \in H^1, \quad \sigma \in H^2, \quad \v \in \H^1, \quad p \in L^2.
\end{align*}
\item[$(\mathrm{e})$] equation \eqref{MEQ_1} holds a.e.~in $\Omega$, while \eqref{WFORM_1a}, \eqref{WFORM_1d} and
\begin{align}\label{Stat_Form}
0 = \int_\Omega \nabla \mu \cdot \nabla \zeta + (\nabla \varphi \cdot \v + \varphi \Gamma_\v(\varphi, \sigma) - \Gamma_\varphi(\varphi, \sigma) ) \zeta \dx
\end{align}
are satisfied for all $\boldsymbol{\Phi} \in \H^1$ and $\zeta \in H^1$.  
\item[$(\mathrm{f}1)$] \eqref{subdiff} holds along with $\varphi \in \KK = \{ f \in H^1 \, : \, \abs{f} \leq 1 \text{ a.e.~in } \Omega \}$.
\end{enumerate}
We say that $(\varphi, \mu, \sigma, \v, p)$ is a weak solution to the stationary CHB system with the logarithmic potential $\psi_{\log}$ if properties $(\mathrm{d})$ and $(\mathrm{e})$ are satisfied and 
\begin{enumerate}
\item[$(\mathrm{f}2)$] \eqref{log:weak} holds along with $\abs{\varphi(x)} < 1$ for a.e.~$x \in \Omega$.
\end{enumerate}
\end{defn}

\begin{thm}\label{thm:stat}
Under \eqref{ass:dom}-\eqref{ass:visc} and \eqref{ass:do:source}, there exists a weak solution $(\varphi, \mu, \sigma, \v, p)$ to the stationary CHB model with double obstacle potential $\psi_{\Do}$ in the sense of Definition \ref{defn:stat}. If, in addition, \eqref{ass:log:source} and \eqref{ass:log:source:2} hold, then there exists a weak solution $(\varphi, \mu, \sigma, \v, p)$ to the stationary CHB model with logarithmic potential $\psi_{\log}$ in the sense of Definition \ref{defn:stat}. Furthermore, in both cases, $0 \leq \sigma \leq 1$ a.e.~in ~$\Omega_T$ and the following regularity properties hold
\begin{equation*}
\mu\in H_n^2,\quad \v\in \W^{2,q},\quad p\in W^{1,q},
\end{equation*}
for $q < \infty$ if $d = 2$ and $q = 6$ if $d = 3$.  In particular, the quintuple $(\varphi, \mu, \sigma, \v, p)$ is a strong stationary solution.
\end{thm}

\subsection{Darcy's law}
By setting the viscosities $\eta(\cdot)$ and $\lambda(\cdot)$ to zero, the CHB model \eqref{MEQ}-\eqref{BIC} reduces to a Cahn--Hilliard--Darcy  (CHD) model consisting of \eqref{MEQ_1}, \eqref{MEQ_3}-\eqref{MEQ_5} and the Darcy law
\begin{align}\label{MEQ_D}
\v = -\frac{1}{\nu} \Big (\nabla p - (\mu + \chi \sigma) \nabla \varphi \Big ) \quad\text{in } \Omega_T,
\end{align}
furnished with the initial-boundary conditions \eqref{BC_1}-\eqref{BC_2}, \eqref{IC} together with the Dirichlet boundary condition
\begin{align}\label{BC_D}
p = 0 \quad\text{on } \Sigma \times (0,T).
\end{align}
We begin with a notion of weak solutions for the CHD model with singular potentials.
\begin{defn}\label{defn:Darcy}
A quintuple $(\varphi,\mu,\sigma,\v,p)$ is a weak solution to the CHD system \eqref{MEQ_1}, \eqref{MEQ_3}-\eqref{MEQ_5},  \eqref{BC_1}-\eqref{BC_2}, \eqref{IC}, \eqref{MEQ_D}, \eqref{BC_D} with the double obstacle potential $\psi_{\Do}$ if property $(\mathrm{c}1)$ from Definition \ref{defn:timedep} holds along with:
\begin{enumerate}
\item[$(\mathrm{g})$] the functions satisfy
\begin{align*}
& \varphi \in W^{1,\frac{8}{5}}(0,T;(H^1)^*) \cap  L^{\infty}(0,T;H^1) \cap L^2(0,T;H^2_n), \quad \mu \in L^2(0,T;H^1), \\
& \sigma \in L^{\infty}(0,T;H^2),  \quad \v \in L^2(0,T;\LP^2), \quad p \in L^2(0,T;L^2) \cap L^{\frac{8}{5}}(0,T;H^1_0)
\end{align*}
with $\varphi(0) = \varphi_0$ a.e.~in $\Omega$.
\item[$(\mathrm{h})$] for a.e.~$t \in (0,T)$ and for all $\boldsymbol{\Phi} \in \H^1$, $\chi \in H^1_0$ and $\zeta \in H^1$, \eqref{WFORM_1b} and \eqref{WFORM_1d} are satisfied along with
\begin{subequations}
\begin{alignat}{2}
0 & = \int_\Omega (\nu \v - (\mu + \chi \sigma) \nabla \varphi) \cdot \boldsymbol{\Phi} - p \div(\boldsymbol{\Phi}) \dx, \label{WFORM_D_1a} \\
0 & = \int_\Omega \frac{1}{\nu} (\nabla p - (\mu + \chi \sigma) \nabla \varphi) \cdot \nabla \chi - \Gamma_{\v}(\varphi, \sigma) \chi \dx \label{WFORM_D_1b}.
\end{alignat}
\end{subequations}
\end{enumerate}
We say that $(\varphi, \mu, \sigma, \v, p)$ is a weak solution to the CHD system with the logarithmic potential $\psi_{\log}$ if property $(\mathrm{c}2)$ in Definition \ref{defn:timedep} holds along with properties $(\mathrm{g})$ and $(\mathrm{h})$.
\end{defn}

\begin{rem}
The variational equality \eqref{WFORM_D_1a} comes naturally from \eqref{WFORM_1a} when we neglect the viscosities $\eta(\varphi)$ and $\lambda(\varphi)$. Meanwhile, the variational equality \eqref{WFORM_D_1b} arises from the weak formulation of the elliptic problem obtained from taking the divergence of Darcy's law \eqref{MEQ_D} in conjunction with the equation \eqref{MEQ_1} and the boundary condition \eqref{BC_D}.
\end{rem}

\begin{thm}\label{thm:Darcy}
Under \eqref{ass:dom}-\eqref{ass:visc}, \eqref{ass:do:source} and \eqref{ass:do:ini}, there exists a weak solution $(\varphi, \mu, \sigma, \v, p)$ to the CHD model with double obstacle potential $\psi_{\Do}$ in the sense of Definition \ref{defn:Darcy}. If, in addition, \eqref{ass:log:source}-\eqref{ass:log:ini} hold, then there exists a weak solution $(\varphi, \mu, \sigma, \v, p)$ to the CHD model with logarithmic potential $\psi_{\log}$ in the sense of Definition \ref{defn:Darcy}, which satisfies, for a.e.~$t \in (0,T)$, the inequality \eqref{log:energy} with $\eta(\varphi) \equiv 0$. Furthermore, in both cases, $0 \leq \sigma \leq 1$ a.e.~in $\Omega_T$.
\end{thm}

\section{Proof of Theorem \ref{thm:timedep} -- Time dependent solutions}\label{sec:time}
The standard procedure is to approximate the singular potentials with a sequence of regular potentials, employ Lemma \ref{THM_WSOL_1} to obtain approximate solutions, derive uniform estimates and pass to the limit. 

\subsection{Approximation potentials and their properties}
\subsubsection{Double obstacle potential}
We point out that in order to use Lemma \ref{THM_WSOL_1} the approximate potential should at least belong to $C^2(\R)$.  Fix $\delta > 0$, which serves as the regularisation parameter, and we define
\begin{equation}
\label{DEF_APPROX_POT_1} \quad \hat \beta_{\Do,\delta}(r) =\begin{cases}
\frac{1}{2\delta}\left(r-\left(1+\frac{\delta}{2}\right)\right)^2+\frac{\delta}{24}&\text{ for }r\geq 1+\delta,\\
\frac{1}{6\delta^2}(r-1)^3&\text{ for }r\in (1,1+\delta),\\
0&\text{ for }|r|\leq 1,\\
-\frac{1}{6\delta^2}(r+1)^3 &\text{ for }r\in (-1-\delta,-1),\\
\frac{1}{2\delta}\left(r+\left(1+\frac{\delta}{2}\right)\right)^2+\frac{\delta}{24}&\text{ for }r\leq -1-\delta.
\end{cases}
\end{equation}
Formally, it is easy to see that $\hat \beta_{\Do,\delta}(r) \to \mathbb{I}_{[-1,1]}(r)$ as $\delta \to 0$, and so 
\begin{align}\label{DO_POT}
\psi_{\Do,\delta}(r) := \hat \beta_{\Do,\delta}(r) + \frac{1}{2}(1-r^2)
\end{align}
will serve as our approximation for the double obstacle potential.  Let $\beta_{\Do,\delta}(r) =  \hat \beta_{\Do,\delta}'(r) = (r + \psi_{\Do,\delta}'(r)) \in C^1(\R)$ denote the derivative of the convex part $\hat \beta_{\Do,\delta}$, which has the form
\begin{equation}
\label{DEF_APPROX_POT_2}\beta_{\Do,\delta}(r)=\begin{cases}
\frac{1}{\delta} \left (r-\left(1+\frac{\delta}{2}\right ) \right )&\text{ for }r\geq 1+\delta,\\
\frac{1}{2\delta^2}(r-1)^2&\text{ for }r\in (1,1+\delta),\\
0&\text{ for }|r|\leq 1,\\
-\frac{1}{2\delta^2}(r+1)^2&\text{ for }r\in (-1-\delta,-1),\\
\frac{1}{\delta} \left (r+\left(1+\frac{\delta}{2}\right) \right )&\text{ for }r\leq -1-\delta.
\end{cases}
\end{equation}
Then, it is clear that $\beta_{\Do,\delta}$ is Lipschitz continuous with $0 \leq \beta_{\Do,\delta}'(r) \leq \frac{1}{\delta}$ for all $r \in \R$. 
\begin{prop}\label{prop:DO:prop}
Let $\hat \beta_{\Do,\delta}$ and $\psi_{\Do,\delta}$ be defined as above.  Then, there exist positive constants $C_0$ and $C_1$ such that for all $r \in \R$ and for all $\delta \in (0,1/4)$,
\begin{subequations}
\begin{alignat}{3}
\label{PROP_PSIDELTA_1}\psi_{\Do,\delta}(r) & \geq C_0| r |^2 - C_1, \\
\label{PROP_PSIDELTA_3} \delta (\beta_{\Do,\delta}(r))^2  & \leq 2 \hat \beta_{\Do,\delta}(r) \leq \delta(\beta_{\Do,\delta}(r))^2 + 1, \\ \label{PROP_PSIDELTA_5}\delta( \beta_{\Do,\delta}'(r))^2 & \leq \beta_{\Do,\delta}'(r). 
\end{alignat}
\end{subequations}
\end{prop}

\begin{proof}
As $\psi_{\Do,\delta}$ is bounded for $|r| \leq 1 + \delta$, it suffices to show \eqref{PROP_PSIDELTA_1} holds for $|r| > 1 + \delta$.  By Young's inequality it is clear that for $r > 1 + \delta$ with $\delta \in (0,1/4)$,
\begin{align*}
\psi_{\Do,\delta}(r) \geq 2 \left(r-\left(1+\tfrac{\delta}{2}\right)\right)^2 - r^2 \geq C_0 |r|^2 - C_1,
\end{align*}
and a similar assertion holds also for $r < -1 - \delta$.  This establishes \eqref{PROP_PSIDELTA_1}.\\ 

\noindent From the definitions of $\hat \beta_{\Do,\delta}$ and $\beta_{\Do,\delta}$ we see that
\begin{align*}
\delta (\beta_{\Do,\delta}(r))^2 = \frac{1}{4 \delta^3}( r - 1)^4 \leq \frac{1}{4 \delta^2} (r-1)^3 \leq 2 \hat \beta_{\Do,\delta}(r) \leq \frac{\delta}{3} \leq 1 + \delta (\beta_{\Do,\delta}(r))^2
\end{align*}
for $r \in (1, 1+\delta)$, and 
\begin{align*}
\delta (\beta_{\Do,\delta}(r))^2 = \tfrac{1}{\delta} \left ( r - \left ( 1 + \tfrac{\delta}{2} \right ) \right )^2 \leq 2 \hat \beta_{\Do,\delta}(r) \leq 1+ \tfrac{1}{\delta}\left ( r - \left ( 1 + \tfrac{\delta}{2} \right ) \right )^2  \leq 1+ \delta (\beta_{\Do,\delta}(r))^2 
\end{align*}
for $r > 1 + \delta$.  Similar assertions also hold for the cases $r \in (-1-\delta, -1)$ and $r < -1 - \delta$, which then yield \eqref{PROP_PSIDELTA_3}.\\

\noindent A straightforward computation shows
\begin{align*}
\beta_{\Do,\delta}'(r) & = \delta (\beta_{\Do,\delta}'(r))^2 = \begin{cases}
\frac{1}{\delta} &  \text{ for } |r| \geq 1 + \delta, \\
0 & \text{ for } |r| \leq 1, 
\end{cases} \\
\delta (\beta_{\Do,\delta}'(r))^2 & = \frac{1}{\delta^3}(r-1)^2 \leq \frac{1}{\delta^2} (r-1) = \beta_{\Do,\delta}'(r)\text{ for } r \in (1, 1+\delta), \\
\delta (\beta_{\Do,\delta}'(r))^2 & = \frac{1}{\delta^3} (-(r+1))^2 \leq -\frac{1}{\delta^2}(r+1) = \beta_{\Do,\delta}'(r) \text{ for } r \in (-1-\delta, -1),
\end{align*}
and so \eqref{PROP_PSIDELTA_5} is established.
\end{proof}

Aside from approximating the singular potential, it would be necessary to extend the source functions $b_{\v}, b_\varphi, f_{\v}$ and $f_{\varphi}$ from $[-1,1]$ to the whole real line.  Since the solution variable $\varphi$ is supported in $[-1,1]$ (see $(\mathrm{c}1)$ of Definition \eqref{defn:timedep}), the particular form of extensions outside $[-1,1]$ does not play a significant role and we have the flexibility to choose extensions that would easily lead to uniform estimates.  Hence, unless stated otherwise we assume that $b_\v, b_\varphi, f_\v$ and $f_\varphi$ can be extended to $\R$ such that $f_{\varphi} \in C^0(\R) \cap L^{\infty}(\R)$, $f_{\v} \in C^1(\R)\cap W^{1,\infty}(\R)$, $b_{\varphi} \in C^0(\R) \cap L^{\infty}(\R)$ is nonnegative, $b_{\v} \in C^1(\R) \cap W^{1,\infty}(\R)$ is nonnegative, and fulfill
\begin{align}
\label{ASS_SOURCE_1}
b_{\v}(r)=0, \quad  b_{\varphi}(r)=0 & \quad \forall \; |r| \geq 1, \\
\label{ASS_SOURCE_2} r(f_{\varphi}(r)-f_{\v}(r)r) <0  & \quad \forall \; |r| > 1.
\end{align}
The latter implies $f_\varphi(r) - f_{\v}(r)r$ is strictly negative (resp.~positive) for $r > 1$ (resp.~$r < -1$).  Then, we can employ Lemma \ref{THM_WSOL_1} to deduce the existence of a quintuple $(\varphi_\delta, \mu_\delta, \sigma_\delta, \v_\delta, p_\delta)$ to \eqref{MEQ}-\eqref{BIC} with $\psi'_{\Do,\delta}$ replacing $\psi'$ and source terms $b_{\v}$, $b_{\varphi}$, $f_{\v}$ and $f_{\varphi}$ modified as above.  Uniform estimates will be derived in the next section and then we can pass to the limit $\delta \to 0$ to infer the existence of a weak solution to \eqref{MEQ} with the double obstacle potential in the sense of Definition \ref{defn:timedep}.  

\subsubsection{Logarithmic potential}
We define the convex part of $\psi_{\log}$ and its corresponding derivative as
\begin{equation*}
\hat{\beta}_{\log}(r) := \psi_{\log}(r)-\frac{\theta_c}{2}(1-r^2),\quad \beta_{\log}(r):= \hat{\beta}_{\log}'(r)\quad\forall r\in (-1,1).
\end{equation*}
For fixed $\delta > 0$, the approximation of $\psi_{\log}$ is the standard one:
\begin{equation}
\label{DEF_APPROX_POT_3}
\psi_{\log,\delta}(r)=\begin{cases}
\psi_{\log}(1-\delta) + \psi_{\log}'(1-\delta)\big(r-(1-\delta)\big) \\
 \quad + \frac{1}{2}\psi_{\log}''(1-\delta)\big(r-(1-\delta)\big)^2&\text{ for }r\geq 1-\delta,\\
\psi_{\log}(r)&\text{ for }|r|\leq 1- \delta,\\
\psi_{\log}(\delta-1) + \psi_{\log}'(\delta-1)\big(r-(\delta-1)\big) \\
 \quad + \frac{1}{2}\psi_{\log}''(\delta-1)\big(r-(\delta-1)\big)^2&\text{ for }r\leq -1+\delta,
\end{cases}
\end{equation}
with convex part and corresponding derivative
\begin{equation}
\label{DEF_APPROX_POT_4}\hat{\beta}_{\log,\delta}(r) := \psi_{\log,\delta}(r)-\frac{\theta_c}{2}(1-r^2), \quad \beta_{\log,\delta}(r) = \hat{\beta}_{\log,\delta}'(r) \quad \forall r \in \R.
\end{equation}

\begin{prop}\label{prop:Log:prop}
Let $\hat \beta_{\log,\delta}$ and $\psi_{\log,\delta}$ be defined as above.  Then, there exist positive constants $C_0, \dots, C_3$, such that for all $r \in \R$ and for all  $0 < \delta \leq \min(1,\theta/(4 \theta_c))$,
\begin{subequations}
\begin{alignat}{3}
\label{PROP_PSIDELTA_LOG_1}\psi_{\log,\delta}(r) & \geq C_0| r |^2 - C_1, \\ 
\label{PROP_PSIDELTA_LOG_1a} 4 \delta \theta^{-1} \hat{\beta}_{\log,\delta}(r) & \geq (| r | - 1)_{+}^2 := (\max(0, |r|-1))^2, \\
\label{PROP_PSIDELTA_LOG_2} \delta (\beta_{\log,\delta}(r))^2 & \leq 2\theta \hat \beta_{\log,\delta}(r)  + C_2 \leq C_3\left(\delta(\beta_{\log,\delta}(r))^2 + 1\right), \\ 
\label{PROP_PSIDELTA_LOG_4}\delta( \beta_{\log,\delta}'(r))^2  & \leq \theta \beta_{\log,\delta}'(r).\end{alignat}
\end{subequations}
\end{prop}
\begin{proof}
For $r \geq 1 - \delta$ with $\delta \leq \frac{\theta}{4 \theta_c}$, a short calculation shows that
\begin{align*}
\hat \beta_{\log}(1-\delta) > 0, \quad \beta_{\log}(1-\delta)(r - (1-\delta)) \geq 0, \\
\beta_{\log}'(1-\delta)(r - (1-\delta))^2 \geq  2\theta_c(r - (1-\delta))^2,
\end{align*}
where the inequality $\beta_{\log}'(1-\delta) \geq 2 \theta_c$ comes from the facts $\beta_{\log}'(1-\delta) = \frac{\theta}{\delta(2-\delta)}$ and $\delta(2-\delta) \leq 2 \delta \leq \frac{\theta}{2 \theta_c}$.  Then, it is easy to see that \eqref{PROP_PSIDELTA_LOG_1} holds with the help of Young's inequality.  Analogously, using $\hat \beta_{\log}(\delta-1) > 0$, $\beta_{\log}'(\delta-1) = \beta_{\log}'(1-\delta)$ and $\beta_{\log}(\delta-1)(r - (\delta-1)) \geq 0$ for $r \leq -1 + \delta$, we infer that \eqref{PROP_PSIDELTA_LOG_1} also holds for $r \leq -1 + \delta$ with $\delta \leq \tfrac{\theta}{4 \theta_c}$.  Meanwhile for $\abs{r} \leq 1- \delta$, using the nonnegativity of $\hat \beta_{\log}$ yields
\begin{align*}
\psi_{\log}(r) \geq \frac{\theta_c}{2}(1-r^2) \geq C_0 \abs{r}^2 - C_1 \quad \forall \abs{r} \leq 1- \delta.
\end{align*}
This completes the proof of \eqref{PROP_PSIDELTA_LOG_1}.\\

\noindent For $|r| \leq 1$, we see that $(|r|-1)_{+}^2 = 0$ and so \eqref{PROP_PSIDELTA_LOG_1a} holds trivially due to the nonnegativity of $\hat{\beta}_{\log,\delta}$.  For $r>1$, the definition of $\hat{\beta}_{\log,\delta}$ gives
\begin{equation*}
\hat{\beta}_{\log,\delta}(r)\geq \frac{1}{2}\beta_{\log}'(1-\delta)(r-(1-\delta))^2 \geq \frac{\theta}{2\delta(2-\delta)}(r-1)^2\geq \frac{\theta}{4\delta}(r-1)^2,
\end{equation*}
and similar arguments yield $\hat{\beta}_{\log,\delta}(r)\geq \frac{\theta}{4\delta}(|r|-1)^2$ for $r<-1$, which shows \eqref{PROP_PSIDELTA_LOG_1a}.\\

\noindent For \eqref{PROP_PSIDELTA_LOG_2} we first observe that $\delta (\beta_{\log,\delta}(0))^2 = 0 = \hat \beta_{\log,\delta}(0)$ and for $\delta \in (0,1]$ we have
\begin{align*}
\beta_{\log,\delta}'(s) \leq \frac{\theta}{\delta(2-\delta)} & \leq \frac{\theta}{\delta},  \quad \beta_{\log,\delta}(s) \geq 0 \quad \forall s \in [0,1-\delta], \\
0 < \delta \beta_{\log,\delta}'(s) = \frac{\delta \theta}{1-s^2} & \leq \theta,  \quad \beta_{\log,\delta}(s) \leq 0 \quad \forall s \in [-1+\delta,0],
\end{align*}
which imply
\begin{align*}
[\delta ( \beta_{\log,\delta}(s))^2]' = 2 \delta \beta_{\log,\delta}(s) \beta_{\log,\delta}'(s) \leq 2 \theta \beta_{\log,\delta}(s) \quad &\forall s \in [0,1-\delta], \\
[\delta ( \beta_{\log,\delta}(s))^2]'  = 2 \delta \beta_{\log,\delta}(s) \beta_{\log,\delta}'(s) \geq 2 \theta \beta_{\log,\delta}(s) \quad&  \forall  s\in [-1+\delta, 0].
\end{align*}
Integrating the first inequality from $0$ to $r \in (0,1-\delta]$ and the second inequality from $r \in [-1+\delta,0)$ to $0$ yields
\begin{align*}
\delta (\beta_{\log,\delta}(r))^2 \leq 2 \theta \hat \beta_{\log,\delta}(r) \quad \forall r \in [-1+\delta, 1-\delta].
\end{align*}
Taking note that over $[-1+\delta, 1-\delta]$, $\hat \beta_{\log,\delta}(r)$ is bounded uniformly in $\delta \in (0,1]$, and so we easily infer the upper bound $2 \theta \hat \beta_{\log,\delta}(r) \leq C_2 ( \delta (\beta_{\log,\delta}(r))^2 + 1)$ for some positive constant $C_2 >0$ holding for all $r \in [-1+\delta, 1-\delta]$.  Meanwhile, a direct calculation shows that for $r \geq 1- \delta$ and $\delta \in (0,1]$ we have
\begin{align*}
\delta \beta_{\log,\delta}(r)^2 & = \frac{\theta^2}{\delta(2-\delta)^2} (r - (1-\delta))^2 + \frac{\theta^2 }{2-\delta} \log \Big ( \frac{2-\delta}{\delta} \Big ) (r - (1-\delta)) + \frac{\theta^2}{4} \delta \log \Big ( \frac{2-\delta}{\delta} \Big )^2 \\
& \leq 2 \theta \hat \beta_{\log,\delta}(r)  + \frac{\theta^2}{4} \delta \log \Big ( \frac{2-\delta}{\delta} \Big)^2 - 2 \theta \hat \beta_{\log}(1-\delta) \leq 2 \theta \hat \beta_{\log,\delta}(r) + C_1, \\
\frac{\theta}{2} \hat \beta_{\log,\delta}(r) & = \frac{\theta}{2} \hat \beta_{\log}(1-\delta) + \frac{\theta^2}{4} \log \Big ( \frac{2-\delta}{\delta} \Big ) (r - (1-\delta)) + \frac{\theta^2}{2(2-\delta)\delta} (r - (1-\delta))^2 \\
& \leq  \delta \beta_{\log,\delta}(r)^2 + \frac{\theta}{2} \hat \beta_{\log}(1-\delta) - \frac{\theta^2}{4} \delta \log \Big ( \frac{2-\delta}{\delta} \Big )^2 \leq \delta \beta_{\log,\delta}(r)^2 + C
\end{align*}
on account of $\frac{1}{2} \leq \frac{1}{2-\delta} \leq 1$, the positivity of $\log((2-\delta)/\delta) (r - (1-\delta))$ and the boundedness of $\hat{\beta}_{\log}(1-\delta)$ and $\delta \log ((2-\delta)/\delta)^2$ for $\delta \in (0,1]$.  An analogous calculation leads to similar inequalities for $r \leq -1 + \delta$, and thus \eqref{PROP_PSIDELTA_LOG_2} is established.\\

\noindent For \eqref{PROP_PSIDELTA_LOG_4} a straightforward calculation using $\frac{1}{2-\delta} \leq 1$ gives
\begin{align*}
\delta (\beta_{\log,\delta}'(r))^2 = \frac{\theta^2}{(2-\delta)^2 \delta} \leq \theta \beta_{\log,\delta}'(r) \quad \forall \abs{r} \geq 1 - \delta, \\
\delta (\beta_{\log,\delta}'(r))^2 = \frac{\delta \theta^2}{(1-r^2)^2} \leq \frac{\theta^2}{(2-\delta)(1-r^2)} \leq \theta \beta_{\log,\delta}'(r) \quad \forall \abs{r} \leq 1 - \delta.
\end{align*}
This completes the proof.
\end{proof}

Once again, we extend the source functions $b_{\v}, b_\varphi, f_{\v}$ and $f_{\varphi}$ from $[-1,1]$ to the whole real line in a way that satisfies \eqref{ASS_SOURCE_1} and additionally
\begin{align}\label{ass:log:f}
r f_\v(r)  - f_\varphi(r) \begin{cases}
> 0 \text{ for } r \in [1, r_0], \\
= 0 \text{ for } \abs{r} \geq 2 r_0, \\
< 0 \text{ for } r \in [-r_0, -1],
\end{cases}
\end{align}
with smooth interpolation in $[r_0, 2 r_0]$ and $[-2r_0, r_0]$ for some fixed constant $r_0 > 1$.  

\subsection{Existence of approximate solutions}
To unify our analysis, we use the notation
\begin{align*}
\psi_\delta = \begin{cases}
\psi_{\Do,\delta} & \text{ for } \psi_{\Do}, \\
 \psi_{\log,\delta} & \text{ for } \psi_{\log},
\end{cases} \quad \Theta_c = \begin{cases}
1 & \text{ for } \psi_{\Do}, \\
\theta_c & \text{ for } \psi_{\log},
\end{cases}
\end{align*}
and denote by $\hat \beta_\delta$ the convex part of $\psi_\delta$.  Due to Proposition \ref{prop:DO:prop} and \ref{prop:Log:prop} and by the definition of $\psi_\delta(\cdot)$, we can employ Lemma \ref{THM_WSOL_1}, and infer, for every $\delta \in (0,1)$, the existence of a weak solution quintuple $(\varphi_{\delta},\mu_{\delta},\sigma_{\delta},\v_{\delta},p_{\delta})$ to \eqref{MEQ}-\eqref{BIC} with $\psi_{\delta}'$ in \eqref{MEQ_4}.  More precisely, 
\begin{subequations}
\begin{alignat}{3}
\label{WFORM_1_new}
\div(\v_\delta) = \Gamma_{\v}(\varphi_\delta, \sigma_\delta) & \quad\text{a.e.~in } \Omega_T,  \\
\mu_{\delta} = \psi_{\delta}'(\varphi_{\delta}) - \lap \varphi_{\delta}-\chi\sigma_{\delta} & \quad\text{a.e. in }\Omega_T \label{WFORM_1c_new}, 
\end{alignat}
\end{subequations}
and 
\begin{subequations}
\begin{alignat}{3}
\label{WFORM_1a_new}0 &= \intO \T(\v_\delta ,p_\delta ) : \grad \boldsymbol{\Phi}+\nu \v_\delta \cdot \boldsymbol{\Phi}  - (\mu_\delta +\chi\sigma_\delta ) \grad\varphi_\delta \cdot \boldsymbol{\Phi}\dx, \\
\label{WFORM_1b_new} 0 &= \inn{\delt \varphi_\delta}{\zeta}_{H^1}  + \intO \grad\mu_\delta \cdot\grad \zeta+ (\grad\varphi_\delta \cdot\v_\delta+ \varphi_\delta \Gamma_{\v}(\varphi_\delta,\sigma_\delta) - \Gamma_{\varphi}(\varphi_\delta,\sigma_\delta)) \zeta \dx,  \\
\label{WFORM_1d_new} 0  &=  \intO \grad\sigma_\delta \cdot\grad\zeta + h(\varphi_\delta )\sigma_\delta \zeta \dx +\intS K(\sigma_\delta-1)\zeta \dH,
\end{alignat}
\end{subequations} 
for a.e.~$t\in(0,T)$ and for all $\boldsymbol{\Phi}\in \H^1$ and $\zeta \in H^1$.  

\subsection{Uniform estimates}\label{sec:unif:do}
We begin with an auxiliary result for the product of $\beta_{\log,\delta}$ and the source terms.
\begin{lem}\label{lem:log:term}
Let $\beta_{\log,\delta}$ denote the derivative of \eqref{DEF_APPROX_POT_4}, and let $b_{\v} \in C^1(\R) \cap W^{1,\infty}(\R)$ nonnegative, $b_\varphi \in C^0(\R) \cap L^{\infty}(\R)$ nonnegative, $f_\v \in C^1(\R) \cap W^{1,\infty}(\R)$, and $f_\varphi \in C^0(\R) \cap L^{\infty}(\R)$ be given such that \eqref{ass:log:source}, \eqref{ass:log:source:2}, \eqref{ASS_SOURCE_1} and \eqref{ass:log:f} are satisfied.  Then, there exists $\delta_0 > 0$ and a positive constant $C$ independent of $\delta\in (0,\delta_0)$ such that for all $\delta \in (0,\delta_0)$, $s \geq 0$ and $r \in \R$,
\begin{align}\label{log:source:est:A}
(r \Gamma_\v(r,s) -  \Gamma_{\varphi}(r,s)) \beta_{\log,\delta}(r) \geq - C(1 + |s| + |r|).
\end{align}
\end{lem}
\begin{proof}
	We define
	\begin{align*}
	G(r,s) = s (b_{\v}(r)r - b_\varphi(r))\beta_{\log,\delta}(r).
	\end{align*}
	Then, due to \eqref{ASS_SOURCE_1}, $G(r,s) = 0$ for $s \geq 0$ and $|r| \geq 1$.  Meanwhile, for $r \in [-1, -1+\delta] \cup [1-\delta, 1]$, using the nonnegativity of $b_{\v}$ and the relation $r \beta_{\log,\delta}(r) \geq 0$ for all $r \in \R$ allow us to neglecting the nonnegative part $s b_\v(r) r \beta_{\log,\delta}(r)$ of $G(r,s)$.  Then, by \eqref{ass:log:source}, we have for $\delta \in (0,1)$, $s \geq 0$ and $r \in [1-\delta,1]$,
	\begin{align*}
	G(r,s) & \geq - s b_\varphi(r) \beta_{\log,\delta}(r) \geq - s c \delta \beta_{\log,\delta}(r) \\
	& = - s c \big ( \tfrac{\theta}{2-\delta}(r - (1-\delta)) + \tfrac{\theta \delta}{2} \log \tfrac{2-\delta}{\delta} \big ) \geq - C |s|.
	\end{align*}
A similar assertion holds for $r \in [-1,-1+\delta]$.  Lastly, for $|r| \leq 1-\delta$, we use \eqref{ass:log:source:2} to deduce that 
	\begin{align*}
	G(r,s) \geq - s b_\varphi(r) \beta_{\log,\delta}(r) = - s b_\varphi(r) \log \tfrac{1+r}{1-r} \geq - |s| \max_{r \in [-1,1]} \left | b_\varphi(r) \log \tfrac{1+r}{1-r} \right | \geq - C |s|.
	\end{align*}
	Therefore, for all $\delta > 0$, $s \geq 0$ and $r \in \R$, it holds that $G(r,s) \geq - C|s|$.  Next, we define
	\begin{align*}
	H(r) = (r f_\v(r) - f_\varphi(r))\quad \forall r\in\R.
	\end{align*}
	By continuity of $H(r)$ and \eqref{ass:log:f}, we can find a constant $\delta_0 \in (0, r_0-1)$ such that 
	\begin{align*}
	H(r)  > 0 & \text{ for } r \in (1-\delta_0, 1+\delta_0) ,\\ H(r) < 0 & \text{ for } r \in (-1-\delta_0, -1+\delta_0).
	\end{align*}
	Then, it is clear that for $|r| \geq 2 r_0$, $H(r) = 0$ thanks to \eqref{ass:log:f} satisfied by the extensions of $f_{\v}$ and $f_{\varphi}$.  Meanwhile, for any $\delta \in (0,\delta_0)$, we see that if $|r| \leq 1-\delta_0 < 1 -\delta$, then 
	\begin{align*}
	|\beta_{\log,\delta}(r)| = \big | \log \tfrac{1+r}{1-r} \big | \leq \log \tfrac{2-\delta_0}{\delta_0}, \quad |H(r)| \leq C(1 + |r|),
	\end{align*}
	which implies that 
	\begin{align*}
	H(r) \beta_{\log,\delta}(r) \geq - C ( 1 + |r|) \text{ for } |r| \leq 1- \delta_0.
	\end{align*}
	On the other hand, as $r_0 > 1$, for $r \in [-2r_0, -1+\delta_0] \cup [1-\delta_0, 2r_0]$, we use that $\beta_{\log,\delta}(r)$ and $H(r)$ have the same sign, so that their product $H(r) \beta_{\log,\delta}(r)$ is nonnegative.  Hence, combining with the above analysis for the function $G$, we obtain the assertion \eqref{log:source:est:A}.
\end{proof}
In the following, we derive uniform estimates in $\delta$, and denote by $C$ a generic constant independent of $\delta$ which may change its value even within one line.
\paragraph{Nutrient estimates.} Choosing $\zeta = \sigma_{\delta}$ in \eqref{WFORM_1d_new} and using the nonnegativity of $h(\cdot) $ leads to
\begin{equation}
\label{APRI_EQ_1}\intO |\grad\sigma_{\delta}|^2\dx   + K\intS |\sigma_{\delta}|^2 \dH \leq  K\int_{\del\Omega}\sigma_{\delta}\dH \leq \frac{K}{2}\norm{\sigma_{\delta}}_{L^2(\Sigma)}^2 + \frac{K}{2}|\Sigma|,
\end{equation}
from which we deduce that $\sigma_\delta$ is uniformly bounded in $L^{\infty}(0,T;H^1)$.  Elliptic regularity additionally yields
\begin{equation}
\label{APRI_EQ_4}\norm{\sigma_{\delta}}_{L^{\infty}(0,T;H^2)} \leq C.
\end{equation}
By a standard argument with the comparison principle, testing with $\zeta = - (\sigma_\delta)_{-}$ and $\zeta = (\sigma_\delta - 1)_{+}$ will yield the pointwise a.e.~boundedness of $\sigma_\delta$:
\begin{align*}
0 \leq \sigma_\delta \leq 1 \quad\text{a.e.~in } \Omega_T.
\end{align*}
Then, the continuous embedding $H^2\hookrightarrow L^\infty$ and the assumptions on $b_{\v}$, $b_{\varphi}$, $f_{\v}$ and $f_{\varphi}$ lead to
\begin{equation}
\label{APRI_EQ_5}\norm{\Gamma_{\varphi}}_{L^{\infty}(0,T;L^{\infty})} + \norm{\Gamma_{\v}}_{L^{\infty}(0,T;L^{\infty})}\leq C.
\end{equation}

\paragraph{Estimates from energy identity.} Thanks to \eqref{APRI_EQ_5}, the function $\u := \DD(\Gamma_{\v}) \in \H^1$ satisfies
the estimate
\begin{align}
\label{APRI_EQ_7}\normW{1}{p}{\u}\leq C\norml{p}{\Gamma_{\v}}\leq C\quad\forall p\in (1,\infty).
\end{align}
Technically, we should stress the dependence of $\u$ on $\delta$, but in light of the uniform estimate \eqref{APRI_EQ_7} we infer that $\u_\delta$ is bounded in $L^{\infty}(0,T;\W^{1,p})$ for any $p \in (1,\infty)$.  Henceforth, we drop the index $\delta$ and reuse the variable $\u$.

Choosing $\boldsymbol{\Phi}=\v_{\delta}-\u$ in \eqref{WFORM_1a_new}, $\zeta = \mu_{\delta}+\chi\sigma_{\delta}$ in \eqref{WFORM_1b_new}, using \eqref{WFORM_1c_new}  and summing the resulting identities, we obtain
\begin{equation}\label{APRI_EQ_8}
\begin{aligned}
 &\frac{\d}{\dt} \intO \psi_{\delta}(\varphi_{\delta})+\frac{1}{2}|\grad\varphi_{\delta}|^2 \dx  +\intO |\grad\mu_{\delta}|^2 + 2\eta(\varphi_{\delta})|\D\v_{\delta}|^2 +\nu|\v_{\delta}|^2\dx \\
&\quad = \intO -\chi\grad\mu_{\delta}\cdot\grad\sigma_{\delta} + (\Gamma_{\varphi}-\varphi_{\delta}\Gamma_{\v})(\beta_{\delta}(\varphi_\delta) - \Theta_c\varphi_\delta-\lap\varphi_{\delta})\dx \\
& \qquad +\intO 2\eta(\varphi_{\delta}) \D\v_{\delta}\colon \D\u  + \nu\v_{\delta}\cdot \u \dx  - \intO (\beta_{\delta}(\varphi_\delta) - \Theta_c\varphi_\delta -\lap\varphi_{\delta})\grad\varphi_{\delta} \cdot \u \dx,
\end{aligned}
\end{equation}
where we used the fact that $\psi_{\delta}$ is a quadratic perturbation of a convex function $\hat \beta_{\delta}$, in conjunction with \cite[Lem.~4.1]{RS} to obtain the time derivative of the energy.  By Young's inequality and the estimates \eqref{APRI_EQ_4}, \eqref{APRI_EQ_5} and \eqref{APRI_EQ_7}, we find that
\begin{equation}\label{APRI_EQ_9}
\begin{aligned}
& \left | \intO 2\eta(\varphi_{\delta}) \D\v_{\delta}\colon \D\u  + \nu\v_{\delta}\cdot \u  -\chi\grad\mu_{\delta}\cdot\grad\sigma_{\delta} \dx \right | \\
& \qquad + \left | \intO (\Gamma_{\varphi}-\varphi_{\delta}\Gamma_{\v} - \grad\varphi_{\delta} \cdot \u)( \Theta_c\varphi_\delta + \lap\varphi_{\delta}) \dx \right | \\
& \quad \leq \normL{2}{\sqrt{\eta(\varphi_\delta)} \D \v_\delta}^2 + \frac{\nu}{2} \normL{2}{\v_\delta}^2 + \frac{1}{4} \normL{2}{\nabla \mu_\delta}^2 + 2 \eps \norml{2}{\lap \varphi_\delta}^2 + C \left( 1 + \normh{1}{\varphi_\delta}^2 \right)
\end{aligned}
\end{equation}
for a positive constant $\eps$ yet to be determined.  It remains to control the two terms with $\beta_\delta(\varphi_\delta)$.  Integrating by parts and employing the definition of $\u = \DD(\Gamma_{\v})$ leads to 
\begin{equation}\label{APRI_EQ_16}
\begin{aligned}
& \intO \beta_{\delta}(\varphi_{\delta})\grad\varphi_{\delta}\cdot\u\dx = \intO \grad(\hat{\beta}_{\delta}(\varphi_{\delta}))\cdot\u\dx\\
 & \quad = \frac{1}{|\Sigma|}\left(\intO \Gamma_{\v}\dx\right)\intS \hat{\beta}_{\delta}(\varphi_{\delta})\dH - \intO \hat{\beta}_{\delta}(\varphi_{\delta})\Gamma_{\v}\dx.
\end{aligned}
\end{equation}
Using \eqref{APRI_EQ_7} and the relation $\hat \beta_{\delta}(r) = \psi_{\delta}(r) -\frac{\Theta_c}{2}(1-r^2)$, we obtain
\begin{equation}\label{APRI_EQ_17}
\begin{aligned}
\left|\intO \hat \beta_{\delta}(\varphi_{\delta})\Gamma_{\v}\dx\right| &= \left|\intO \left(\psi_{\delta}(\varphi_{\delta})-\tfrac{\Theta_c}{2}(1-\varphi_{\delta}^2)\right)\Gamma_{\v}\dx \right| \\
&\leq C\left(1+\norml{1}{\psi_{\delta}(\varphi_{\delta})} + \norml{2}{\varphi_{\delta}}^2\right).
\end{aligned}
\end{equation}
Meanwhile, using \eqref{PROP_PSIDELTA_3}, \eqref{PROP_PSIDELTA_5}, \eqref{PROP_PSIDELTA_LOG_2}, \eqref{PROP_PSIDELTA_LOG_4} and the trace theorem, it follows that
\begin{equation}\label{APRI_EQ_18}
\begin{aligned}
\norm{\hat \beta_{\delta}(\varphi_\delta)}_{L^1(\Sigma)} 
&\leq C + C\norm{\sqrt{\delta}\,\beta_{\delta}(\varphi_{\delta})}_{L^2(\Sigma)}^2 \\
 &\leq C \left(\norml{2}{\sqrt{\delta}\, \beta_{\delta}(\varphi_{\delta})}^2 + \normL{2}{\sqrt{\delta} \beta_{\delta}'(\varphi_{\delta})\grad\varphi_{\delta}}^2\right) +  C\\
 &\leq  C\left(\norml{1}{\hat{\beta}_{\delta}(\varphi_{\delta})} + \intO \beta_{\delta}'(\varphi_{\delta})|\grad\varphi_{\delta}|^2\dx \right) +  C\\
&\leq C \left ( 1 + \norml{2}{\varphi_\delta}^2 + \norml{1}{\psi_\delta(\varphi_\delta)} \right ) + C_* \intO \beta_{\delta}'(\varphi_{\delta})|\grad\varphi_{\delta}|^2\dx
\end{aligned}
\end{equation}
for positive constants $C$ and $C_*$ independent of $\delta$. To deal with the remaining term we use the definition \eqref{DEF_APPROX_POT_2} of $\beta_{\Do,\delta}$ and \eqref{ASS_SOURCE_1}-\eqref{ASS_SOURCE_2} to deduce that, for any $s \geq 0$ and $r \in \R$,
\begin{align*}
\beta_{\Do,\delta}(r)(\Gamma_{\varphi}(r,s) - r \Gamma_{\v}(r,s))  \begin{cases} = 0 & \text{ for } |r| \leq 1 \text{ and } s \geq 0, \\
< 0 & \text{ for } |r| > 1 \text{ and } s \geq 0,
\end{cases}
\end{align*}
which implies that
\begin{align}\label{source:doleq0}
\intO (\Gamma_{\varphi}(\varphi_\delta, \sigma_\delta)-\varphi_{\delta}\Gamma_{\v} (\varphi_\delta, \sigma_\delta))\beta_{\Do,\delta}(\varphi_{\delta})\dx \leq 0.
\end{align}
Meanwhile, for the other part with $\beta_{\log,\delta}$, we use \eqref{log:source:est:A} and \eqref{APRI_EQ_4} to obtain
\begin{align}\label{log:source:est}
\int_\Omega ( \Gamma_{\varphi}(\varphi_\delta, \sigma_\delta) - \varphi_\delta \Gamma_\v(\varphi_\delta, \sigma_\delta)) \beta_{\log,\delta}(\varphi_\delta) \dx \leq  C \Big ( 1 + \norml{1}{\varphi_\delta} \Big ).
\end{align} 
Combining \eqref{source:doleq0} with \eqref{log:source:est}, we find that 
\begin{align*}
\intO (\Gamma_{\varphi}(\varphi_\delta, \sigma_\delta)-\varphi_{\delta}\Gamma_{\v} (\varphi_\delta, \sigma_\delta))\beta_{\delta}(\varphi_{\delta})\dx  \leq C \big ( 1 + \norml{2}{\varphi_\delta} \big ),
\end{align*}
and so, when substituting \eqref{APRI_EQ_9}-\eqref{APRI_EQ_18} and \eqref{log:source:est} into \eqref{APRI_EQ_8}, we arrive at
\begin{align*}
& \frac{d}{dt}  \intO \psi_\delta(\varphi_\delta) + \frac{1}{2} |\nabla \varphi_\delta|^2 \dx + \intO \frac{3}{4} |\nabla \mu_\delta|^2 +  \eta(\varphi_\delta) |\D \v_\delta|^2 + \frac{\nu}{2} |\v_\delta|^2 \dx \\
& \quad \leq C \Big ( 1 + \norml{1}{\psi_\delta(\varphi_\delta)} + \norml{2}{\varphi_\delta}^2 + \normL{2}{\nabla \varphi_\delta}^2 \Big ) \\
& \qquad + 2 \eps \norml{2}{\lap \varphi_\delta}^2 +  C_* \intO \beta_\delta'(\varphi_\delta) |\nabla \varphi_\delta|^2 \dx.
\end{align*}
Testing  \eqref{WFORM_1c_new} with $-A\lap\varphi_{\delta}$ for some positive constant $A$, integrating by parts and using $\deln \varphi_\delta = 0$ on $\Sigma$ and \eqref{APRI_EQ_4} yields
\begin{equation}\label{APRI_EQ_15}
\begin{aligned}
A\intO |\lap\varphi_{\delta}|^2 + \beta_{\delta}'(\varphi_{\delta})|\grad\varphi_{\delta}|^2\dx & = A\intO \grad(\mu_{\delta}+\chi\sigma_{\delta})\cdot\grad\varphi_{\delta} + \Theta_c |\grad\varphi_{\delta}|^2\dx \\
& \leq C \left(1+\normL{2}{\grad\varphi_{\delta}}^2\right) + \frac{1}{4} \norml{2}{\grad\mu_{\delta}}^2.
\end{aligned}
\end{equation}
Then, summing up the last two inequalities and choosing $A > C_*$ and $\eps < \frac{A}{4}$ yields
\begin{equation}\label{APRI_EQ_19}
\begin{aligned}
& \frac{d}{dt} \left ( \norml{1}{\psi_{\delta}(\varphi_\delta)} + \norml{2}{\nabla \varphi_\delta}^2 \right ) - C \left ( \norml{1}{\psi_{\delta}(\varphi_\delta)} + \norml{2}{\nabla \varphi_\delta}^2 \right ) + \norml{2}{\nabla \mu_\delta}^2 \\
& \qquad + \normL{2}{(\beta_{\delta}'(\varphi_\delta))^{1/2} \nabla \varphi_\delta}^2 + \norml{2}{\lap \varphi_\delta}^2 + \norml{2}{(\eta(\varphi_\delta))^{1/2}  \D \v_\delta}^2 + \nu \norml{2}{\v_\delta}^2  \\
& \quad \leq C.
\end{aligned}
\end{equation}
Before applying a Gronwall argument, we note that for the double obstacle potential, the assumption \eqref{ass:do:ini} implies $\hat \beta_{\Do,\delta}(\varphi_0) = 0$, and for the logarithmic potential, the assumption \eqref{ass:log:ini} implies there exists $\delta_1 > 0$ such that $|\varphi_0(x)| \leq 1 - \delta_1$ for a.e.~$x \in \Omega$, and so $\hat \beta_{\log, \delta}(\varphi_0)$ is uniformly bounded.  Hence, for $0 < \delta < \min (1, \theta/(4 \theta_c), \delta_0, \delta_1)=: \delta_*$, we see that 
\begin{align*}
\norml{1}{\hat \beta_{\delta}(\varphi_0)} \leq C.
\end{align*}
Integrating \eqref{APRI_EQ_19} in time from $0$ to $s\in (0,T]$, using \eqref{PROP_PSIDELTA_1}, \eqref{PROP_PSIDELTA_LOG_1}, Korn's inequality and elliptic regularity theory, we deduce the uniform estimate
\begin{align}
\nonumber&\norm{\psi_{\delta}(\varphi_{\delta})}_{L^{\infty}(0,T;L^1)} + \norm{\varphi_{\delta}}_{L^{\infty}(0,T;H^1)\cap L^2(0,T;H^2)} + \norm{\grad\mu_{\delta}}_{L^2(0,T;\LP^2)} \\
\label{APRI_EQ_20}&\quad + \norm{(\beta_{\delta}'(\varphi_{\delta}))^{1/2}\grad\varphi_{\delta}}_{L^2(0,T;\LP^2)}+ \norm{\v_{\delta}}_{L^2(0,T;\H^1)}\leq C.
\end{align}
By the Sobolev embedding $\H^1\subset \LP^6$ and \eqref{APRI_EQ_20}, it follows that
\begin{align*}
\intT \norml{3/2}{\grad\varphi_{\delta} \cdot \v_{\delta}}^2\dt\leq \intT \normL{2}{\grad\varphi_{\delta}}^2\normL{6}{\v_{\delta}}^2\dt\leq \norm{\varphi_{\delta}}_{L^{\infty}(0,T;H^1)}^2\norm{\v_{\delta}}_{L^2(0,T;\H^1)}^2 \leq C.
\end{align*}
A similar argument together with \eqref{APRI_EQ_5} shows that $\varphi_\delta \Gamma_{\v}$ is bounded in $L^2(0,T;L^{\frac{3}{2}})$.  Then, from \eqref{WFORM_1b} we obtain
\begin{align}
\label{APRI_EQ_22} \norm{\delt\varphi_{\delta}}_{L^2(0,T;(H^1)^*)} + \norm{\div(\varphi_{\delta}\v_{\delta})}_{L^2(0,T; L^{3/2})}\leq C.
\end{align}
Furthermore, we find that the mean value $(\varphi_\delta)_\Omega$ satisfies
\begin{align*}
\abs{\delt  (\varphi_\delta)_{\Omega}} = \frac{1}{\abs{\Omega}} \left | \intO \Gamma_{\varphi}(\varphi_\delta, \sigma_\delta) - \varphi_\delta \Gamma_{\v}(\varphi_\delta, \sigma_\delta) - \nabla \varphi_\delta \cdot \v_\delta \dx \right | \in L^2(0,T),
\end{align*}
and so 
\begin{align}\label{APRI_EQ_23}
\| (\varphi_\delta)_{\Omega} \|_{H^1(0,T)} \leq C.
\end{align}
In particular, by the fundamental theorem of calculus, it holds that
\begin{align}\label{APRI_EQ_24}
\left |(\varphi_\delta)_{\Omega}(r) - (\varphi_\delta)_{\Omega}(s) \right | = \left | \int_s^r \delt (\varphi_\delta)_{\Omega}(t) \dt \right | \leq C |r - s|^{\frac{1}{2}}.
\end{align}

\subsection{Estimates for the mean value of the chemical potential}
\label{sec:mu:mean} 
In order to pass to the limit $\delta \to 0$ rigorously, it remains to derive uniform estimates for $\mu_\delta$, $\beta_{\delta}(\varphi_\delta)$ and $p_\delta$ in $L^2(0,T;L^2)$. To do so we appeal to the method introduced in \cite{GLS}, which involves first deducing that the limit $\varphi$ of $\varphi_\delta$ has mean value in the open interval $(-1,1)$ for all times.
We first state a useful auxiliary result.
\begin{prop}\label{prop:CFM}
	For $\delta \in (0,1)$, let $\beta_{\log,\delta}$ denote the derivative of \eqref{DEF_APPROX_POT_4}.  Then, there exist positive constants $c_1$ and $c_2$ independent of $\delta$ such that 
	\begin{align}\label{beta:log}
	r \beta_{\log,\delta}(r) \geq |\beta_{\log,\delta}(r)| - c_1 |r| - c_2 \quad \forall r \in \R.
	\end{align}
\end{prop}

\begin{rem}
	This is more refined than the following estimate from \cite[(2.12)]{CFM}:
	\begin{align}\label{CFM:est}
	r \beta_{\log,\delta}(r) \geq \tilde c_0 |\beta_{\log,\delta}(r)| - \tilde c_1 \quad \forall r \in \R
	\end{align}
	with positive constant $\tilde c_0$ and nonnegative constant $\tilde c_1$ that are independent of $\delta$, provided $\delta$ is sufficiently small, in which the constant $\tilde c_0$ is not quantified.
\end{rem}

\begin{proof}
	From the definition of $\hat \beta_{\log,\delta}$ in \eqref{DEF_APPROX_POT_4}, we infer that for $r \geq 1 - \delta$, $\delta \in (0,1)$,
	\begin{align*}
	 \beta_{\log,\delta}(r) r & \geq \beta_{\log,\delta}(r) - \delta \beta_{\log,\delta}(r) \\
&  = \beta_{\log,\delta}(r) - \tfrac{\theta}{2-\delta}(r - (1-\delta)) - \delta \log \tfrac{2-\delta}{\delta} \\
	&  \geq \beta_{\log,\delta}(r) - \theta (r - (1-\delta)) - c
	\end{align*}
	for some positive constant $c$ independent of $\delta \in (0,1)$.  In a similar fashion, using $\beta_{\log}(\delta-1) < 0$ and $\beta_{\log}'(\delta - 1) > 0$, we have for $r \leq -1 + \delta$,
	\begin{align*}
	\beta_{\log}(\delta-1) r  & \geq - \beta_{\log}(\delta-1) + \delta \beta_{\log}(\delta-1) \geq |\beta_{\log}(\delta-1)| - c, \\
	\beta_{\log}'(\delta-1)(r - (\delta-1)) r & \geq -\beta_{\log}'(\delta-1)(r - (\delta-1)) + \delta \beta_{\log}'(\delta-1)(r - (\delta-1)) \\
	& = |\beta_{\log}'(\delta-1)(r - (\delta-1))| + \tfrac{\theta}{2-\delta}(r - (\delta-1)), 
	\end{align*}
	and when combined this yields
	\begin{align*}
	\beta_{\log,\delta}(r)r & \geq |\beta_{\log}'(\delta-1)(r - (\delta-1))| + |\beta_{\log}(\delta-1)| - \theta |r-(\delta-1)| - c \\
	& \geq |\beta_{\log,\delta}(r)| - \theta |r| - c.
	\end{align*}
	For the remaining case $\abs{r} \leq 1- \delta$, we employ the fact that $\beta_{\log,\delta}(r) = \beta_{\log}(r)$ and
	\begin{align*}
	\lim_{r \to 1^-} (1-r) \beta_{\log}(r) = 0, \quad \lim_{r \to (-1)^+} (r+1) \beta_{\log}(r) = 0
	\end{align*}
	to infer the existence of a constant $c > 0$ independent of $\delta \in (0,1)$ such that
	\begin{align*}
	\beta_{\log}(r) (r - 1) \geq -c \text{ for } 0 \leq r < 1, \quad \beta_{\log}(r)(r+1) \geq - c \text{ for } -1 < r \leq 0.
	\end{align*}
	Hence, for $\delta \in (0,1)$,
	\begin{align*}
	\beta_{\log,\delta}(r) r = \beta_{\log}(r) r \geq |\beta_{\log}(r)| - c = |\beta_{\log,\delta}(r)| - c \quad \forall |r| \leq 1 - \delta.
	\end{align*}
	This completes the proof.
\end{proof}

Now, using reflexive weak compactness arguments (Aubin--Lions theorem) and \cite[Sec.~8, Cor.~4]{Simon}, for $\delta\to 0$ along a non-relabelled subsequence, we infer that
\begin{equation}\label{conv}
\begin{alignedat}{3}
\varphi_{\delta} &\to \varphi &&\quad\text{weakly-}*&&\quad \text{ in }H^1(0,T;(H^1)^*)\cap L^{\infty}(0,T;H^1)\cap L^2(0,T;H^2),\\
\varphi_{\delta} & \to\varphi && \quad \text{strongly} && \quad \text{ in }C^0([0,T];L^r)\cap L^2(0,T;W^{1,r}) \text{ and a.e. in }\Omega_T, \\
\sigma_{\delta} &\to\sigma &&\quad \text{weakly-}*&&\quad \text{ in }L^{\infty}(0,T;H^2),\\
\grad\mu_{\delta} &\to\boldsymbol{\xi} &&\quad\text{weakly}&&\quad\text{ in }L^2(0,T;\LP^2),\\
\v_{\delta} &\to \v &&\quad \text{weakly}&&\quad\text{ in }L^{2}(0,T;\H^1),\\
\div(\varphi_{\delta}\v_{\delta})& \to\theta &&\quad \text{weakly}&&\quad\text{ in } L^2(0,T;L^{\frac{3}{2}}),
\end{alignedat}
\end{equation}
for some limit functions $\boldsymbol{\xi}\in L^2(0,T;\LP^2)$, $\theta\in L^2(0,T;L^{\frac{3}{2}})$ and for all $r\in [1,6)$.  The interpolation inequality $\normh{1}{f} \leq C \norml{2}{f}^{1/2} \normh{2}{f}^{1/2}$, the boundedness of $\varphi_\delta - \varphi$ in $L^2(0,T;H^2)$ and the strong convergence $\varphi_\delta \to \varphi$ in $L^{\infty}(0,T;L^2)$ also allow us to deduce that 
\begin{align}
\label{APRI_EQ_25}\varphi_{\delta}\to\varphi \text{ strongly in }L^4(0,T;H^1).
\end{align}
Let $\lambda \in L^4(0,T;L^3)$ be an arbitrary test function, then \eqref{APRI_EQ_25} implies $\lambda \varphi_\delta \to \lambda \varphi$ strongly in $L^2(0,T;L^2)$ and $\lambda \nabla \varphi_\delta \to \lambda \nabla \varphi$ strongly in $L^2(0,T;L^{\frac{6}{5}})$. Using the weak convergence of $\v_{\delta}\to\v$ in $L^2(0,T;\H^1)$ and the product of weak-strong convergence we obtain
\begin{align}\label{div:id}
\intT\intO \div(\varphi_{\delta}\v_{\delta})\lambda \dx\dt\to \intT\intO \div(\varphi\v)\lambda \dx\dt \quad\text{as }\delta\to 0.
\end{align}
This implies $\div(\varphi_{\delta}\v_{\delta})\to\div(\varphi\v)$ weakly in $L^{\frac{4}{3}}(0,T;L^{\frac{3}{2}})$ as $\delta\to 0$. Since $L^2(0,T;L^{\frac{3}{2}})\subset L^{\frac{4}{3}}(0,T;L^{\frac{3}{2}})$, by uniqueness of weak limits we obtain $\div(\varphi\v) = \theta$. Using the assumption on $\Gamma_{\varphi}$ and the above convergences, we can pass to the limit in \eqref{WFORM_1b} to obtain
\begin{align}
\label{APRI_EQ_26} \inn{\delt \varphi}{\zeta}_{H^1} + \intO \div(\varphi\v)\zeta \dx = \intO -\boldsymbol{\xi}\cdot\grad \zeta + \Gamma_{\varphi}(\varphi,\sigma)\zeta \dx
\end{align}
for a.e.~$t\in (0,T)$ and for all $\zeta \in H^1$.   Technically, one would multiply \eqref{WFORM_1b} with a function $\kappa \in C^{\infty}_c(0,T)$, integrate the resulting product over $(0,T)$, pass to the limit $\delta \to 0$ and then recover \eqref{APRI_EQ_26} with the fundamental lemma of calculus of variations.  

Now, for the obstacle potential, the uniform boundedness of $\psi_{\delta}(\varphi_\delta)$ in $L^1(0,T;L^1)$ and \eqref{PROP_PSIDELTA_3} imply $\sqrt{\delta} \beta_{\Do,\delta}(\varphi_\delta)$ is uniformly bounded in $L^2(0,T;L^2)$, and so $\delta \beta_{\Do,\delta}(\varphi_\delta) \to 0$ strongly in $L^2(0,T;L^2)$. 
However, from the definition of $\beta_{\Do,\delta}$ we have
\begin{align}\label{del:beta}
\delta \beta_{\Do,\delta}(r) \to g(r) := \begin{cases}
r-1 & \text{ if } r \geq 1,\\
0 & \text{ if } \abs{r} \leq 1, \\
r+1 & \text{ if } r \leq -1,
\end{cases}\quad \text{as } \delta \to 0,
\end{align}
which implies that (c.f.~\cite[Proof of Thm. 2.2]{BE})
\begin{align}
\label{APRI_EQ_27} |\varphi| \leq 1 \quad\text{a.e. in } \Omega_T.
\end{align}
For the logarithmic potential, we use \eqref{PROP_PSIDELTA_LOG_1a} and the uniform boundedness of $\hat{\beta}_{\log,\delta}(\varphi_\delta)$ in $L^1(0,T;L^1)$ to obtain that
\begin{equation*}
\intOT (|\varphi_{\delta}|-1)_{+}^2\dx\dt\leq C\delta.
\end{equation*}
Since $\varphi_\delta\to\varphi$ a.e.~in $\Omega_T$ and strongly in $L^2(0,T;L^2)$, passing to the limit $\delta\to 0$ in the inequality above also implies \eqref{APRI_EQ_27}.
From this we claim that $\varphi_{\Omega}(t) \in (-1,1)$ for all $t \in (0,T)$.  Indeed, choosing $\zeta = 1$ in \eqref{APRI_EQ_26} leads to
\begin{align}
\label{APRI_EQ_28}\inn{\delt \varphi}{1}_{H^1}+ \intO \grad\varphi\cdot \v\dx = \intO \Gamma_{\varphi}(\varphi,\sigma)-\varphi\Gamma_{\v}(\varphi,\sigma)\dx \quad\text{for a.e.~}t\in (0,T).
\end{align}
Suppose to the contrary there exists a time $t_*\in (0,T)$ such that $\varphi_{\Omega}(t_*) = 1$ and \eqref{APRI_EQ_28} holds. Due to \eqref{APRI_EQ_27}, this implies $\varphi(t_*,x)\equiv 1$ a.e.~in $\Omega$ and thus $\grad \varphi(t_*,x)\equiv 0$ a.e.~in $\Omega$. Using \eqref{APRI_EQ_28} and \eqref{ASS_SOURCE_1}-\eqref{ASS_SOURCE_2}, we obtain
\begin{equation*}
\inn{\delt \varphi(t_*)}{1}_{H^1} = \intO f_{\varphi}(1)-f_{\v}(1) \dx<0.
\end{equation*}
Hence, by continuity of $t \mapsto (\varphi(t))_\Omega$, the mean value $(\varphi(t))_\Omega$ must be strictly decreasing in a neighbourhood of $t_*$, i.e., $(\varphi(t))_\Omega>1$ for $t< t_*$, which contradicts \eqref{APRI_EQ_27}. Using a similar argument and the assumption $f_{\varphi}(-1)+f_{\v}(-1)>0$ leads to the conclusion that $(\varphi(t))_\Omega > -1$ for all $t \in (0,T)$.  In particular, $(\varphi(t))_\Omega \in (-1,1)$ for all $t\in (0,T)$.  

This allows us to derive uniform estimates on the mean value of $\mu_\delta$. 
Integrating \eqref{WFORM_1c_new} and taking the modulus on both sides gives 
\begin{align}\label{mean:mu:1}
\left|\intO \mu_{\delta}(t)\dx\right| \leq \intO |\beta_{\delta}(\varphi_{\delta}(t))| + \Theta_c|\varphi_\delta(t)| + \chi |\sigma_\delta(t)| \dx
\end{align}
for a.e.~$t \in (0,T)$.  Using \eqref{beta:log} and the fact 
\begin{align}\label{DO:fact}
|\beta_{\Do,\delta}(r)|\leq r \beta_{\Do,\delta}(r) \quad \text{ for all } r\in\R
\end{align}
(which unfortunately does not hold for $\beta_{\log,\delta}$, hence the necessity of Proposition \ref{prop:CFM}), we deduce that 
\begin{align*}
|\beta_{\delta}(r)| \leq r \beta_{\delta}(r) + c_1 |r| + c_2  \quad \text{ for all } r \in \R,
\end{align*}
and so, we arrive at
\begin{equation}\label{APRI_EQ_29}
\begin{aligned}
\left|\intO \mu_{\delta}(t)\dx\right| & \leq \intO |\beta_{\delta}(\varphi_\delta(t))| + \Theta_c|\varphi_\delta(t)| + \chi |\sigma_\delta(t)|\dx,  \\
& \leq 
 \intO \varphi_\delta(t) \beta_{\delta}(\varphi_\delta(t)) + (c_1 + \Theta_c)|\varphi_\delta(t)| + \chi |\sigma_\delta(t)| + c_2 \dx.
\end{aligned}
\end{equation}
Together with the following equality obtained from testing \eqref{WFORM_1c_new} with $\zeta = \varphi_{\delta}$,
\begin{align*}
\normL{2}{\grad\varphi_{\delta}(t)}^2 + \intO \beta_{\delta}(\varphi_{\delta}(t))\varphi_{\delta}(t) - \chi\sigma_{\delta}(t)\varphi_{\delta}(t)\dx = \intO \mu_{\delta}(t)\varphi_{\delta}(t) + \Theta_c|\varphi_{\delta}(t)|^2\dx
\end{align*}
we see that
\begin{equation}\label{mu:mean:1}
\begin{aligned}
\left|\intO \mu_{\delta}(t) \dx\right| \leq \int_\Omega \mu_\delta(t) \varphi_\delta(t) \dx + C \left ( 1 + \norml{2}{\varphi_\delta(t)}^2 +\norml{2}{\sigma_\delta(t)}^2 \right).
\end{aligned}
\end{equation}
Now, let $f_\delta\in H_n^2\cap L_0^2$ be the unique solution to the Neumann problem
\begin{equation}\label{DO_f_delta}
\begin{cases}
-\lap f_{\delta} = \varphi_{\delta}(t)-(\varphi_{\delta}(t))_\Omega & \text{ in }\Omega,\\
\deln f_{\delta} = 0 & \text{ on }\Sigma,
\end{cases}
\end{equation}
where by Poincar\'e's inequality, we have
\begin{equation}
\label{APRI_EQ_30}\normh{1}{f_{\delta}} \leq C\normL{2}{\grad\varphi_{\delta}(t)}.
\end{equation}
Testing \eqref{WFORM_1b_new} with $f_{\delta}$, integrating by parts and rearranging yields
\begin{align*}
\intO \mu_{\delta}(t)\varphi_{\delta}(t)\dx &= -\inn{\delt\varphi_{\delta}(t)}{f_{\delta}}_{H^1} - \intO \big(\div(\varphi_{\delta}(t)\v_{\delta}(t))  -\Gamma_{\varphi}(\varphi_{\delta}(t),\sigma_{\delta}(t))\big)f_{\delta}\dx\\
&\quad + ((\varphi_{\delta}(t))_\Omega-(\varphi(t))_\Omega + (\varphi(t))_\Omega )\intO \mu_{\delta}(t)\dx.
\end{align*}
Plugging in this identity into \eqref{mu:mean:1} and rearranging again, we deduce that
\begin{equation}\label{APRI_EQ_31}
\begin{aligned}
&\left(1-|(\varphi(t))_\Omega|-\sup_{t\in (0,T)}|(\varphi_{\delta}(t)-\varphi(t))_\Omega|\right)\left|\intO \mu_{\delta}(t)\dx\right|\\
&\qquad \leq   C \left ( \norml{2}{\sigma_{\delta}(t)}^2 + \norml{2}{\varphi_{\delta}(t)}^2 \right )  -\inn{\delt\varphi_{\delta}(t)}{f_{\delta}}_{H^1} \\
&\quad \qquad - \intO \big(\div(\varphi_{\delta}(t)\v_{\delta}(t)) - \Gamma_{\varphi}(\varphi_{\delta}(t),\sigma_{\delta}(t))\big)f_{\delta}\dx
\end{aligned}
\end{equation}
for a.e.~$t\in (0,T)$. Recalling \eqref{APRI_EQ_23}-\eqref{APRI_EQ_24}, we have the equiboundedness and equicontinuity of $\{ (\varphi_\delta)_\Omega\}_{\delta \in (0,1)}$ so that by the Arzel\`a--Ascoli theorem,
\begin{align*}
(\varphi_\delta(t))_\Omega \to (\varphi(t))_\Omega\quad \text{strongly in } C^0([0,T])\quad \text{as } \delta \to 0.
\end{align*}
along a non-relabelled subsequence.  Then, one can find an index $\delta_3 \in (0,1)$ such that for all $\delta < \min(\delta_3, \delta_*) =: \delta_4$, where $\delta_*$ is defined after \eqref{APRI_EQ_19},
\begin{align*}
\sup_{t \in (0,T)} |(\varphi_\delta(t) - \varphi(t))_\Omega| \leq \tfrac{1}{2} \sup_{t \in (0,T)} (1 - |(\varphi(t))_\Omega|).
\end{align*}
Since $|(\varphi(t))_\Omega|<1$ for all $t\in (0,T)$ and $(\varphi)_\Omega$ is continuous on $[0,T]$, the prefactor on the left-hand side of \eqref{APRI_EQ_31} is bounded away from 0 uniformly in $t$.  As the right-hand side of \eqref{APRI_EQ_31} is uniformly bounded in $L^2(0,T)$ by previously established uniform estimates, we obtain $\{(\mu_\delta)_\Omega\}_{\delta \in (0,\delta_4)}$ is bounded in $L^2(0,T)$, and the Poincar\'e inequality gives
\begin{align}
\label{APRI_EQ_33}\norm{\mu_{\delta}}_{L^2(0,T;L^2)}\leq C.
\end{align}
Let us mention that if instead of \eqref{beta:log} we employ the less refined estimate \eqref{CFM:est}, we arrive at 
\begin{align*}
\varphi_\delta \beta_\delta(\varphi_\delta) \geq \min(1, \tilde c_0)|\beta_\delta(\varphi_\delta)| - \tilde c_1
\end{align*}
and ultimately
\begin{align*}
\left ( 1 - \frac{|(\varphi(t))_\Omega|}{\min(1,\tilde c_0)} \right ) \left | \intO \mu_\delta(t) \dx \right | \leq C.
\end{align*}
Since $\tilde c_0$ is not quantified, we may not be able to rule out the situation where $\tilde c_0 < 1$, which may imply that the prefactor $1 - \tfrac{|(\varphi(t))_\Omega|}{\min(1,\tilde c_0)}$ is negative.

The uniform estimate \eqref{APRI_EQ_33} for $\mu_\delta$ allows us to infer further estimates for $\beta_{\delta}(\varphi_\delta)$ and $p_\delta$.   Indeed, testing \eqref{WFORM_1c_new} with $\beta_{\delta}(\varphi_{\delta})$ yields
\begin{align*}
\norml{2}{\beta_{\delta}(\varphi_{\delta})}^2 + \intO \beta_{\delta}'(\varphi_{\delta})|\grad\varphi_{\delta}|^2\dx = \intO (\varphi_{\delta}+\mu_{\delta}+\chi\sigma_{\delta}) \beta_{\delta}(\varphi_{\delta})\dx.
\end{align*}
Integrating this identity in time from $0$ to $T$, using the nonnegativity of $\beta_{\delta}'(\cdot)$, \eqref{APRI_EQ_20} and \eqref{APRI_EQ_33}, it follows that
\begin{align}\label{APRI_EQ_34}
\norm{\beta_{\delta}(\varphi_\delta)}_{L^2(0,T;L^2)} \leq C.
\end{align}
For the pressure $p_\delta$, we invoke Lemma \ref{LEM_DIVEQU} to deduce the existence of $\q_\delta := \DD(p_\delta) \in \H^1$ 
such that for a positive constant $C$ depending only on $\Omega$,
\begin{align}\label{APRI_EQ_35}
\norm{\q_\delta}_{\H^1} \leq C \norm{p_\delta}_{L^2}.
\end{align}
Then, testing \eqref{WFORM_1a_new} with $\boldsymbol{\Phi} = \q_\delta$ yields
\begin{align*}
\norm{p_\delta}_{L^2}^2 & \leq  2  \sqrt{\eta_1} \norm{\sqrt{\eta(\varphi_\delta)} \D \v_\delta}_{\LP^2} \norm{\D \q_\delta}_{\LP^2} + \lambda_1 \norm{\Gamma_{\v}(\varphi_\delta, \sigma_\delta)}_{L^2} \norm{p_\delta}_{L^2} \\
& \quad + \nu \norm{\v_\delta}_{\LP^2} \norm{\q_\delta}_{\LP^2} + \norm{(\mu_\delta + \chi \sigma_\delta)}_{L^3} \norm{\nabla \varphi_\delta}_{\LP^2} \norm{\q_\delta}_{\LP^6}.
\end{align*}
Applying Young's inequality and using the uniform estimates \eqref{APRI_EQ_4}, \eqref{APRI_EQ_20}, \eqref{APRI_EQ_33} and \eqref{APRI_EQ_35} leads to
\begin{align}\label{APRI_EQ_36}
\norm{p_\delta}_{L^2(0,T;L^2)} \leq C.
\end{align}

\subsection{Passing to the limit}
Let us first consider the double obstacle case.  In addition to the convergence statements in \eqref{conv}, we further obtain
\begin{alignat*}{3}
\mu_{\delta} &\to\mu &&\quad\text{weakly}&&\quad\text{in }L^2(0,T;H^1),\\
\beta_{\Do,\delta}(\varphi_{\delta})&\to\tau&&\quad \text{weakly}&&\quad\text{in }L^2(0,T;L^2), \\
p_{\delta} &\to p &&\quad \text{weakly}&&\quad\text{in }L^2(0,T;L^2),
\end{alignat*}
for some limit function $\tau\in L^2(0,T;L^2)$.  Moreover, due to \eqref{APRI_EQ_33} we have $\boldsymbol{\xi} = \grad\mu$, which allows us to fully recover \eqref{WFORM_1b} in the limit.  To obtain \eqref{WFORM_1a} and \eqref{WFORM_1d} in the limit we refer the reader to similar arguments as outlined in \cite{EG_jde}, while the divergence equation \eqref{MEQ_1} can be obtained by similar arguments in \cite{EG2}.  It remains to show \eqref{subdiff} is recovered in the limit $\delta \to 0$ from \eqref{WFORM_1c_new}.  By arguing as in \cite[Sec.~5.2]{GLdcdc}, using the weak convergence $\beta_{\Do,\delta}(\varphi_\delta) \rightharpoonup \tau$ in $L^2(0,T;L^2)$, the strong convergence $\varphi_\delta \to \varphi$ in $L^2(0,T;L^2)$, and the maximal monotonicity of the subdifferential $\del \mathbb{I}_{[-1,1]}$ we can infer that $\tau$ is an element of the set $\del \mathbb{I}_{[-1,1]}(\varphi)$, which implies that for any $\zeta \in \KK$ and a.e.~$t \in (0,T)$,
\begin{align*}
\intO \tau(t) (\zeta - \varphi(t)) \dx \leq 0.
\end{align*}
Hence, testing \eqref{WFORM_1c_new} (where $\psi_\delta = \psi_{\Do,\delta}$) with $\zeta - \varphi$ and passing to the limit $\delta \to 0$ allows us to recover \eqref{subdiff}.  This completes the proof of Theorem \ref{thm:timedep} for the double obstacle potential.

For the logarithmic case, the uniform estimate for $\beta_{\log,\delta}(\varphi_\delta)$ in $L^2(0,T;L^2)$ allow us to infer, by the arguments in \cite[Sec.~4]{CFM} or \cite[Sec.~3.3]{GGW}, that the limit $\varphi$ satisfies the tighter bounds 
\begin{align*}
|\varphi(x,t)| < 1 \quad\text{a.e.~in } \Omega_T.
\end{align*}
Furthermore, by the a.e.~convergence of $\varphi_\delta$ to $\varphi$ we have $\beta_{\log,\delta}(\varphi_\delta) \to \beta_{\log}(\varphi)$ a.e.~in $\Omega_T$.   Hence, in the limit $\delta \to 0$, we can recover \eqref{log:weak}.
Meanwhile, the inequality \eqref{log:energy} is obtained from integrating \eqref{APRI_EQ_8} over $(0,t)$ for $t \in (0,T)$ and then passing to the limit with the compactness assertions \eqref{conv}, weak lower semicontinuity, and Fatou's lemma.  This completes the proof of Theorem \ref{thm:timedep} for the logarithmic potential.


\section{Proof of Theorem \ref{thm:stat} -- Stationary solutions}\label{sec:stat}
As with the time-dependent case, we extend $b_\v, b_\varphi, f_\v$ and $f_\varphi$ from $[-1,1]$ to $\R$ such that $f_{\varphi} \in C^0(\R) \cap L^{\infty}(\R)$, $f_{\v} \in C^1(\R)\cap W^{1,\infty}(\R)$, $b_{\varphi} \in C^0(\R) \cap L^{\infty}(\R)$ is nonnegative, $b_{\v} \in C^1(\R) \cap W^{1,\infty}(\R)$ is nonnegative, and fulfill \eqref{ASS_SOURCE_1}, \eqref{ASS_SOURCE_2} and \eqref{ass:log:f}.   

\subsection{Approximation scheme}
We consider a smooth function $\hat g :\R \to [0,1]$ such that $\hat g(r) = 1$ for $r \geq 3$ and $\hat g(r) = 0$ for $r \leq 2$, and define $F: L^2(\Omega) \to \R$ as
\begin{align*}
F(v) := C_F \hat g \big ( \tfrac{1}{\abs{\Omega}} \norml{2}{v}^2 \big ) \quad \text{ for } v \in L^2(\Omega),
\end{align*}
where $C_F$ is a positive constant to be specified later.  Furthermore, we denote by $\gamma(r,s)$ the function
\begin{align*}
\gamma(r,s) :=  r \Gamma_{\v}(r,s) - \Gamma_{\varphi}(r,s),
\end{align*}
and introduce, for $\delta \in (0,1)$, the regular cutoff operator $\Td \in C^{1,1}(\R)$ defined as
\begin{align}
\Td(s) = \begin{cases}
1-\tfrac{3 \delta}{4} & \text{ if } s \geq 1 - \delta, \\
1 - \tfrac{3\delta}{4} + \tfrac{1}{\delta}(s - (1-\delta))^2 & \text{ if } 1-\tfrac{\delta}{2} \leq s \leq 1 - \delta, \\
s & \text{ if } \abs{s} \leq 1 - \tfrac{\delta}{2}, \\
\tfrac{3 \delta}{4} -1 - \tfrac{1}{\delta} ((\delta - 1) - s)^2 & \text{ if } -1 +\delta \leq s \leq -1 + \tfrac{\delta}{2}, \\
\tfrac{3 \delta}{4} - 1 & \text{ if } s \leq -1 + \delta.
\end{cases}
\end{align}
It is clear that $\Td$ and its Lipschitz continuous derivative $\Td'$ are bounded independently of $\delta$.  Then, we seek a solution $\varphi_\delta \in W := H^2_n$ to the approximate system
\begin{align}\label{stat:approx}
\begin{cases}
 \sqrt{\delta} \beta_\delta(\varphi_\delta) + F(\varphi_\delta) \varphi_\delta +  \Delta (\Delta \varphi_\delta - \psi_\delta'(\varphi_\delta) + \chi \sigma_{\delta}) & \\
 \quad  = - \gamma(\varphi_\delta, \sigma_{\delta}) - \v_{\delta} \cdot \nabla \Td(\varphi_\delta) & \text{ in } \Omega, \\
\deln \varphi_\delta = \deln(\Delta \varphi_\delta + \chi \sigma_{\varphi_\delta}) = 0 & \text{ on } \Sigma,
\end{cases}
\end{align}
where $\sigma_{\delta} \in H^2$ is the unique nonnegative and bounded solution to the nutrient subsystem
\begin{align}\label{stat:nut}
\begin{cases}
0 =  \lap \sigma_{\delta} - h(\varphi_\delta) \sigma_{\delta} & \text{ in } \Omega, \\
\deln \sigma_{\delta} = K(1-\sigma_{\delta}) & \text{ on } \Sigma,
\end{cases}
\end{align}
and $\v_{\delta}$ is the first component of the unique solution $(\v_{\delta}, p_{\delta})$ to the Brinkman subsystem 
\begin{align}\label{stat:Brink}
\begin{cases}
- \div ( \T(\v_{\delta}, p_{\delta})) + \nu \v_{\delta}  = (\psi_\delta'(\varphi_\delta) - \lap \varphi_\delta) \nabla \Td(\varphi_\delta) & \text{ in } \Omega, \\
\div(\v_{\delta}) = \Gamma_{\v}(\varphi_\delta, \sigma_{\delta}) & \text{ in } \Omega, \\
\T(\v_{\delta}, p_{\delta}) \n = (2 \eta(\varphi_\delta) \D \v_{\delta} + \lambda(\varphi_\delta) \div(\v_{\delta}) \I - p_{\delta} \I)\n = \0  &\text{ on } \Sigma.
\end{cases}
\end{align}

We aim to use pseudomonotone operator theory, akin to the methodology used in \cite{GLS}, to deduce the existence of at least one solution $\varphi_\delta \in W$ to \eqref{stat:approx} for each $\delta \in (0,1)$.  Then, we derive enough uniform estimates to pass to the limit $\delta \to 0$ in order to prove Theorem \ref{thm:stat}.  The new element in the analysis is how we treat the convection term.


\subsection{Preparatory results}
For fixed $u \in W := H^2_n$ it is clear that there is a unique solution $\sigma_u \in H^2$ to the nutrient subsystem \eqref{stat:nut}, and
\begin{align}\label{stat:s:cts}
\norm{\sigma_{u_1} - \sigma_{u_2}}_{H^2} \leq C \norm{u_1 - u_2}_{L^2}
\end{align}
for any pair $u_1,u_2 \in W$ with corresponding solutions $\sigma_{u_1}, \sigma_{u_2}$.  For the Brinkman subsystem \eqref{stat:Brink} we have the following.

\begin{lem}\label{lem:stat:Brink}
For fixed $u \in W$ and $\delta \in (0,1)$, there exists a unique strong solution $(\v_u, p_u) \in \W^{2,\frac{5}{4}} \times W^{1,\frac{5}{4}}$ to the Brinkman subsystem \eqref{stat:Brink} with $\varphi_\delta$ replaced by $u$.
\end{lem}
Let us mention Lemma \ref{lem:Brink} gives the strong existence but not the uniqueness assertion, since in \eqref{BM_SUBSY} the phase field variable $c$ is fixed and not related to the data $\f$, $g$ and $\bh$.  Hence, Lemma \ref{lem:Brink} gives uniqueness of $(\v,p)$ only with respect to $(\f, g, \bh)$.

\begin{proof}
For the existence part, we invoke Lemma \ref{lem:Brink} with $u = c \in W$, $g = \Gamma_{\v}(u, \sigma_u) \in H^1$, $\f = (\psi_\delta(u) - \lap u) \nabla \Td(u) \in \LP^{\frac{5}{4}}$, and $\bh = \0$.  

Let $(\v_1, p_1), (\v_2, p_2) \in \W^{2,\frac{5}{4}} \times W^{1,\frac{5}{4}}$ be the solutions to \eqref{stat:Brink} corresponding to data $u_1, u_2 \in W$, respectively.  For convenience we write $\hat \v := \v_1 - \v_2$, $\hat p = p_1 - p_2$, $\hat u = u_1 - u_2$, $\hat \eta = \eta(u_1) - \eta(u_2)$, $\hat{\psi}'_\delta = \psi_\delta'(u_1) - \psi_\delta'(u_2)$, $\hat{\Td'} = \Td'(u_1) - \Td'(u_2)$ and $\hat{\Gamma}_\v := \Gamma_\v(u_1, \sigma_{u_1}) - \Gamma_{\v}(u_2, \sigma_{u_2})$.  Let $\hat \w:= \DD(\hat{\Gamma}_\v) \in \H^1$ 
where by \eqref{stat:s:cts} it holds that 
\begin{align}\label{stat:div:cts}
\normH{1}{\w} \leq C \norml{2}{\hat \Gamma_{\v}} \leq C \norml{2}{\hat u}.
\end{align}
Then, testing the difference of $\eqref{stat:Brink}_1$ with $\hat \v - \hat \w$ yields
\begin{align*}
& \int_\Omega \nu \abs{\hat \v}^2 - \nu \hat \v \cdot \hat \w + 2\eta(u_1) (\abs{\D \hat \v}^2 - \D \hat \v : \D \hat \w) + 2 \hat \eta \D \v_{2} : \D (\hat \v - \hat \w) \dx \\
& \quad = \int_\Omega (\hat \v - \hat \w) \cdot \Big ( (\hat{\psi}_\delta' - \lap \hat u ) \nabla \Td(u_1) + (\psi_\delta'(u_2) - \lap u_2) (\Td'(u_1) \nabla \hat u + \hat{\Td'}\nabla u_2 ) \Big ) \dx.
\end{align*}
The left-hand side can be bounded below by
\begin{align}\label{stat:B:cts:L}
\text{LHS} \geq \frac{\nu}{2} \normL{2}{\hat \v}^2 + \frac{\eta_0}{2} \normL{2}{\D \hat \v}^2 - \frac{\nu}{2} \normL{2}{\hat \w}^2 -C \normL{2}{\D \hat \w}^2 - C \normL{2}{\D \v_2}^2 \norml{\infty}{\hat u}^2,
\end{align}
while the right-hand side can be bounded above by
\begin{align*}
\text{RHS}  &\leq C\normH{1}{\hat \v - \hat \w} \normL{3}{\nabla u_1}\big ( \norml{2}{\lap \hat u} + \norml{2}{\hat \psi_\delta'} \big ) \\
& \quad + C \normH{1}{\hat \v - \hat \w}  \big ( \norml{2}{\lap u_2} + \norml{2}{\psi_\delta'(u_2)} \big ) \big (\normL{3}{\nabla \hat u} + \normL{3}{\hat{\Td'} \nabla u_2} \big ).
\end{align*}
In light of the regularities $u_1, u_2 \in W$ and $\v_1, \v_2 \in \W^{2,\frac{5}{4}}$, the polynomial growth of $\psi_\delta$ leading to the following difference estimate
\begin{align*}
\abs{\psi_\delta'(s) - \psi_\delta'(t)} \leq C_\delta (1 + \abs{s} + \abs{t} ) \abs{s-t} \quad \forall s, t \in \R,
\end{align*}
the Lipschitz continuity and boundedness of $\Td'$, as well as \eqref{stat:div:cts}, we find that 
\begin{align*}
\text{RHS} \leq C \normH{1}{\hat \v -\hat \w} \big ( \normh{2}{\hat u} + \normL{6}{\nabla u_2} \norml{6}{\hat{\Td'}} \big ) \leq C \normh{2}{\hat u}^2 + \eps \normH{1}{\hat v}^2
\end{align*}
for some small constant $\eps > 0$.  Invoking Korn's inequality and \eqref{stat:div:cts} in \eqref{stat:B:cts:L} and combining with the above estimate gives
\begin{align*}
\normH{1}{\hat \v}^2 \leq C \normh{2}{\hat u}^2.
\end{align*}
This yields uniqueness of the mapping $u \mapsto \v_u$.  For the pressure, we test the difference of $\eqref{stat:Brink}_1$ with $\hat{\q} := \DD(\hat p) \in \H^1$ which satisfies 
\begin{align*}
\normH{1}{\hat \q} \leq C \norml{2}{\hat p},
\end{align*}
and in turn we obtain
\begin{align*}
\norml{2}{\hat p}^2 \leq C \big ( \normH{1}{\hat \v} + \normh{2}{\hat u} \big ) \normH{1}{\hat \q} \leq \tfrac{1}{2} \norml{2}{\hat p}^2 + C \normh{2}{\hat u}^2,
\end{align*}
which yields the uniqueness of the mapping $u \mapsto p_u$.
\end{proof}

The unique solvability of the nutrient and Brinkman subsystems in turn yields the following result.

\begin{lem}\label{lem:calA}
For each $\delta \in (0,1)$, the operator $\mathcal{A} : W \to W^*$ defined as
\begin{align*}
\inn{\mathcal{A}u}{\zeta}_{W} := \intO \big ( \gamma(u, \sigma_u) + \v_u \cdot \nabla \Td(u) + \chi \lap \sigma_u \big ) \zeta  \dx
\end{align*}
is strongly continuous, i.e., if $u_n \rightharpoonup u$ in $W$, then $\mathcal{A}u_n \to \mathcal{A}u$ in $W^*$.
\end{lem}
\begin{proof}
Let $\{u_n\}_{n \in \N} \subset W$ be a sequence of functions such that $u_n \rightharpoonup u$ in $W$.  Denoting by $\sigma_n$ and $(\v_n,p_n)$ the corresponding unique solutions to the nutrient subsystem \eqref{stat:nut} and Brinkman subsystem \eqref{stat:Brink} where $\varphi_\delta = u_n$. Then, we easily infer that 
\begin{align*}
\norm{\sigma_n}_{H^2} \leq C, \quad \sigma_n \in [0,1] & \text{ a.e.~in } \Omega,
\end{align*}
for a positive constant $C$ independent of $n$. Furthermore, using the assumptions on $\Gamma_{\v}$ by Lemma \ref{LEM_DIVEQU}, there exists a function $\u_n := \DD(\Gamma_\v(u_n, \sigma_n)) \in \H^1$ satisfying
\begin{equation*}
\normH{1}{\mathbf{u}_n}\leq C\norml{2}{\Gamma_{\v}(u_n,\sigma_n)}\leq C,
\end{equation*}
with $C$ independent of $\delta$. Testing $\eqref{stat:Brink}_1$ (for $\varphi_{\delta} = u_n$) with $\v_n-\mathbf{u}_n$, using $\eqref{stat:Brink}_2$-$\eqref{stat:Brink}_3$ and Korn's inequality, we obtain
\begin{align*}
\normH{1}{\v_n}&\leq C\left(1+\normL{\frac{6}{5}}{(\psi_{\delta}'(u_n)-\lap u_n)\grad\Td(u_n)}\right)\\
&\leq C\left(1+ \normL{2}{\grad u_n}\norml{3}{\psi_{\delta}'(u_n)}+\norml{2}{\lap u_n}\normL{3}{\grad u_n}\right).
\end{align*}
Similarly, let $\mathbf{q}_n := \DD(p_n) \in \H^1$ 
which satisfies
\begin{equation*}
\normH{1}{\mathbf{q}_n}\leq C\norml{2}{p_n}.
\end{equation*}
Testing $\eqref{stat:Brink}_1$ (for $\varphi_{\delta} = u_n$) with $\q_n$ and using  $\eqref{stat:Brink}_2$-$\eqref{stat:Brink}_3$, we obtain
\begin{equation*}
\norml{2}{p_n}\leq C\left(1+ \normH{1}{\v_n}+ \normL{2}{\grad u_n}\norml{3}{\psi_{\delta}'(u_n)}+\norml{2}{\lap u_n}\normL{3}{\grad u_n}\right)
\end{equation*}
In particular, this shows that
\begin{equation}\label{stat:B:1}
\normH{1}{\v_n} + \norml{2}{p_n} 
\leq C\left(1+ \normL{2}{\grad u_n}\norml{3}{\psi_{\delta}'(u_n)}+\norml{2}{\lap u_n}\normL{3}{\grad u_n}\right)  \leq C_\delta,
\end{equation}
for a positive constant $C_\delta$ depending on $\delta$ but independent of $n$, thanks to the boundedness of $\{u_n\}_{n \in \N}$ in $W$.  Hence, for fixed $\delta \in (0,1)$, there exist functions $\sigma \in H^2$, $\v \in \H^1$ and $p \in L^2$, such that along a nonrelabelled subsequence, 
\begin{align*}
\sigma_n \rightharpoonup \sigma_u \text{ in } H^2, \quad \v_n \rightharpoonup \v_u \text{ in } \H^1, \quad p_n \rightharpoonup p_u \text{ in } L^2.
\end{align*}
It is clear that $\sigma_u$ is the unique solution to \eqref{stat:nut} corresponding to $u$, while $(\v_u, p_u)$ is the unique weak solution to \eqref{stat:Brink} corresponding to $u$.  Indeed, for the highest order term in \eqref{stat:Brink}, we employ the dominated convergence theorem to deduce that $\Td'(u_n) \to \Td'(u)$ strongly in $L^{\frac{15}{2}}$.  Together with $\lap u_n \rightharpoonup \lap u$ in $L^2$ and $\grad u_n \to \grad u$ in $\LP^5$, we obtain 
\begin{align*}
\lap u_n \grad \Td(u_n) \rightharpoonup \lap u \grad \Td(u) \text{ in } \LP^{6/5},
\end{align*}
while the other terms in the Brinkman system can be handled using similar compactness assertions.

By Rellich's theorem, it is easy to see that
\begin{align*}
\intO  (\gamma(u_n, \sigma_n) + \chi \lap \sigma_n) \zeta \dx \to \intO ( \gamma(u, \sigma_u) + \chi \lap \sigma_{u}) \zeta \dx 
\end{align*}
for all $\zeta \in W$.  Meanwhile, integrating by parts and using $\div (\v_n) = \Gamma_{\v}(u_n, \sigma_n) =: \Gamma_{\v_n}$ yields
\begin{align*}
\intO \zeta ( \nabla \Td(u_n) \cdot \v_{n} ) \dx & = \intS \Td(u_n) \zeta \v_{n} \cdot \n \dH - \intO \Td(u_n) \big ( \Gamma_{\v_n} \zeta + \v_{n} \cdot \nabla \zeta \big ) \dx \\
& \to  \intS \Td(u) \zeta \v_{u} \cdot \n \dH - \intO \Td(u) \big ( \Gamma_{\v_u} \zeta + \v_u \cdot \nabla \zeta \big ) \dx \\
& = \intO \zeta (\nabla \Td(u) \cdot \v_u) \dx
\end{align*}
on account of $\Td(u_n) \to \Td(u)$ in $L^2$ and $L^2(\Sigma)$, $\v_n \to \v_u$ in $\LP^3$ and $\LP^2(\Sigma)$ (see e.g. \cite[Lem. 1.3]{EG_jde}), and $\Gamma_{\v_n} \to \Gamma_{\v_u}$ in $L^2$ .  This shows that $\mathcal{A}$ is strongly continuous.
\end{proof}

\subsection{Existence of approximate solutions}\label{sec:stat:approxsoln}
In this section we fix $\delta \in (0,1)$, and define operators $A_1, A_2: W \to W^*$ by
\begin{align*}
\inn{A_1 u}{\zeta}_{W} & := \int_\Omega \sqrt{\delta} \beta_\delta(u) \zeta + \Delta u \Delta \zeta \dx \\
 \inn{A_2 u}{\zeta}_W & := \int_\Omega \big ( F(u) u + \gamma(u,\sigma_u) + \v_u \cdot \nabla \Td(u) + \chi \lap \sigma_u \big ) \zeta + \nabla \psi_\delta'(u) \cdot \nabla \zeta \dx.
 \end{align*}
Then, $\varphi_\delta \in W$ is a weak solution to \eqref{stat:approx} if $\inn{(A_1 + A_2)\varphi_\delta}{\zeta}_W = 0$ for all $\zeta \in W$.  

Since $\beta_\delta'$ is bounded and $\beta_\delta$ has sublinear growth, we deduce that the operator $A_1$ is monotone and hemicontinuous.  On the other hand, Lemma \ref{lem:calA} together with the continuity and sublinear growth of $\psi_\delta'$, and the continuity and boundedness of $F$ imply that $A_2$ is strongly continuous.  Then, by \cite[Thm.~27.6]{Zeidler} the sum $A = A_1 + A_2$ is a pseudomonotone operator.

We now claim that $A$ is additionally coercive over $W$, i.e., 
\begin{align*}
\lim_{\norm{u}_W \to + \infty} \frac{\inn{Au}{u}}{\norm{u}_{W}} = + \infty.
\end{align*}
Let us first treat the velocity term in $\inn{Au}{u}_W$.  By the definition of $\Td$ and the trace theorem we see that
\begin{equation}\label{stat:coer:1}
\begin{aligned}
& \left | \intO u \big (\v_{u} \cdot \nabla \Td(u) \big ) \dx \right |=  \left |\intO \tfrac{1}{2} \v_{u} \cdot \nabla \abs{\Td(u)}^2 \dx  \right |\\
& \quad = \left | \intS \tfrac{1}{2} \abs{\Td(u)}^2 \v_{u} \cdot \n \dH - \intO \tfrac{1}{2} \Gamma_{\v}(u, \sigma_u) \abs{\Td(u)}^2 \dx \right | \\
& \quad \leq  \tfrac{1}{2} \norm{\v_u}_{\LP^1(\Sigma)} + C \leq C_* \big ( 1 + \norm{\v_u}_{\H^1} \big ).
\end{aligned}
\end{equation}
Let $\u_u := \DD(\Gamma_{\v_u}(u, \sigma_u)) \in \H^1$ which satisfies
\begin{align*}
\norm{\u_u}_{\H^1} \leq C \norm{\Gamma_{\v}(u, \sigma_u)}_{L^2} \leq C,
\end{align*}
where the boundedness of $\Gamma_{\v}(u,\sigma_u)$ can be deduced using similar arguments in Section \ref{sec:unif:do}.  Testing \eqref{stat:Brink} for $(\v_u, p_u)$ with $\v_u - \u_u$ then yields
\begin{align*}
 \eta_0 \norm{\D \v_u}_{\LP^2}^2 + \tfrac{\nu}{2} \norm{\v_u}_{\LP^2}^2 & \leq C \norm{\u_u}_{\H^1}^2 + \norm{(\psi_\delta'(u) - \lap u) \nabla \Td(u) \cdot (\v_u - \u)}_{L^1} \\
 & \leq C + \norm{(\psi_\delta'(u) - \lap u) \nabla \Td(u)}_{\LP^{5/4}}  \norm{\v_u-\u}_{\LP^5}.
\end{align*}
By Korn's inequality, H\"older's inequality and Young's inequality we obtain
\begin{align*}
\norm{\v_u}_{\H^1} & \leq C \Big ( \norm{\u_u}_{\H^1} + \norm{\psi_\delta'(u) \nabla \Td(u)}_{\LP^{5/4}} + \norm{\lap u \nabla \Td(u)}_{\LP^{5/4}} \Big ) \\
& \leq C \Big ( 1 + \norm{\lap u}_{L^2} \norm{\nabla u}_{\LP^{10/3}} \Big ) + \frac{1}{C_*} \Big ( 1+ \norm{\psi_\delta'(u) \nabla T_\delta(u)}_{\LP^{5/4}}^{\frac{5}{4}} \Big ),
\end{align*}
where $C_*$ is the constant in \eqref{stat:coer:1}.  Let us note that as a consequence of elliptic regularity and integration by parts, for $u \in W$ we obtain the following useful inequalities:
\begin{align}\label{ell:est}
\norm{\nabla u}_{\LP^2} \leq \norm{\lap u}_{L^2}^{1/2} \norm{u}_{L^2}^{1/2}, \quad \norm{u}_{H^2} \leq C_\Omega \big ( \norm{\lap u}_{L^2} + \norm{u}_{L^2} \big ).
\end{align}
Employing the Gagliardo--Nirenberg inequality and the elliptic estimate \eqref{ell:est}, we have
\begin{align*}
C\norm{\lap u}_{L^2} \norm{\nabla u}_{\LP^{10/3}} & \leq C \norm{\lap u}_{L^2} \big ( \norm{u}_{H^2}^{4/5} \norm{u}_{L^2}^{1/5} + \norm{u}_{L^2} \big ) \leq \frac{1}{8 C_*} \norm{\lap u}_{L^2}^2 + C \norm{u}_{L^2}^2.
\end{align*}
Meanwhile, we observe that
\begin{align*}
\Td'(s) \beta_\delta(s) =  \begin{cases}
\Td'(s) \beta_{\log}(s) & \text{ for the logarithmic potential}, \\
0 & \text{ for the obstacle potential},
\end{cases}
\end{align*} 
since $\beta_{\Do,\delta}(s) = 0$, $\beta_{\log,\delta}(s) = \beta_{\log}(s)$ for $s \in [-1+\delta, 1-\delta]$.  Hence, as $0 \leq \Td'(s) \leq 1$, we see that
\begin{equation}\label{stat:coer:psi}
\begin{aligned}
& \int_\Omega \abs{\psi_\delta'(u) T_\delta'(u) \nabla u}^{\frac{5}{4}} \dx = \int_{\{ \abs{u} \leq 1-\delta \}} \abs{(\beta_{\log}(u) - \Theta_c u) \nabla u}^{\frac{5}{4}} \dx \\
& \quad \leq \int_{\{ \abs{u} \leq 1-\delta \}} \frac{|\beta_{\log}(u)|^{\frac{5}{4}}}{|\beta_{\log}'(u)|^{\frac{5}{8}}} \abs{\beta_{\log}'(u)}^{\frac{5}{8}} \abs{\nabla u}^{\frac{5}{4}} \dx + C\big ( 1 + \norm{\nabla u}_{\LP^2}^2 \big )\\
& \quad \leq C \int_\Omega \abs{ \beta_{\log,\delta}'(u)}^{\frac{5}{8}} \abs{\nabla u}^{\frac{5}{4}} \dx + C \big ( 1 + \norm{\nabla u}_{\LP^2}^2 \big )  \\
& \quad \leq \frac{1}{2 C_*} \int_\Omega \beta_{\log,\delta}'(u) \abs{\nabla u}^2 \dx + \frac{1}{8 C_*} \norm{\lap u}_{L^2}^2 + C \big ( 1 + \norm{u}_{L^2}^2 \big ),
\end{aligned}
\end{equation}
where for the second last inequality we used L'Hoptial's rule to confirm that $\frac{\beta_{\log}(s)}{(\beta_{\log}'(s))^{1/2}}$ is continuous in $[-1,1]$, hence bounded.  Returning to \eqref{stat:coer:1}, we deduce that
\begin{align}\label{stat:coer:2}
\left | \intO u \big ( \v_u \cdot \nabla \Td(u) \big ) \dx \right | \leq \frac{3}{8} \norm{\lap u}_{L^2}^2 + \frac{1}{2} \intO \beta_\delta'(u) \abs{\nabla u}^2 + C_1 \big ( 1 + \norm{u}_{L^2}^2 \big )
\end{align}
for a positive constant $C_1$ independent of $u$ and $\delta$.  Now, computing $\inn{Au}{u}_W$ gives
\begin{equation}\label{stat:coer:3}
\begin{aligned}
\inn{Au}{u}_W & = \int_\Omega F(u)\abs{u}^2 + \sqrt{\delta} \beta_\delta(u)u + \abs{\lap u}^2 + \beta_\delta'(u) \abs{\nabla u}^2 - \Theta_c \abs{\nabla u}^2 \dx \\
& \quad + \int_\Omega  u \gamma(u, \sigma_u) + u \big ( \v_{u} \cdot \nabla \Td(u) \big ) + \chi u \lap \sigma_{u} \dx \\
& \geq \int_\Omega F(u)\abs{u}^2 + \sqrt{\delta} \beta_{\delta}(u)u + \tfrac{1}{2} \abs{\lap u}^2 + \tfrac{1}{2} \beta_\delta'(u) \abs{\nabla u}^2 \dx \\
& \quad -C (1 +  \norm{u}_{L^2}^2),
\end{aligned}
\end{equation}
for a positive constant $C$ independent of $u$ and $\delta$.  Recalling that $F(u) = C_F$ for $\norm{u}_{L^2}^2 > 3 \abs{\Omega}$, and so, in choosing  $C_F = 2C$, we arrive at
\begin{align*}
\inn{Au}{u}_W \geq \int_\Omega \sqrt{\delta} \beta_\delta(u)u + \tfrac{1}{2}C_F \abs{u}^2 + \tfrac{1}{2} \abs{\lap u}^2 \dx - C \geq c \norm{u}_{H^2}^2 - C
\end{align*}
for $\norm{u}_{L^2}^2 \geq 3 \abs{\Omega}$, which in turn implies coercivity of $A$.

Invoking \cite[Thm.~27.A]{Zeidler}, we deduce for every $\delta\in (0,1)$ the existence of a weak solution $\varphi_{\delta} \in W$ to \eqref{stat:approx}.  Setting
\begin{align}\label{stat:mudelta}
\mu_\delta = - \lap \varphi_\delta + \psi_\delta'(\varphi_\delta) - \chi \sigma_{\varphi_\delta}
\end{align}
we see that the equation $A \varphi_\delta = 0$ in $W^*$ implies
\begin{align}\label{stat:mu:eq}
\int_\Omega \mu_\delta \lap \zeta \dx = \int_\Omega f_\delta \zeta \dx \quad \forall \zeta \in W,
\end{align} 
with right-hand side
\begin{align*}
f_\delta := \sqrt{\delta} \beta_\delta(\varphi_\delta)  + F(\varphi_\delta) \varphi_\delta - \gamma(\varphi_\delta, \sigma_\delta) -  \nabla \Td(\varphi_\delta) \cdot \v_{\delta}.
\end{align*}
Thanks to the regularity $\varphi_\delta \in W$, $\v_{\delta} \in \H^1$, $\sigma_{\delta} \in H^2$, and the sublinear growth of $\beta_\delta$, we easily infer that $f_\delta \in L^2$.  On the other hand, choosing $\zeta = 1$ in \eqref{stat:mu:eq} implies that $f \in L^2_0$.  Then, by arguing as in \cite[Sec.~3.1]{GLS}, we obtain that $\mu_\delta \in W$ for all $\delta \in (0,1)$.

\subsection{Uniform estimates}
From \eqref{stat:approx} and \eqref{stat:mu:eq}, the pair $(\varphi_\delta, \mu_\delta) \in W \times W$ satisfies
\begin{subequations}\label{stat:w}
\begin{alignat}{2}
0 & = \int_\Omega  \big (F(\varphi_\delta) \varphi_\delta + \sqrt{\delta} \beta_\delta(\varphi_\delta) + \gamma(\varphi_\delta, \sigma_\delta) +  \v_\delta \cdot \nabla T_\delta(\varphi_\delta) - \lap \mu_\delta \big ) \zeta \dx, \label{stat:w:1} \\
0 & = \int_\Omega \big ( \beta_\delta(\varphi_\delta) - \Theta_c \varphi_\delta - \mu_\delta - \chi \sigma_\delta - \lap \varphi_\delta \big ) \zeta \dx \label{stat:w:2}
\end{alignat}
\end{subequations}
for all $\zeta \in L^2$.  Returning to the proof of the coercivity of the operator $A$, replacing $u$ with $\varphi_\delta$ in \eqref{stat:coer:3} gives
\begin{align}\label{stat:est:0}
\int_\Omega F(\varphi_\delta) |\varphi_\delta|^2 + \sqrt{\delta} \beta_\delta(\varphi_\delta) \varphi_\delta + \tfrac{1}{2} \beta_\delta'(\varphi_\delta) \abs{\nabla \varphi_\delta}^2 + \tfrac{1}{2} \abs{\lap \varphi_\delta}^2 \leq \tfrac{1}{2} C_F \norm{\varphi_\delta}_{L^2}^2 + C.
\end{align}
If $\norm{\varphi_\delta}_{L^2}^2 \geq 3 \abs{\Omega}$, then as before we have
\begin{align}\label{stat:est:1}
\int_\Omega \sqrt{\delta} \beta_\delta(\varphi_\delta) \varphi_\delta +  \beta_\delta'(\varphi_\delta) \abs{\nabla \varphi_\delta}^2  \dx + \norm{\varphi_\delta}_{H^2}^2 \leq C.
\end{align}
If $\norm{\varphi_\delta}_{L^2}^2 < 3 \abs{\Omega}$, then adding $\norm{\varphi_\delta}_{L^2}^2$ to both sides of \eqref{stat:est:0} and neglecting the nonnegative term $F(\varphi_\delta) |\varphi_\delta|^2$ on the left-hand side yields the uniform estimate \eqref{stat:est:1}.  Hence, $\{\varphi_\delta\}_{\delta \in (0,1)}$ is bounded in $W$, and along a non-relabelled subsequence, it holds that
\begin{align*}
\varphi_\delta \rightharpoonup \varphi, \quad \sigma_\delta \rightharpoonup \sigma_\varphi \text{ in } H^2,
\end{align*}
where $\sigma_\varphi$ is the unique solution to the nutrient subsystem \eqref{stat:nut} with data $\varphi$.  

Convexity of $\hat \beta_\delta$ and $\hat \beta_\delta(0) = 0$ imply the inequality 
\begin{align*}
\hat \beta_\delta (s) \leq \beta_\delta(s) s \quad \text{ for all }s \in \R,
\end{align*}
and together with \eqref{PROP_PSIDELTA_3}, \eqref{PROP_PSIDELTA_LOG_1a} and \eqref{stat:est:1} we deduce that
\begin{align*}
 \delta^{2} \int_\Omega \abs{\beta_{\Do,\delta}(\varphi_\delta)}^2 \dx \leq 2 \delta \int_\Omega \hat \beta_{\Do,\delta}(\varphi_\delta) \dx \leq 2 \delta \int_\Omega  \beta_{\Do,\delta} (\varphi_\delta)\varphi_{\delta} \dx \leq C \sqrt{\delta}, \\
\int_\Omega (\abs{\varphi_\delta}-1)_{+}^2 \dx \leq \int_\Omega \frac{4}{\theta} \delta \hat \beta_{\log,\delta}(\varphi_\delta) \dx \leq C\delta  \int_\Omega \beta_{\log,\delta}(\varphi_\delta) \varphi_\delta \dx \leq C \sqrt{\delta}.
\end{align*}
Using \eqref{del:beta} for the double obstacle potential, we deduce that for both cases the limit $\varphi$ satisfies
\begin{align}\label{stat:phi:1}
\abs{\varphi} \leq 1 \text{ a.e.~in } \Omega.
\end{align}
In particular, we have
\begin{align}\label{Td:conv}
\norm{\varphi}_{L^2}^2 \leq \abs{\Omega}, \quad \Td(\varphi_\delta) \to \varphi \text{ in } L^p \text{ for any } p < \infty.
\end{align}
The latter can be deduced from the a.e.~convergence $\Td(\varphi) \to \varphi$, the Lipschitz continuity of $\Td$ and hence a.e.~convergence of $\Td(\varphi_\delta) - \Td(\varphi) \to 0$, and the dominating convergece theorem.  Using the norm convergence $\norm{\varphi_\delta}_{L^2}^2 \to \norm{\varphi}_{L^2}^2$, we then infer the existence of $\delta_5 > 0$ such that $\norm{\varphi_\delta}_{L^2}^2 \leq 2 \abs{\Omega}$ for $\delta \in (0,\delta_5)$.  Subsequently, $F(\varphi_\delta) = 0$ for $\delta \in (0,\delta_5)$, and in the sequel we will neglect the term $F(\varphi_\delta)\varphi_\delta$.

Choosing $\zeta = - \lap \mu_\delta$ in \eqref{stat:w:2}, and choosing $\zeta = \beta_\delta(\varphi_\delta)$ and also $\zeta = - \lap \varphi_\delta$ in \eqref{stat:w:1} yields after summation and integrating by parts
\begin{equation}\label{stat:mu:1}
\begin{aligned}
& \norm{\nabla \mu_\delta}_{L^2}^2 + \sqrt{\delta} \norm{\beta_\delta(\varphi_\delta)}_{L^2}^2 \\
& \quad \leq \int_\Omega \gamma(\varphi_\delta, \sigma_\delta) (\lap \varphi_\delta - \beta_{\delta}(\varphi_\delta)) - \nabla ( \Theta_c \varphi_\delta + \chi \sigma_\delta) \cdot \nabla \mu_\delta  \dx \\
& \qquad + \int_\Omega  \v_{\delta} \cdot \nabla \Td(\varphi_\delta) ( \lap \varphi_\delta -  \beta_\delta(\varphi_\delta) ) \dx \\
& \quad \leq C + \frac{1}{2} \norm{\nabla \mu_\delta}_{L^2}^2 + \int_\Omega \v_\delta \cdot \nabla T_\delta(\varphi_\delta) ( \lap \varphi_\delta - \beta_\delta(\varphi_\delta)) \dx
\end{aligned}
\end{equation}
on account of the boundedness of $\varphi_\delta$ and $\sigma_\delta$ in $H^2$, and the estimates \eqref{source:doleq0} and \eqref{log:source:est} for the term $\gamma(\varphi_\delta, \sigma_\delta) \beta_\delta(\varphi_\delta)$.  Let $\u_\delta := \DD(\Gamma_{\v_\delta}) \in \H^1$ which satisfies
\begin{align*}
\norm{\u_\delta}_{\H^1} \leq C \norm{\Gamma_{\v_\delta}}_{L^2} \leq C.
\end{align*}
Then, testing $\v_\delta - \u_\delta$ with the Brinkman subsystem \eqref{stat:Brink} yields
\begin{equation}\label{stat:mu:2}
\begin{aligned}
& \int_\Omega \v_\delta \cdot \nabla \Td(\varphi_\delta) (\beta_\delta(\varphi_\delta) - \lap \varphi_\delta) \dx - 2  \norm{\eta^{1/2}(\varphi_\delta)\D \v_\delta}_{\LP^2}^2 - \nu \norm{\v_\delta}_{\LP^2}^2\\
& \quad =  \int_\Omega \Theta_c \varphi_\delta \nabla \Td(\varphi_\delta) \cdot \v_\delta + ( \psi_\delta'(\varphi_\delta) - \lap \varphi_\delta) \nabla \Td(\varphi_\delta) \cdot \u_\delta \dx \\
& \qquad - \int_\Omega 2 \eta(\varphi_\delta) \D \v_\delta \colon \D \u_\delta + \nu \v_\delta \cdot \u_\delta \dx
\end{aligned}
\end{equation}
Recalling \eqref{stat:coer:1} and \eqref{stat:coer:psi}, when substituting \eqref{stat:mu:2} to \eqref{stat:mu:1}, we have
\begin{align*}
& \frac{1}{2} \norm{\nabla \mu_\delta}_{\LP^2}^2 + \sqrt{\delta} \norm{\beta_\delta(\varphi_\delta)}_{L^2}^2 +  \eta_0 \norm{\D \v_\delta}_{\LP^2}^2 + \frac{\nu}{2} \norm{\v_\delta}_{\LP^2}^2 \\
& \quad \leq  \int_\Omega \tfrac{1}{2} \Theta_c \v_\delta \cdot \nabla \abs{\Td(\varphi_\delta)}^2  + (\beta_{\log}(\varphi_\delta) - \Theta_c \varphi_\delta - \lap \varphi_\delta) \nabla \Td(\varphi_\delta) \cdot \u_\delta  \dx \\
& \qquad + C \big ( 1 + \norm{\u_\delta}_{\H^1}^2 \big ) \\
& \quad \leq \eps \normH{1}{\v_\delta}^2 + C \Big ( 1 + \intO \beta_\delta'(\varphi_\delta)|\grad \varphi_\delta|^2 \dx + \normh{2}{\varphi_\delta}^2 + \normH{1}{\u_\delta}^2 \Big ) \\
& \quad \leq \eps \normH{1}{\v_\delta}^2 + C, 
\end{align*}
where $\eps > 0$ is a constant yet to be determined.  Invoking Korn's inequality and choosing $\eps$ sufficiently small, we arrive at
\begin{align}\label{stat:est:2}
\norm{\nabla \mu_\delta}_{\LP^2}^2 + \sqrt{\delta} \norm{\beta_\delta(\varphi_\delta)}_{L^2}^2 + \norm{\v_\delta}_{\H^1}^2 \leq C.
\end{align}
In particular, we deduce that $\v_\delta \rightharpoonup \v$ in $\H^1$ to some limit velocity $\v$.  Moreover,
\begin{align*}
\norm{\sqrt{\delta} \beta_\delta(\varphi_\delta)}_{L^2}^2 \leq C \sqrt{\delta} \to 0 \text{ as } \delta \to 0,
\end{align*}
and so $\sqrt{\delta} \beta_\delta(\varphi_\delta) \to 0$ in $L^2$.  Then, choosing $\zeta = 1$ in \eqref{stat:w:1}, using that $\mu_\delta \in W$ and passing to the limit $\delta \to 0$ yields
\begin{align}\label{stat:mean}
0 = \intO \gamma(\varphi, \sigma) + \v \cdot \nabla \varphi \dx,
\end{align}
where the convergence of the convection term follows from an analogous integration by parts as in the proof of Lemma \ref{lem:calA}, and the convergence \eqref{Td:conv}.  Then, we infer from \eqref{ASS_SOURCE} and \eqref{stat:mean} that the limit $\varphi$ has mean value $\varphi_\Omega \in (-1,1)$.  Indeed, subsituting $\varphi = 1$ or $-1$ in \eqref{stat:mean} leads to a contradiction on account of \eqref{ASS_SOURCE}, and as $\abs{\varphi} \leq 1$ a.e.~in $\Omega$, we have that $\varphi_\Omega \in (-1,1)$.

Arguing as in the time-dependent case we can derive a uniform estimate on the mean value of $\mu_\delta$, and consequently 
\begin{align*}
\norm{\mu_\delta}_{L^2} + \norm{\beta_{\log,\delta}(\varphi_\delta)}_{L^2} + \norm{\beta_{\Do,\delta}(\varphi_\delta)}_{L^2} + \norm{p_\delta}_{L^2} \leq C,
\end{align*}
where the boundedness of $\beta_{\log,\delta}(\varphi_\delta)$ in $L^2$ implies the tighter bounds
\begin{align*}
\abs{\varphi} < 1 \text{ a.e.~in } \Omega
\end{align*}
in the case of the logarithmic potential.
\subsection{Passing to the limit}
In \eqref{stat:w:1} we take $\zeta \in H^1$ and apply integration by parts to get
\begin{align*}
0 = \int_\Omega \big ( \sqrt{\delta} \beta_\delta(\varphi_\delta) + \varphi_\delta \Gamma_{\v_\delta} - \Gamma_{\varphi_\delta} + \v_\delta \cdot \nabla \Td(\varphi_\delta) \big ) \zeta + \nabla \mu_\delta \cdot \nabla \zeta \dx \quad \forall \zeta \in H^1.
\end{align*}
Passing to the limit $\delta \to 0$ then yields \eqref{Stat_Form}.  Meanwhile, \eqref{subdiff} or \eqref{log:weak} can be recovered in the limit $\delta \to 0$ from \eqref{stat:w:2} in a fashion similar to the time-dependence case, as with the recovery of \eqref{MEQ_1} and \eqref{WFORM_1d}.

It remains to show that the weak limit $p$ of $(p_\delta)_{\delta \in (0,\delta_5)}$ together with $\v$ constitutes a solution to the corresponding Brinkman system in the sense of \eqref{WFORM_1a}.  From the definition \eqref{stat:mudelta} of $\mu_\delta$ we observe from \eqref{stat:Brink} that $(\v_\delta, p_\delta)$ satisfies
\begin{equation}\label{stat:Brink:weak}
\begin{aligned}
0 & = \int_\Omega 2 \eta(\varphi_\delta) \D \v_\delta : \D \boldsymbol{\Phi} + \lambda(\varphi_\delta) \div (\v_\delta) \div (\boldsymbol{\Phi}) + \nu \v_\delta \cdot \boldsymbol{\Phi} \dx \\
& \quad - \int_\Omega p_\delta \div (\boldsymbol{\Phi}) + (\mu_\delta + \chi \sigma_\delta) \nabla \Td(\varphi_\delta) \cdot \boldsymbol{\Phi} \dx
\end{aligned}
\end{equation}
for $\boldsymbol{\Phi} \in \H^1$.  For the last term, after integrating by parts and using $\Td(\varphi_\delta) \bPhi \to \varphi \bPhi$ in $\LP^2(\Omega)$ and in $\LP^2(\Sigma)$, $\mu_\delta \to \mu$ in $L^4(\Omega)$ and in $L^2(\Sigma)$, we see that 
\begin{align*}
\intO (\mu_\delta + \chi \sigma_\delta) \nabla \Td(\varphi_\delta) \cdot \bPhi \dx \to \intO (\mu + \chi \sigma) \nabla \varphi \cdot \bPhi \dx
\end{align*}
for all $\bPhi \in \H^1$.  Hence, passing to the limit $\delta \to 0$ in \eqref{stat:Brink:weak} allows us to recover \eqref{WFORM_1a} and thus the quintuple $(\varphi, \mu, \sigma, \v, p)$ is a stationary solution in the sense of Definition \ref{defn:stat}.

Moreover, from the above estimates and weak lower semicontinuity of norms, we know that
\begin{equation}
\label{stat:est:3}\normh{2}{\varphi} + \normh{1}{\mu} + \normh{2}{\sigma}+ \normH{1}{\v}+\norml{2}{p}\leq C.
\end{equation}
Then, from \eqref{Stat_Form} and elliptic regularity, we deduce that
\begin{equation}
\label{stat:est:4}\normh{2}{\mu}\leq C.
\end{equation}
In light of this improved regularity and the Sobolev embedding $H^2 \subset L^{\infty}$, it is easy to see that $(\mu + \chi \sigma) \nabla \varphi \in \LP^q$ where $q < \infty$ for $d = 2$ and $q = 6$ for $d = 3$.  Invoking Lemma \ref{lem:Brink}, we then infer 
\begin{equation}
\label{stat:est:5} \normW{2}{q}{\v} + \normw{1}{q}{p}\leq C,
\end{equation}
which completes the proof.

\section{Proof of Theorem \ref{thm:Darcy} -- Darcy's law}\label{sec:Darcy}
We can adapt most of the arguments and estimates from the proof of Theorem \ref{thm:timedep}.  The main idea is to consider a weak solution quintuple $(\varphi_\delta, \mu_\delta, \sigma_\delta, \v_\delta, p_\delta)$ to the CHB model \eqref{MEQ}-\eqref{BIC} with stress tensor
\begin{align*}
\T_\delta(\v_\delta, p_\delta) := 2\delta \D v_\delta + \delta \div (\v_\delta) \I - p_\delta \I,
\end{align*}
where we have set $\eta(\cdot) = \lambda(\cdot) = \delta$.  Proceeding as in the proof of Theorem \ref{thm:timedep} we obtain the uniform estimates \eqref{APRI_EQ_4}, \eqref{APRI_EQ_5} and 
\begin{equation}\label{APRI_D_EQ_1}
\begin{aligned}
& \norm{\psi_{\delta}(\varphi_\delta)}_{L^{\infty}(0,T;L^1)} + \norm{\varphi_\delta}_{L^{\infty}(0,T;H^1) \cap L^2(0,T;H^2)} + \norm{\nabla \mu_\delta}_{L^2(0,T;\LP^2)} \\
& \quad + \norm{(\beta_\delta'(\varphi_\delta))^{1/2} \nabla \varphi_\delta}_{L^2(0,T;\LP^2)} + \norm{\v_\delta}_{L^2(0,T;\LP^2_{\div})} + \sqrt{\delta} \norm{\D \v_\delta}_{L^2(0,T;\LP^2)} \leq C,
\end{aligned}
\end{equation}
where in the above $\psi_\delta$ and $\beta_\delta$ denote the approximations to either singular potentials and the derivatives of the corresponding convex part.  From the first equality of \eqref{APRI_EQ_15} with $A = 1$, it holds that
\begin{align*}
\norm{\Delta \varphi_\delta}_{L^2}^2 \leq C\Big ( 1 +  \norm{\nabla (\mu_\delta +\chi \sigma_\delta)}_{\LP^2} \Big ) \in L^2(0,T),
\end{align*}
so that by elliptic regularity we can infer
\begin{align}\label{APRI_D_EQ_2}
\norm{\varphi_\delta}_{L^4(0,T;H^2)} \leq C.
\end{align}
Moreover, by the Gagliardo--Nirenburg inequality, we find that
\begin{align*}
\int_0^T \norm{\nabla \varphi_\delta \cdot \v_\delta}_{L^{6/5}}^{\frac{8}{5}} \dt \leq C \norm{\varphi_\delta}_{L^{\infty}(0,T;H^1)}^{\frac{4}{5}} \norm{\varphi_\delta}_{L^{4}(0,T;H^2)}^{\frac{4}{5}} \norm{\v_\delta}_{L^2(0,T;\LP^2)}^{\frac{8}{5}} \leq C,
\end{align*}
so that from \eqref{WFORM_1b} and previous uniform estimates we arrive at
\begin{align}
\label{APRI_D_EQ_3}
\norm{\del_t \varphi_\delta}_{L^{8/5}(0,T;(H^1)^*)} + \norm{\nabla \varphi_\delta \cdot \v_\delta}_{L^{8/5}(0,T;L^{6/5})}  \leq C, \\
\label{APRI_D_EQ_4} \norm{\del_t \varphi_\delta + \div(\varphi_\delta \v_\delta)}_{L^2(0,T;(H^1)^*)} + \norm{(\varphi_\delta)_\Omega}_{W^{1,8/5}(0,T)} \leq C, \\
\label{APRI_D_EQ_5}
| (\varphi_\delta)_\Omega(r) - (\varphi_\delta)_\Omega(s)| \leq C |r-s|^{3/8} \quad \forall r,s \in (0,T).
\end{align}
Let us mention that the sum $\del_t \varphi_\delta + \div(\varphi_\delta \v_\delta)$ has better temporal integrability than either of its constituents, a fact which will play an important role for deriving uniform estimates for $(\mu_\delta)_\Omega$ below.

By reflexive weak compactness arguments and \cite[Sec.~8, Cor.~4]{Simon}, for $\delta \to 0$ along a non-relabelled subsequence, it holds that for any $r \in [1,6)$,
\begin{equation*}
\begin{alignedat}{3}
\varphi_\delta & \to \varphi && \quad \text{ weakly-*} && \quad \text{ in } W^{1,\frac{8}{5}}(0,T;(H^1)^*) \cap L^{\infty}(0,T;H^1) \cap L^4(0,T;H^2), \\
\varphi_\delta & \to \varphi && \quad \text{ strongly } && \quad \text{ in } C^0([0,T];L^r) \cap L^4(0,T;W^{1,r}) \text{ and a.e. in } \Omega_T, \\
\sigma_\delta & \to \sigma && \quad \text{ weakly-*} && \quad \text{ in } L^{\infty}(0,T;H^2), \\
\v_\delta & \to \v && \quad \text{ weakly } && \quad \text{ in } L^2(0,T;\LP^2), \\
\div(\varphi_\delta \v_\delta) & \to \theta && \quad \text{ weakly } && \quad \text{ in } L^{\frac{8}{5}}(0,T;L^{\frac{6}{5}}).
\end{alignedat}
\end{equation*}
The identification $\theta = \div(\varphi \v)$ follows analogously as in Section \ref{sec:mu:mean}, where the assertion \eqref{div:id} now holds for arbitrary $\lambda \in L^4(0,T;L^6)$ by the strong convergence $\nabla \varphi_\delta \to \nabla \varphi$ in $L^4(0,T;\LP^3)$ and the weak convergence $\v_\delta \rightharpoonup \v$ in $L^2(0,T;\LP^2)$.

In order to obtain uniform estimates for the chemical potential $\mu_\delta$ in $L^2(0,T;L^2)$, we again follow the argument in Section \ref{sec:mu:mean}.  Namely, we pass to the limit $\delta \to 0$ in \eqref{WFORM_1b} to obtain \eqref{APRI_EQ_26}, and use the uniform boundedness of $\psi_\delta(\varphi_\delta)$ in $L^1(0,T;L^1)$ from \eqref{APRI_D_EQ_1} to obtain that the limit $\varphi$ satisfies the pointwise bound \eqref{APRI_EQ_27}.  Choosing $\zeta = 1$ in \eqref{APRI_EQ_26} leads to \eqref{APRI_EQ_28} and obtain by contradiction argument that $(\varphi(t))_\Omega \in (-1,1)$ for all $t \in (0,T)$.

Defining $f_\delta \in H^2_n \cap L^2_0$ as the unique solution to \eqref{DO_f_delta} which satisfies \eqref{APRI_EQ_30}.  Then, the right-hand side of  \eqref{APRI_EQ_31} can be estimated as
\begin{equation}\label{APRI_D_EQ_6}
\begin{aligned}
\mathrm{RHS} & \leq C \Big ( \norm{\sigma_\delta(t)}_{L^2}^2 + \norm{\varphi_\delta(t)}_{L^2}^2 \Big ) + \norm{\Gamma_\varphi(\varphi_\delta(t), \sigma_\delta(t))}_{L^2} \norm{f_\delta}_{L^2} \\
& \quad + \norm{\del_t\varphi_\delta(t) + \div(\varphi_\delta(t) \v_\delta(t))}_{(H^1)^*} \norm{f_\delta}_{H^1},
\end{aligned}
\end{equation}
which is bounded in $L^2(0,T)$ by \eqref{APRI_D_EQ_4}.  This modification allows us to infer that $(\mu_\delta)_\Omega$ is uniformly bounded in $L^2(0,T)$, whereas simply using \eqref{APRI_D_EQ_3} would only give the uniform boundedness of $(\mu_\delta)_\Omega$ in $L^{\frac{8}{5}}(0,T)$.  Hence, we recover the uniform $L^2(0,T;L^2)$-estimate \eqref{APRI_EQ_33} for $\mu_\delta$ and also \eqref{APRI_EQ_34} for $\beta_{\Do, \delta}(\varphi_\delta)$ and $\beta_{\log,\delta}(\varphi_\delta)$.  Moreover, setting $\eta(\cdot) = \delta$ and $\eta_1 = \lambda_1 = \delta$, we obtain as in the end of Section \ref{sec:mu:mean} the uniform $L^2(0,T;L^2)$-estimate \eqref{APRI_EQ_36} for $p_\delta$.  

Then, proceeding as in the proof of Theorem \ref{thm:timedep}, we can recover \eqref{WFORM_1b}, \eqref{WFORM_1d} and \eqref{subdiff} (resp.~\eqref{log:weak}) for the double obstacle (resp.~logarithmic) case in the limit $\delta \to 0$, whereas recovery of \eqref{WFORM_D_1a}, \eqref{WFORM_D_1b}, the improved regularity $p \in L^{\frac{8}{5}}(0,T;H^1)$ and the boundary condition \eqref{BC_D} follow from similar arguments outlined in \cite[Sec.~4.2]{EG2}.

\section{Well-posedness of the Brinkman system \eqref{BM_SUBSY}}\label{sec:appendix}
\subsection{Weak solvability}
\begin{thm}\label{thm:A:Brink}
Let $\Omega \subset \R^d$, $d = 2,3$, be a bounded domain with $C^{2,1}$-boundary $\Sigma$.  Let $c \in W^{1,r}$ with $r > d$ be given, fix exponent $q$ such that $q \in (1,2]$ for $d = 2$ and $q \in (\frac{6}{5}, 2]$ for $d = 3$, and suppose $\eta(\cdot)$ and $\lambda(\cdot)$ satisfy \eqref{ass:visc}.  Then, for any $(\f, g, \bh) \in \LP^q \times L^2 \times (\H^{1/2}(\Sigma))^*$, there exists a unique weak solution $(\v, p) \in \H^1 \times L^2$ to
\begin{subequations}\label{A:Brink}
\begin{alignat}{2}
-\div( 2 \eta(c) \D \v + \lambda(c) \div (\v) \I - p \I) + \nu\v = \f& \text{ in }\Omega,\\
\div(\v)=g& \text{ in }\Omega,\\
\T_c(\v, p) \n := (2 \eta(c) \D \v + \lambda(c) \div(\v) \I - p \I) \n=\bh & \text{ on }\Sigma,
\end{alignat} 
\end{subequations}
in the sense
\begin{align*}
\int_\Omega 2 \eta(c) \D \v : \D \bPhi + (\lambda(c) \div (\v) - p) \div \bPhi + \nu \v \cdot \bPhi \dx = \int_\Omega \f \cdot \bPhi \dx + \inn{\bh}{\bPhi}_{\H^{1/2}}
\end{align*}
for all $\bPhi \in \H^1$.  Furthermore, it holds that 
\begin{align}\label{A:Brink:est}
\normH{1}{\v} + \norml{2}{p} \leq C \Big ( \normL{q}{\f} + \norml{2}{g} + \norm{\bh}_{(\H^{1/2})^*} \Big )
\end{align}
for a constant $C > 0$ depending only on $\Omega$, $q$, $\eta_0$, $\eta_1$, $\lambda_0$ and $\nu$.
\end{thm}

\begin{proof}
For $(x_1, \dots, x_d)^{\top}\in\Omega$, we define
\begin{align*}
g_\Omega := \frac{1}{\abs{\Omega}} \int_\Omega g \dx, \quad \v_0 := \frac{g_\Omega}{d}(x_1, \dots, x_d)^{\top}, \quad p_0 := g_\Omega \Big ( \frac{2 \eta(c)}{d} + \lambda(c) \Big ).
\end{align*}
Then, a short calculation shows that $\div (\v_0) = g_\Omega$ and $\T_c(\v_0, p_0) = \0$.  Next, we will show there exist unique weak solutions $(\w, \pi)$ and $(\y, \theta)$ to the systems
\begin{align*}
(\mathrm{P}_1) \begin{cases}
- \div (\T_c(\w, \pi))+ \nu \w = \0 & \text{ in } \Omega, \\
\div (\w) = g - g_\Omega & \text{ in } \Omega, \\
\w = \0 & \text{ on } \Sigma, 
\end{cases} \quad 
(\mathrm{P}_2)\begin{cases}
-\div (\T_c(\y, \theta)) + \nu  \y = \f - \nu \v_0 & \text{ in } \Omega, \\
\div (\y)= 0 & \text{ in } \Omega, \\
\T_c(\y, \theta) \n = \bh - \T_c(\w, \pi) \n& \text{ on } \Sigma.
\end{cases}
\end{align*}
Then, one can check that the pair $\v := \w + \y + \v_0$ and $p := \theta + \pi + p_0$ is the unique weak solution to \eqref{A:Brink}.

\paragraph{Solvability of $(\mathrm{P}_1)$.} By Lemma \ref{LEM_DIVEQU}, there exists $\u := \DD(g - g_\Omega) \in \H^1_0$ satisfying
\begin{align*}
\normH{1}{\u} \leq C \norml{2}{g - g_\Omega} \leq C \norml{2}{g}.
\end{align*}
Let $\tilde{\f} := - \nu \u + \div (2 \eta(c) \D \u + \lambda(c) \div(\u)\I) \in (\H^1_0)^*$, and consider the function space $\W_0 := \{ \f \in \H^1_0 : \div (\f) = 0 \text{ a.e.~in } \Omega \}$.  By the Lax--Milgram theorem, there exists a unique solution $\hat{\w} \in \W_0$ to
\begin{align*}
\int_\Omega 2 \eta(c) \D \hat{\w} : \D \bPhi + \nu \hat \w \cdot \bPhi \dx = \inn{\tilde{\f}}{\bPhi}_{\H^1_0} \quad \forall \bPhi \in \W_0
\end{align*}
and satisfies
\begin{align}\label{A:hatw}
\normH{1}{\hat{\w}} \leq C \norm{\tilde{\f}}_{(\H^1_0)^*} \leq C \normH{1}{\u} \leq C \norml{2}{g}.
\end{align}
Then, the function $\w := \hat{\w} + \u \in \H^1_0$ satisfies $\div (\w) = g - g_\Omega$ a.e.~in $\Omega$ and
\begin{align*}
\int_\Omega 2 \eta(c) \D \w : \D \bPhi + \lambda(c) \div (\w) \div (\bPhi) + \nu \w \cdot \bPhi \dx = 0 \quad \forall \bPhi \in \W_0,
\end{align*}
i.e., $\w$ is the first component of the solution to $(\mathrm{P}_1)$.  The recovery of the unique pressure variable $\pi \in L^2_0$ follows from the application of \cite[p.~75, Lem.~2.2.2]{Sohr}, which also yields
\begin{align}\label{A:pi}
\norml{2}{\pi} \leq C \normH{1}{\w} \leq C \norml{2}{g}.
\end{align}
It is also clear that $(\w, \pi)$ constructed above is the unique solution to $(\mathrm{P}_1)$, and for any $\bPhi \in \H^1$ it holds that 
\begin{align*}
\inn{\T_c(\w, \pi)\n}{\bPhi}_{\H^{1/2}(\Sigma)} = \int_\Omega \T_c(\w, \pi) : \nabla \bPhi + \nu \w \cdot \bPhi \dx \leq C \big (\normH{1}{\w} + \norml{2}{\pi} \big ) \normH{1}{\bPhi},
\end{align*}
and so
\begin{align}\label{A:Tcw}
\norm{\T_c(\w, \pi)\n}_{(\H^{1/2}(\Sigma))^*} \leq C \norml{2}{g}.
\end{align}

\paragraph{Solvability of $(\mathrm{P}_2)$.}  We define the function space
\begin{align*}
\W_{\div}^{1,r} := \{ \f \in \W^{1,r} : \div (\f) = 0 \text{ a.e.~in } \Omega \}.\end{align*}
By the Lax--Milgram theorem, there exists a unique solution $\y \in \W_{\div}^{1,2}$ to
\begin{align}\label{A:1}
\F(\bPhi) := \int_\Omega 2 \eta(c) \D \y : \D \bPhi + \nu \y \cdot \bPhi - \f \cdot \bPhi \dx + \inn{\bh - \T_c(\w, \pi) \n}{\bPhi}_{\H^{1/2}(\Sigma)} = 0
\end{align}
holding for all $\bPhi \in \W_{\div}^{1,2}$ and satisfies
\begin{align}\label{A:y}
\normH{1}{\y} \leq C \Big ( \normL{q}{\f} + \norm{\bh}_{(\H^{1/2}(\Sigma))^*} + \norml{2}{g} \Big ).
\end{align}
By \cite[p.~75, Lem.~2.2.2]{Sohr}, there exists a unique pressure $\hat \theta \in L^2_0$ such that $- \nabla \hat \theta = \F$ in the sense of distribution, with 
\begin{align}\label{A:hatthe}
\norml{2}{\hat \theta} \leq C \norm{\F}_{(\H^1)^*} \leq C \Big ( \normL{q}{\f} + \norm{\bh}_{(\H^{1/2}(\Sigma))^*} + \norml{2}{g} \Big ).
\end{align}
It remains to adjust this pressure by a uniquely defined constant $c_0$ so that $\y$ and $\theta := \hat \theta + c_0$ satisfy the boundary condition $\T_c(\y, \theta) \n = \bh - \T_c(\w, \pi)\n$.  Let $q' = \frac{q}{q-1}$ denote the conjugate of $q$, where by the hypothesis it holds that $q' \geq 2$.  From the distributional equation $- \nabla \hat \theta = \F$, we find that the weak divergence of $2 \eta(c) \D \y - \hat \theta \I$ satisfies
\begin{align}\label{A:div}
- \div (2 \eta(c) \D \y - \hat \theta \I) = \f - \nu \y \in \LP^q,
\end{align}
and so $2 \eta(c) \D \y - \hat \theta \I \in \LP^q_{\div}$.  By \cite[Thm.~III.2.2]{Galdi}, it holds that $(2 \eta(c) \D \y - \hat \theta \I) \n \in (\W^{1/q,q'}(\Sigma))^*$ where by the generalised Gauss identity and \eqref{A:div} we have
\begin{align*}
\inn{(2 \eta(c) \D \y - \hat \theta \I) \n}{\bPhi}_{W^{\frac{1}{q},q'}(\Sigma)} = \int_\Omega (2 \eta(c) \D \y - \hat \theta \I) : \nabla \bPhi - \bPhi (\f - \nu \y) \dx \quad \forall \bPhi \in \W^{1,q'}.
\end{align*}
In turn, as $\W^{1,q'} \subset \H^1$, this gives
\begin{align}\label{A:Dy}
\norm{(2 \eta(c) \D \y - \hat \theta \I) \n}_{(\W^{\frac{1}{q},q'}(\Sigma))^*} \leq C \Big ( \normL{q}{\f} + \norm{\bh}_{(\H^{1/2}(\Sigma))^*} + \norml{2}{g} \Big ).
\end{align}
From testing \eqref{A:div} with arbitrary $\bPhi \in \W^{1,q'}_{\div}$, we obtain
\begin{align}\label{A:2}
\inn{(2 \eta(c) \D \y - \hat \theta \I) \n}{\bPhi}_{\W^{\frac{1}{q},q'}(\Sigma)} = \int_\Omega 2 \eta(c) \D \y : \D \bPhi + \nu \y \cdot \bPhi - \f \cdot \bPhi \dx.
\end{align}
Since $q' \geq 2$, we have $(\H^{1/2}(\Sigma))^* \subset (\W^{1/q,q'}(\Sigma))^*$ which implies that $\bh - \T_c(\w, \pi) \n \in  (\W^{1/q,q'}(\Sigma))^*$.  Comparing \eqref{A:1} with $\bPhi \in \W^{1,q'}_{\div}$ and \eqref{A:2} then gives
\begin{align}\label{A:3}
\inn{(2 \eta(c) \D \y - \hat \theta \I) \n - \bh + \T_c(\w, \pi) \n}{\bPhi}_{\W^{\frac{1}{q},q'}(\Sigma)} = 0 \quad \forall \bPhi \in \W^{1,q'}_{\div}.
\end{align}
The regularity of the boundary $\Sigma$ implies the normal vector $\n \in \W^{1/q,q'}(\Sigma)$, and for arbitrary $\bpsi \in \W^{1/q,q'}(\Sigma)$ we define $\widehat{\bpsi} := \bpsi - \frac{1}{\abs{\Sigma}} \int_\Sigma \bpsi \cdot \n \dH$ which satisfies $\int_\Sigma \widehat{\bpsi} \cdot \n \dH = 0$.  By Lemma \ref{LEM_DIVEQU} there exists a solution $\widehat \u \in \W^{1,q'}_{\div}$ to the divergence problem
\begin{align*}
\begin{cases}
\div (\widehat \u) = 0 & \text{ in } \Omega, \\
\widehat \u = \widehat{\bpsi} & \text{ on } \Sigma.
\end{cases}
\end{align*}
Then, substituting $\bPhi = \widehat \u$ in \eqref{A:3} leads to
\begin{align*}
& \inn{(2 \eta(c) \D \y - \hat \theta \I) \n - \bh + \T_c(\w, \pi) \n}{\bpsi}_{W^{\frac{1}{q},q'}(\Sigma)} = \inn{c_0 \n}{\bpsi}_{W^{\frac{1}{q},q'}(\Sigma)} \forall \bpsi \in W^{\frac{1}{q},q'}(\Sigma), \\
& \quad \text{ where } c_0 := \frac{1}{\abs{\Sigma}} \inn{(2 \eta(c) \D \y - \hat \theta \I) \n - \bh + \T_c(\w, \pi) \n}{\n}_{\W^{\frac{1}{q},q'}(\Sigma)}.
\end{align*}
Setting $\theta := \hat \theta + c_0$ in turn shows that $(\y, \theta)$ satisfies the boundary condition $\T_c(\y, \theta) = \bh - \T_c(\w, \pi)$ on $\Sigma$ in the following sense:
\begin{align*}
\inn{(2 \eta(c) \D \y - \theta \I )\n - \bh + \T_c(\w, \pi)\n}{\bpsi}_{W^{\frac{1}{q},q'}(\Sigma)} = 0 \quad \forall \bpsi \in W^{\frac{1}{q},q'}(\Sigma).
\end{align*}
Lastly, combining \eqref{A:hatw}, \eqref{A:pi}, \eqref{A:Tcw}, \eqref{A:y} and \eqref{A:hatthe} and \eqref{A:Dy} it is easy to infer the estimate \eqref{A:Brink:est} for $\v = \w + \y + \v_0$ and $p = \hat \theta + c_0 + p_0 + \pi_0$.
\end{proof}

\subsection{Strong solvability}
Let us first state the following auxiliary result for the Brinkman system with constant viscosities.

\begin{lem}\label{lem:B:const}
Let $\Omega \subset \R^d$, $d = 2,3$, be a bounded domain with $C^{3}$-boundary $\Sigma$ and outer unit normal $\n$.  Let $\eta, \lambda$ and $\nu$ be positive constants, and fix exponent $q$ such that $q > 1$ for $d = 2$ and $q \geq \frac{6}{5}$ for $d = 3$.  Then, for any $\f \in \LP^q$, $g \in W^{1,q}$ and $\bh \in \W^{1-1/q,q}(\Sigma)$, there exists a unique solution $(\v, p) \in \W^{2,q} \times W^{1,q}$ of the system
\begin{subequations}\label{App:B:sys}
\begin{alignat}{2}
- \div ( 2 \eta \D \v + \lambda \div (\v) \I) + \nu \v + \nabla p = \f & \text{ a.e.~in } \Omega, \\
\div(\v) = g & \text{ a.e.~in } \Omega, \\
(2 \eta \D \v + \lambda \div(\v) \I - p \I)\n = \bh  &\text{ a.e.~on } \del \Omega,
\end{alignat}
\end{subequations}
satisfying the following estimate
\begin{align}\label{App:B:constvis}
\normW{2}{q}{\v} + \normw{1}{q}{p} \leq C \Big (\normL{q}{\f} + \normw{1}{q}{g} + \norm{\bh}_{\W^{1-\frac{1}{q}, q}(\Sigma)} \Big ),
\end{align}
with a constant $C$ depending only on $\eta$, $\lambda$, $\nu$, $q$ and $\Omega$.
\end{lem}

\begin{proof}
We consider the functions analogous to $\v_0$, $p_0$, $\w$, $\pi$, $\y$ and $\theta$ defined in the proof of Theorem \ref{A:Brink} for the case $\eta(c) = \eta$ and $\lambda(c) = \lambda$, whilst reusing the notation.  It is clear that 
\begin{align*}
\normW{2}{q}{\v_0} + \normw{1}{q}{p_0} \leq C \normw{1}{q}{g}.
\end{align*}
By \cite[Thm.~1.2]{FS}, there exists a unique solution $(\z, \phi) \in (\W^{2,q} \cap \W^{1,q}_0) \times (W^{1,q} \cap L^q_0)$ to the problem
\begin{align*}
\begin{cases}
\nu \z - \eta \lap \z + \nabla \phi = \f + (\eta + \lambda) \nabla g & \text{ in } \Omega, \\
\div(\z) = g - g_\Omega & \text{ in } \Omega, \\
\z = \0 & \text{ on } \Sigma,
\end{cases}
\end{align*}
satisfying
\begin{align*}
\normW{2}{q}{\z} + \normw{1}{q}{\phi} \leq C \Big ( \norml{q}{ \f + (\eta + \lambda) \nabla g} + \normw{1}{q}{g-g_\Omega} \Big ).
\end{align*}
By the assumption on the exponent $q$, it holds that $\W^{2,q} \subset \H^1$ and $W^{1,q} \subset L^2$.  Hence, the difference $\widetilde{\w} := \w - \z$ and $\widetilde{\pi} := \pi - \phi$ constitutes a weak solution in $\H^1_0 \times L^2_0$ to the system
\begin{align*}
\begin{cases}
- \div (2 \eta \D \widetilde{\w} + \lambda \div(\widetilde{\w}) \I - \widetilde{\pi} \I) + \nu \widetilde{\w} = \0 & \text{ in } \Omega, \\
\div (\widetilde{\w}) = 0 & \text{ in } \Omega, \\
\widetilde{\w} = \0 & \text{ on } \Sigma, 
\end{cases}
\end{align*}
Unique solvability implies $(\widetilde{\w}, \widetilde{\pi}) = (\0, 0)$ and hence the unique weak solution $(\w, \pi)$ to $(\mathrm{P}_1)$ with constant viscosities is in fact a strong solution satisfying
\begin{align*}
\normW{2}{q}{\w} + \normw{1}{q}{\pi} \leq C \Big ( \norml{q}{\f} + \normw{1}{q}{g} \Big ).
\end{align*}
Furthermore, by the trace theorem, $\T(\w, \pi) \n:= (2 \eta \D \w + \lambda \div (\w) \I - \pi \I) \n \in \W^{1/q',q}(\Sigma)$ where $q' = \frac{q}{q-1}$ denotes the conjugate of $q$.  Hence, there exists an extension $\mathbf{E} \in \W^{1,q}$ of $\bh - \T(\w, \pi)\n$ such that $\normW{1}{q}{\mathbf{E}} \leq C \norm{\bh - \T(\w, \pi)\n}_{\W^{1/q',q}(\Sigma)}$.  We now consider $(\mathrm{P}_2)$ with boundary data $\mathbf{E} \vert_{\Sigma}$, and invoke \cite[Thm.~1.1]{ShiShi} (which requires a $C^3$-boundary for $\Omega$) to deduce the existence of a unique solution $(\y, \theta) \in \W^{2,q} \times W^{1,q}$ satisfying
\begin{align*}
\normW{2}{q}{\y} + \normw{1}{q}{\theta} & \leq C \Big ( \norml{q}{\f - \nu \v_0} + \normw{1}{q}{\mathbf{E}} \Big ) \\
& \leq C \Big ( \norml{q}{\f} + \normw{1}{q}{g} + \norm{\bh}_{\W^{1 - \frac{1}{q}, q}(\Sigma)} \Big ).
\end{align*}
Hence, $\v := \v_0 + \w + \y \in \W^{2,q}$ and $p := p_0 + \pi + \theta \in W^{1,q}$ constitute the unique strong solution to \eqref{App:B:sys} satisfying \eqref{App:B:constvis}.
\end{proof}

\subsection{Proof of Lemma \ref{lem:Brink}}
Fix $c \in W^{1,r}$ with exponent $r > d$ and data $\f \in \LP^q$, $g \in W^{1,q}$ and $\bh \in \W^{1-1/q,q}(\Sigma)$, where the exponent $q \leq r$ satisfies $q > 1$ for $d = 2$ and $q \geq \frac{6}{5}$ for $d = 3$.  Note that for $d = 3$, we have the Sobolev embedding
\begin{align*}
\W^{\frac{1}{q'},q}(\Sigma) \subset \LP^{\frac{4}{3}}(\Sigma) \subset (\H^{\frac{1}{2}}(\Sigma))^*,
\end{align*}
and for $d = 2$ we have $\W^{1/q',q}(\Sigma) \subset \LP^k(\Sigma)$ and $\H^{1/2}(\Sigma) \subset \LP^{k/(k-1)}(\Sigma)$ for any $k > 1$.  Hence, $\bh \in (\H^{1/2}(\Sigma))^*$, and we invoke Theorem \ref{thm:A:Brink} to deduce the existence of a unique weak solution $(\v, \hat p) \in \H^1 \times L^2$ to 
\begin{align}\label{B:strong:sys1}
\begin{cases}
- \div (2 \eta(c) \D \v - \hat p \I) + \nu \v = \f & \text{ in } \Omega, \\
\div (\v) = g & \text{ in } \Omega, \\
(2 \eta(c) \D \v - \hat p \I) \n = \bh & \text{ on } \Sigma
\end{cases}
\end{align}
in the sense that 
\begin{align*}
\int_\Omega (2 \eta(c) \D \v - \hat p \I): \nabla \bPhi + \nu \v \cdot \bPhi = \int_\Omega \f \cdot \bPhi \dx + \inn{\bh}{\bPhi}_{\H^{1/2}(\Sigma)}
\end{align*}
for all $\bPhi \in \H^1$.  Introducing $\bpsi := \eta(c) \bPhi$, we see that
\begin{equation}\label{B:constvisc} 
\begin{aligned}
& \int_\Omega (2 \D \v - \tfrac{1}{\eta(c)} \hat p \I) : \nabla \bpsi + \nu \v \cdot \bpsi \dx - \inn{\tfrac{1}{\eta(c)} \bh}{\bpsi}_{\H^{1/2}(\Sigma)} \\
& \quad = \int_\Omega \left (\tfrac{1}{\eta(c)} \f  + \nu ( 1 - \tfrac{1}{\eta(c)} ) \v - (2 \eta(c) \D \v - \hat p \I) \nabla (\tfrac{1}{\eta(c)}) \right ) \cdot \bpsi \dx =: \int_\Omega \k \cdot \bpsi \dx
\end{aligned}
\end{equation}
for all $\bpsi \in \H^1$.  Since $q > 1$ and $1 - \frac{1}{q} \notin \mathbb{Z}$, surjectivity of the trace operator yields the existence of an extension $\widehat{\bh} \in \W^{1,q}(\Omega)$ of $\bh \in \W^{1 -1/q,q}(\Sigma)$ such that $\normW{1}{q}{\widehat{\bh}} \leq C \norm{\bh}_{\W^{1-1/q,q}(\Sigma)}$.  Furthermore, as $r > d$, we have the Sobolev embedding $\W^{1,q} \subset \LP^{\frac{qr}{r-q}}$ and $c \in W^{1,r} \subset C^{0,1-\frac{d}{r}}(\overline{\Omega})$.  Hence, it is easy to see that $\frac{1}{\eta(c)} \widehat{\bh} \in \W^{1,q}$ and by the trace theorem $\frac{1}{\eta(c)} \bh \in \W^{1-1/q,q}(\Sigma)$.  Next, we define the exponent $s = \frac{2r}{2+r} < 2$ so that \begin{align*}
\frac{1}{s} = \frac{1}{r} + \frac{1}{2} \quad \implies \quad s > \begin{cases}
1 & \text{ if } d = 2, \\
\frac{6}{5} & \text{ if } d = 3.
\end{cases}
\end{align*}
At this point the analysis is divided into two cases:
\paragraph{Case 1 ($q \leq s = \frac{2r}{2+r}$).} We have $q < 2$ and $\frac{2q}{2-q} \leq r$, and so
\begin{align*}
\normL{q}{\k} & \leq C \Big ( \normL{q}{\f} + \normL{2}{\v} + \Big (\normL{2}{\D\v} + \norml{2}{\hat p} \Big ) \normL{r}{\nabla c} \Big )  \\
& \leq C \Big ( 1 + \normL{r}{\nabla c} \Big ) \Big ( \normL{q}{\f} + \normw{1}{q}{g} + \norm{\bh}_{\W^{1-\frac{1}{q},q}(\Sigma)} \Big ),
\end{align*}
which implies $\k \in \LP^q$.  From \eqref{B:constvisc}, we see that $(\v, \frac{1}{\eta(c)} \hat p) \in \H^1 \times L^2$ is a weak solution to the following system with constant viscosity
\begin{align}\label{B:sys:constvisc}
\begin{cases}
- \div (2 \D \v - \tfrac{1}{\eta(c)} \hat p \I) + \nu \v = \k & \text{ in } \Omega, \\
\div(\v) = g & \text{ in } \Omega, \\
(2 \D \v - \tfrac{1}{\eta(c)} \hat p \I) \n = \tfrac{1}{\eta(c)} \bh & \text{ on } \Sigma,
\end{cases}
\end{align}
where $\k \in \LP^q$, $g \in W^{1,q}$ and $\tfrac{1}{\eta(c)} \bh \in \W^{1-1/q,q}(\Sigma)$.  By Lemma \ref{lem:B:const}, it holds that $(\v, \tfrac{1}{\eta(c)} \hat p) \in \W^{2,q} \times W^{1,q}$ is in fact a strong solution satisfying
\begin{align*}
\normW{2}{q}{\v} + \normw{1}{q}{\tfrac{1}{\eta(c)} \hat p} & \leq C \Big ( \normL{q}{\k} + \normw{1}{q}{g} + \norm{\tfrac{1}{\eta(c)} \bh}_{\W^{1-\frac{1}{q},q}(\Sigma)} \Big ) \\
& \leq C \Big ( 1 + \normL{r}{\nabla c} \Big ) \Big ( \normL{q}{\f} + \normw{1}{q}{g} + \norm{\bh}_{\W^{1-\frac{1}{q},q}(\Sigma)} \Big ).
\end{align*}
By the assumption on $\eta$, we easily infer that $\norml{q}{\hat p} \leq C \norml{q}{\tfrac{1}{\eta(c)} \hat p}$ and 
\begin{align*}
\normw{1}{q}{\hat p} & \leq C \norml{q}{\tfrac{1}{\eta(c)} \hat p} + \normL{q}{\eta(c) \nabla (\tfrac{1}{\eta(c)} \hat p)} + \normL{q}{\hat p \tfrac{\eta'(c)}{\eta(c)} \nabla c} \\
& \leq C \normw{1}{q}{\tfrac{1}{\eta(c)} \hat p} + C \normL{r}{\nabla c} \norml{\frac{qr}{r-q}}{\hat p}  \leq C \Big (  \normw{1}{q}{\tfrac{1}{\eta(c)} \hat p}  + \normL{r}{\nabla c} \norml{2}{\hat p} \Big ) \\
& \leq C  \Big ( 1 + \normL{r}{\nabla c} \Big ) \Big ( \normL{q}{\f} + \normw{1}{q}{g} + \norm{\bh}_{\W^{1-\frac{1}{q},q}(\Sigma)} \Big ).
\end{align*}
where we used $\frac{2q}{2-q} \leq r$ to deduce that $\frac{qr}{r-q} \leq 2$. Defining $p := \hat p + \lambda(c) g$ and making use of the fact $\div (\v) = g$ we see from \eqref{B:strong:sys1} that $(\v, p)$ satisfies the Brinkman system \eqref{BM_SUBSY}.  Furthermore, from the estimate
\begin{align*}
\normw{1}{q}{p} & \leq \normw{1}{q}{\hat p} + \normw{1}{q}{\lambda(c) g} \leq \normw{1}{q}{\hat p} + C \normL{r}{\nabla c} \norml{\frac{qr}{r-q}}{g} + C \normL{q}{\nabla g} \\
& \leq \normw{1}{q}{\hat p} + C \Big ( 1 +  \normL{r}{\nabla c} \Big ) \normw{1}{q}{g} \\
& \leq C  \Big ( 1 + \normL{r}{\nabla c} \Big ) \Big ( \normL{q}{\f} + \normw{1}{q}{g} + \norm{\bh}_{\W^{1-\frac{1}{q},q}(\Sigma)} \Big ),
\end{align*}
where we used $r > d$ to deduce that $1 - \frac{d}{q} \geq - \frac{d(r-q)}{qr}$ and leads to the Sobolev embedding $W^{1,q} \subset L^{\frac{qr}{r-q}}$, we find that $(\v, p)$ also satisfies the estimate \eqref{Brink:strong:est}.

\paragraph{Case 2 ($q > s = \frac{2r}{2+r}$).}  In this case, using $s < 2$ and $r = \frac{2s}{2-s}$, we see that
\begin{align*}
\normL{s}{\k} & \leq C\normL{s}{\f} + C \normL{s}{\v} + C \Big ( \normL{2}{\D \v} + \norml{2}{\hat p} \Big ) \normL{r}{\nabla c} \\
& \leq C \Big ( 1 + \normL{r}{\nabla c} \Big ) \Big ( \normL{q}{\f} + \normw{1}{q}{g} + \norm{\bh}_{\W^{1-\frac{1}{q},q}(\Sigma)} \Big ).
\end{align*}
Using the embedding $g \in W^{1,q} \subset W^{1,s}$ and $\tfrac{1}{\eta(c)} \bh \in \W^{1-1/q,q}(\Sigma) \subset \W^{1-1/s,s}(\Sigma)$, we employ Lemma \ref{lem:B:const} with data $(\k, g, \tfrac{1}{\eta(c)} \bh)$ to deduce that $(\v, \tfrac{1}{\eta(c)} \hat p) \in \W^{2,s} \times W^{1,s}$ is a strong solution to \eqref{B:sys:constvisc} satisfying
\begin{align*}
\normW{2}{s}{\v} + \normw{1}{s}{\tfrac{1}{\eta(c)} \hat p} \leq C \Big ( 1 + \normL{r}{\nabla c} \Big ) \Big ( \normL{q}{\f} + \normw{1}{q}{g} + \norm{\bh}_{\W^{1-\frac{1}{q},q}(\Sigma)} \Big ).
\end{align*}
Then, the Sobolev embedding $W^{1,q} \subset W^{1,s} \subset L^{\frac{sr}{r-s}}$ yields 
\begin{align*}
\normw{1}{s}{\hat p} \leq C \Big ( 1 + \normL{r}{\nabla c} \Big ) \Big ( \normL{q}{\f} + \normw{1}{q}{g} + \norm{\bh}_{\W^{1-\frac{1}{q},q}(\Sigma)} \Big ).
\end{align*}
Let $t$ be the exponent defined by 
\begin{align*}
\frac{1}{t} := \frac{1}{r} + \frac{1}{2} - \frac{1}{d} = \frac{1}{s} - \frac{1}{d},
\end{align*}
then by Sobolev embedding it holds that $(\v, \hat p) \in \W^{1,t} \times L^t$.  Since $r > d$, it follows that $t > 2$.  The idea is to use this improve regularity as a starting point of a bootstrapping argument to show $\k \in \LP^q$.  Once we have $\k \in \LP^q$, following the argument in Case 1 then gives the desired assertion.

Let $s_1< t$ be the exponent defined by
\begin{align*}
\frac{1}{s_1} := \frac{1}{r} + \frac{1}{t} = \frac{1}{s} + \frac{1}{r} - \frac{1}{d}.
\end{align*}
Since $r > d$, we have $s_1 > s$, and from \eqref{B:constvisc}, 
\begin{align*}
\normL{\min(s_1,q)}{\k} \leq C \normL{q}{\f} + C \normL{t}{\v} + C \Big ( \normL{t}{\D \v} + \norml{t}{\hat p} \Big ) \normL{r}{\nabla c}.
\end{align*}
If $q \leq s_1$ then $\k \in \LP^q$ and the proof is complete, otherwise if $s_1 < q$, we deduce instead that $(\v, \hat p) \in \W^{2, s_1} \times W^{1,s_1} \subset  \W^{1,t_1} \times L^{t_1}$ where 
\begin{align*}
\frac{1}{t_1} := \frac{1}{s_1} - \frac{1}{d} = \frac{1}{s} + \frac{1}{r} - \frac{2}{d} = \frac{1}{t} + \frac{1}{r} - \frac{1}{d} \quad \text{ with } t_1 > t.
\end{align*}
Then, $\k \in \LP^{\min(s_2, q)}$ where the exponent $t_1 > s_2 > s_1$ is defined as
\begin{align*}
\frac{1}{s_2} := \frac{1}{r} + \frac{1}{t_1} = \frac{1}{s_1} + \frac{1}{r} - \frac{1}{d}.
\end{align*}
For the bootstrapping argument, given $s_{k-1}$ we set
\begin{align*}
\frac{1}{t_k} := \frac{1}{s_{k-1}} - \frac{1}{d} = \frac{1}{t_{k-1}} + \frac{1}{r} - \frac{1}{d}, \quad \frac{1}{s_k} := \frac{1}{r} + \frac{1}{t_k} = \frac{1}{s_{k-1}} + \frac{1}{r} - \frac{1}{d},
\end{align*}
then it holds that $s_k > s_{k-1}$ and $t_k > t_{k-1}$.  After a finite number of iterations, we get $\min(s_k,q) = q$ and the proof is complete.

\section*{Acknowledgments}
M.~Ebenbeck is supported by the RTG 2339 ``Interfaces, Complex Structures, and Singular Limits" of the German Science Foundation (DFG). The work of K.F.~Lam is partially supported by a grant from the Research Grants Council of the Hong Kong Special Administrative Region, China [Project No.: CUHK 14302319].

\footnotesize
\bibliographystyle{plain}

\end{document}